 \documentclass[12pt,reqno,psfig,openbib]{amsart}

\usepackage[pagebackref=true, colorlinks=true, citecolor=blue]{hyperref}

\usepackage{amssymb,amsmath,graphicx,amsfonts,euscript}
\usepackage{fancyhdr}
\usepackage{epsfig}
\usepackage{setspace}
\usepackage{dsfont}
\usepackage[all]{xy}
\usepackage{amsmath}
\usepackage[normalem]{ulem}        

\usepackage{mathtools}

\usepackage{bm}

\usepackage{color}

\usepackage{cleveref}

\newtheorem{theorem}{Theorem}[section]
\newtheorem{rem}[theorem]{Remark}

\newtheorem{lemma}[theorem]{Lemma}
\newtheorem{proposition}[theorem]{Proposition}
\newtheorem{cor}[theorem]{Corollary}
\newtheorem{notation}[theorem]{Notation}
\newtheorem{claim}[theorem]{Claim}

\numberwithin{equation}{section}

\date{\today}

\renewcommand{\epsilon}{\varepsilon}

\newcommand{\wi}{\widetilde}

\newcommand{\ov}{\overline}

\newcommand{\pa}{\partial}
\newcommand{\ep}{\epsilon}
\newcommand{\tht}{\theta}
\newcommand{\Tht}{\Theta}

\newcommand{\pri}{\prime}

\newcommand{\sig}{\sigma}
\newcommand{\kap}{\kappa}
\newcommand{\blds}{\boldsymbol}

\usepackage{verbatim}

\newcommand{\abs}[1]{\left\vert#1\right\vert}
\newcommand{\BB}[1]{\ensuremath{\mathbb{#1}}}

\newcommand{\R}{\ensuremath{\BB{R}}}

\newcommand{\iny}{\ensuremath{\infty}}
\newcommand{\grad}{\ensuremath{\nabla}}
\DeclareMathOperator{\dv}{div} %
\DeclareMathOperator{\curl}{curl} %
\DeclareMathOperator{\weak}{weak} %
\newcommand{\prt}{\ensuremath{\partial}}
\newcommand{\brac}[1]{\ensuremath{\left[ #1 \right]}}
\newcommand{\pr}[1]{\ensuremath{\left( #1 \right) }}

\newcommand{\norm}[1]{\ensuremath{\left\Vert #1 \right\Vert}}

\newcommand{\eps}{\ensuremath{\epsilon}}
\newcommand{\Cal}[1]{\ensuremath{\mathcal{#1}}}

\newcommand{\pdx}[2]{\frac{\prt #1}{\prt #2}}

\newcommand{\ol}{\overline}

\newcommand{\Ignore}[1] {}

\theoremstyle{definition}

\newcommand{\spacer}{\vspace{2mm}}
\newcommand{\halfspacer}{\vspace{1mm}}

\newcommand{\Holder}
    {H\"{o}lder }

\newcommand{\e}{\blds{e}}
\newcommand{\f}{\blds{f}}

\newcommand{\n}{\blds{n}}
\newcommand{\oo}{\blds{\omega}}
\newcommand{\uu}{\blds{u}}
\newcommand{\vv}{\blds{v}}

\newcommand{\BoldTau}{\boldsymbol{\tau}}

\DeclarePairedDelimiter{\set}{\{}{\}}

\DeclarePairedDelimiter{\bigppr}{\Big(}{\Big)}

\DeclareMathOperator{\erfc}{erfc} %

\newcommand{\Comment}[1]{{\color{blue}#1}}

\newcommand{\ToDo}[1]{\textbf{\Comment{[#1]}}}

\crefname{cor}{Corollary}{Corollaries} 

\crefname{lemma}{Lemma}{Lemmas}	       

\crefname{section}{Section}{Sections}
\Crefname{section}{Section}{Sections}

\crefname{theorem}{Theorem}{Theorems}
\Crefname{theorem}{Theorem}{Theorems}

\crefname{prop}{Proposition}{Propositions}
\Crefname{prop}{Proposition}{Propositions}

\crefformat{equation}{(#2#1#3)}
\crefrangeformat{equation}{(#3#1#4) through (#5#2#6)}
\crefmultiformat{equation}
    {(#2#1#3)}%
    { and~(#2#1#3)}
    {, (#2#1#3)}
    {, and~(#2#1#3)}

\renewcommand{\autoref}[1]{\textbf{USE cref RATHER THAN autoref!}}

\usepackage{color}
\usepackage{cases}

\newcommand{\red}{\textcolor{red}}
\definecolor{Green}{rgb}{0.010,0.7,0.02}

\newcommand{\RR}{\mathbb{R}}

\newcommand{\ZZ}{\mathbb{Z}}


\begin{document}

\title[Vanishing viscosity limit]
{The Vanishing viscosity limit for some symmetric flows}
\author[G.-M. Gie et al.]
{Gung-Min Gie$^1$, James P. Kelliher$^2$, Milton C. Lopes Filho$^3$, Anna L. Mazzucato$^4$, and Helena J. Nussenzveig Lopes$^3$}
\address{$^1$ Department of Mathematics, University of Louisville, 328 Natural Sciences Building, Louisville, KY 40292, U.S.A.}
\address{$^2$ Department of Mathematics, University of California, Riverside, 900 University Ave., Riverside, CA 92521, U.S.A.}
\address{$^3$ Instituto de Matematica, Universidade Federal do Rio de Janeiro, Caixa Postal 68530, 21941-909, Rio de Janeiro, RJ, Brazil}
\address{$^4$ Department of Mathematics, Penn State University, University Park, PA 16802, U.S.A.}

\email{gungmin.gie@louisville.edu}
\email{kelliher@math.ucr.edu}
\email{mlopes@im.ufrj.br}
\email{alm24@psu.edu}
\email{hlopes@im.ufrj.br}

\begin{abstract}
The focus of this paper is
on the analysis of the boundary layer and the associated
vanishing viscosity limit for two classes of flows with symmetry,
namely, Plane-Parallel Channel Flows and Parallel Pipe Flows.
We construct   explicit boundary layer correctors,
which approximate the difference between the Navier-Stokes and the Euler solutions. Using properties of these correctors, we establish convergence of the Navier-Stokes solution to the Euler solution as viscosity vanishes with optimal rates of convergence.  In addition,
we investigate vorticity production on the boundary in the limit of vanishing viscosity.
Our work significantly extends prior work in the literature.
\end{abstract}

\maketitle


\section{Introduction}\label{S:intro}

This article concerns the behavior of incompressible, viscous
fluids at small viscosity in the presence of boundaries under the
classical ``no-slip'' boundary conditions.
We let $\Omega$ be a bounded domain in two or three space dimensions  with
 boundary $\Gamma$ of class $C^\infty$. Viscous, incompressible (Newtonian)
fluid flow is
modeled by solutions of the Navier-Stokes equations (NSE for short). We
consider the following initial-value problem:
\begin{equation}\label{e:NSE}
        \left\{\begin{array}{rll}
                             \dfrac{\pa \uu^\eps}{\pa t} + (\uu^\eps \cdot \nabla)\uu^\eps
                             & \hspace{-2mm}
                             = - \nabla p^\ep + \ep \Delta \uu^\eps + \f, \, & \text{ in } \Omega\times(0,T),\\
                \halfspacer
                              \dv \uu^\eps
                              & \hspace{-2mm}
                              = 0, \, & \text{ in } \Omega\times(0,T),\\
                              \uu^\eps
                              & \hspace{-2mm}
                              = 0, \, & \text{ on } \Gamma \times (0, T),\\
                              \uu^\eps\big|_{t=0}
                              & \hspace{-2mm}
                              = \uu_0, \, & \text{ in } \Omega.\\
        \end{array}\right.
\end{equation}

Where $\uu^\eps$ is the Eulerian fluid velocity, $p^\ep$ is the pressure, $\f$
are given external forces, and $\uu_0$ is the given initial velocity.
Here $\ep$ is a small, strictly positive  parameter, representing
the kinematic viscosity of the fluid, assumed homogeneous, $T>0$ is a fixed,
positive time, $\f$ and $\uu_0$ are smooth, divergence-free vector fields.
The boundary condition in \cref{e:NSE} is referred to as the \textit{no-slip} condition or
\textit{no-slip, no-penetration} condition.

By formally setting $\ep =0$ in NSE we obtain the Euler equations (EE for
short), which model the flow of inviscid, incompressible fluids. The
initial-value problem for EE is given by:
\begin{equation}\label{e:EE}
        \left\{ \begin{array}{rll}

                   \dfrac{\pa \uu^0}{\pa t} + (\uu^0 \cdot \nabla)\uu^0
                             & \hspace{-2mm}
                             = - \nabla p^0 + \f, \, & \text{ in } \Omega\times(0,T),\\
                \halfspacer
                              \dv \uu^0
                              & \hspace{-2mm}
                              = 0, \, & \text{ in } \Omega\times(0,T),\\
                \halfspacer
                              \uu^0 \cdot \n
                              & \hspace{-2mm}
                              = 0, \, & \text{ on } \Gamma \times (0, T),\\
                              \uu^0\big|_{t=0}
                              & \hspace{-2mm}
                              = \uu_0, \, & \text{ in } \Omega.\\
        \end{array}\right.
\end{equation}
where $\n$ is the unit outer normal vector on $\pa\Omega$. The boundary
condition in \cref{e:EE} is referred to simply as \textit{no-penetration}, and
reflects the assumption that the fluid is in a container with rigid walls.
For the types of flows considered in this paper, it is convenient to take the
initial velocity for NSE to be independent of $\ep$ and equal to the initial
velocity for EE, although this assumption can be weakened.
The assumption that the data and the boundary of the domain are smooth can also
be weakened, but we will not seek optimal regularity conditions,
as our focus is on a detailed analysis of the fluid boundary behavior at small
viscosity. By passing to
a moving frame, it is possible to consider the case in which the boundary is
allowed to move rigidly along itself, as in the classical case of
the Taylor-Couette flow. Then, the no-slip boundary condition reads $\,
\uu^\ep\equiv \bm{U}$ on $\Gamma\times (0,T)$, where $\bm{U}$ is a given vector
field tangent to the boundary.

A main question in fluid mechanics is whether viscous fluids at low viscosity
are well approximated by inviscid fluids. Near the boundary, this approximation
cannot hold uniformly in $\ep$ as there must be  a discrepancy in the tangential
components  of $\uu^\ep$ and $\uu^0$ at the boundary, unless $\uu^0$ happens to
vanish on the boundary identically over time. This discrepancy leads to the
potential creation of large gradients of velocity in a layer near the boundary,
called {\em a viscous boundary layer}, where the fluid is hence neither well
modeled by solutions of NSE nor by solutions of EE. (We refer to
\cite{Schlichting} and references therein for an introduction to the theory of
boundary layers.) Understanding the behavior of a fluid in the viscous layer is
one of the most challenging problems in fluid mechanics, and mathematically it
is far from understood, even though progress has been made recently.
A related mathematical problem is whether solutions of NSE converge
in a suitable norm to solution of EE as $\ep$ goes to zero.
We will say that
the (classical) \textit{vanishing viscosity limit} or \textit{inviscid limit}
holds if solutions of \cref{e:NSE} converge to solutions of \cref{e:EE} in the
energy norm, that is, strongly in $L^\infty((0,T);L^2(\Omega))$.
Whether the classical vanishing viscosity limit holds generically, at least for
short time,
is an open question even for $C^\iny$ initial data and in simple
geometries, such as a disk in the plane. Except in special situations, one does
not expect the
vanishing viscosity limit to hold over long intervals of time (assuming the
Euler solution exists over such intervals) because of the observed phenomenon
of {\em boundary layer separation}. However, the precise relation between layer
separation and the vanishing viscosity limit or lack thereof has not been
established yet.

There is an extensive literature on the vanishing viscosity limit when the
boundary layer is absent or very weak.
For solutions in the whole space or in a
periodic domain, the vanishing viscosity limit has been rigorously proved
in various norms
(\cite{Swann1971, Kato1972, Kato1975, ConstantinFoiasBook, Masmoudi2007}).
The limit also holds if some slip is allowed at the boundary
for viscous flows or if the production of vorticity at the boundary is
prescribed, such as under so-called Navier-friction boundary
conditions \cite{daViegaCrispo2010, daViegaCrispo2011A,XiaoXin2007,
BelloutNeustupa1, BelloutNeustupa2, BelloutNeustupa3, BelloutNeustupa4,
JPT11,GJ13}. In this context, the vanishing viscosity limit
has also been used initially as a mean to establish existence of 2D Euler
solutions (see  (\cite{Y1963},
\cite[pp.~87--98]{JL1969}, \cite{Bardos1972}, and
\cite[pp.~129--131]{L1996}). The boundary layer is studied for Navier conditions
in 2D in \cite{CMR, FLP, KNavier} and in 3D in \cite{IP06, IS10,
MasmoudiRousset2010, GK2012}. Lastly, the limit can be shown to hold for
non-characteristic boundary conditions \cite{TWsuction,TWnonchar, HT07, GHT12}, such as with
injection and suction at the boundary.

For the classical no-slip boundary conditions considered here, a formal
asymptotic analysis as $\ep\to 0$ leads to the Prandtl equations for the
velocity in the boundary layer, which exhibit both ill-posedness and
instabilities \cite{EE97,Grenier00,GSS09,GVD10,GN11,G-VN12}, unless the
boundary and the data have some degree of
analyticity \cite{Asano88,CS98II,LCS03,Maekawa2012,CLM13,KV13} or the data is
monotonic in the normal direction to the boundary
\cite{Oleinik66,OleinikSamokhin99,KMVW14}.
Another situation in which the Prandtl equations are well behaved and the
boundary layer can be analyzed is when the
initial data and the geometry of the domain have special symmetries. In this
paper, we discussed several examples of this last situation.

Specifically, we investigate
{\em plane-parallel channel} and  {\em parallel pipe flows} in three space
dimensions.
These are well-known examples of exact solutions of the fluid equations that
can be viewed as generalizations of plane Couette and Poiseuille flows, and have
been investigated before in the context of  boundary layers and the vanishing
viscosity limit. A special case of parallel pipe flows is that of planar flows,
which reduce to two-dimensional, {\em circularly-symmetric flows}.
These flows are naturally of interest for the study of boundary layers, as the
inviscid limit holds because  Kato's criterion \cite{Kato1983}, and
specifically, the generalization due to Temam and Wang \cite{TW1998,W2001}, applies.
In fact, they represent interesting, physically motivated, test cases, since the
Prandtl approximation can be rigorously established. In addition, an analysis of the
vorticity production by the boundary, in the vanishing viscosity limit, can be carried out.

In this article, we extend significantly prior work on these classes of flows,
some of which was done by the same authors of the present manuscript, giving  a
unified treatment of different classes of flows, focusing in particular on
vorticity production at the boundary and ill-prepared, or non-compatible, data.
By {\em ill-prepared} initial velocity we mean that the tangential component of
$\uu^0$ does not vanish at the boundary, so that the no-slip boundary condition
in \cref{e:NSE} is not satisfied at time $t=0$, and the forcing need not be
compatible with the initial data at $t=0$.  The smooth initial
data is assumed to be only in the space
\begin{align*}
    H
        = \set{\vv \in L^2(\Omega) |
            \, \dv \vv = 0, \, \vv \cdot \n = 0 \text{ on } \Gamma},
\end{align*}
but not in the space
\begin{align*}
    V = \set{\vv \in H^1_0 (\Omega) | \, \dv \vv = 0}.
\end{align*}
The case of ill-prepared data is mathematically more difficult to treat and
physically more interesting when the inviscid solution is steady, which is the
case for  circularly-symmetric data, as recalled below. In this case, there is
constant production of vorticity at the boundary in the limit \cite{LMNT08}.
Production of vorticity at the boundary was already discussed for plane Couette
and Poiseuille flows in \cite{Morton84}, using physical arguments.

In \cref{S:examples}, we introduce the special symmetric flows we will be concerned with, and we make some general
remarks about the vanishing viscosity limit.
The simplest case of symmetric flows is that of circularly symmetric flows,
which are  2D solutions of the fluid equations  for which the
streamlines are circles centered at
the origin. Such solutions can be obtained from any radial stream function or,
equivalently, any radial vorticity function, via the Biot-Savart law. They are a
special case of parallel pipe flows, discussed in \cref{S:IPF}.  It is immediate
to verify that any circularly symmetric, sufficiently regular Euler flow is
steady, that is, $\uu^0(t)=\uu^0(0)= \uu_0$, and that the solution to NSE with
the same data actually solves a two-dimensional heat equation with no pressure.
Since the dynamics is completely linear, this example is more pedagogical in
nature. It arises also in the context of stability of boundary layers around
steady profiles, a challenging and fundamental problem, which we do not tackle
in this paper (but see recent results in
\cite{BM14,GGNpreprint,BJMpreprint1,BJMpreprint2}).
A first proof of the vanishing viscosity limit in this class can be found in
\cite{Matsui1994} (see also \cite{BW02}). A more general convergence result,
allowing for a rough boundary velocity $\bm{U}$, which precludes  the use
of Kato's criterion, appears in \cite{LMNT08}. A simple argument to show that
the vanishing viscosity limit holds is given in \cite[Theorem 6.1]{K2006Disk},
though without a rate of convergence.

In \cref{S:Lighthill} we discuss the Lighthill principle for viscous flows between two
parallel planes and we use it to deduce an $L^1$ estimate for the vorticity of plane
parallel channel flows, uniform with respect to viscosity. We focus on the argument
which leads to the Lighthill principle and on the role of the Lighthill principle
in quantifying vorticity production at the boundary. This section can be read
independently from the remainder of our work.

In \cref{S:PPF}, we discuss plane-parallel flows in a periodized
channel. These are flows for which the streamlines lies on
parallel planes, and the velocity is independent of one of the horizontal
variables, but depends on the vertical variable, making the flow three
dimensional.
For plane-parallel flows, the Euler solution $\uu^0$ will not be steady,
even for zero forcing, and both EE and NSE retain their non-linear nature,
albeit only as a weakly non-linear system with zero pressure, making this a
substantially more difficult problem to study. A proof of the validity of the
vanishing viscosity
limit for ill-prepared data and the analysis of the boundary layer corrector
were carried out in \cite{MT08}, using a parametrix construction for a
diffusion-drift equation and layer potential techniques. Convergence of the
corrected velocity was obtained only in $L^\infty$. A Prandtl-type expansion
was used in \cite{MNW10} to obtain convergence in $H^1$ uniformly in time, but only for well-prepared data. In this article, we extend these results to obtain strong convergence of the corrected velocity in
$L^\infty((0,T);H^1(\Omega))$ for ill-prepared data and study vorticity production at the boundary in the limit.

Parallel pipe flows, the subject of \cref{S:IPF}, combine the features
of both circularly symmetric flows and plane-parallel flows. The domain is
a straight, infinite, circular pipe that is periodized along the direction of
the axis. As with the channel geometry, symmetry and periodicity ensure
uniqueness of solutions to NSE and EE, excluding in particular non-trivial
pressure-driven flows.
The velocity is independent of the variable along the pipe axis and, in  any
circular cross section of the
pipe, it is the sum of a circularly symmetric, planar velocity field and a
velocity field pointing in the direction of the axis. Again, NSE and EE reduce
to a weakly non-linear system.
A substantial complication over plane-parallel flows is that the
non-vanishing curvature now becomes an important factor in the analysis. Also, much as in the case of axisymmetric flows in the whole space, the behavior of the solution near the axis cannot be controlled as well as it can be away from
the axis in cylindrical coordinates. To deal with this difficulty, one can adapt
techniques from \cite{MT11} and \cite{HMNW11}, which entails the use of a
two-step localization, close to the boundary and near the pipe axis, or employ
suitable weighted inequalities. Since our focus in this work is on
the behavior of the
flow near the boundary, we restricts ourselves to considering pipes with
annular cross-section.

We close this Introduction with some notational conventions.
\begin{notation}\label{n:kappa}
\textnormal{ We introduce generic constants,
\begin{equation*}\label{e:kappa}
        \kap := \kap(\uu_0, \, \f, \, \Omega),
            \qquad
        \kap_T := \kap_T(\uu_0, \, \f, \, \Omega, \, T),
\end{equation*}
depending on the indicated data, but independent of $\ep$ or $t$.
}
\end{notation}

\begin{notation}\label{n:est}
\textnormal{
        By the appellative $e.s.t.$ associated to a function t we mean
that the function or constant has exponentially small norm in  all Sobolev
spaces $H^s$ (and thus in all H\"older's spaces $C^s$)
with a bound on the norm of the form $c_{1,s} \, e^{-c_{2,s}/\ep^{\gamma_S}}$,
$c_{1,s}, \,
c_{2,s}, \, \gamma_s > 0$, for each $s$.
We will say that a constant is $e.s.t$
if it satisfies a similar bound.
}
\end{notation}

%
%
\section{Symmetric flows: an overview} \label{S:examples}

The focus of this work is the analysis of the boundary layer and vanishing
viscosity limit for two classes of flows with symmetry.

Below and throughout the paper, we employ the following standard notation: if
$(\zeta,\eta,\xi)$ represents an orthogonal system of coordinates in $\RR^3$,
then $\{ \e_\zeta, \e_\eta, \e_\xi\}$ represents the associated orthonormal
frame, and similarly for coordinates in the plane. We will denote by
$(x,y,z)$ the Cartesian coordinates in $\RR^3$, by $(r,\phi)$ the
polar coordinates in $\RR^2$, and by $(x,r,\phi)$ the cylindrical coordinates in
$\RR^3$.

In this work we will be concerned with the following symmetric flows.

\begin{itemize}
    \item[\textbf{ (CSF)}] {\bf Circularly symmetric flows:} these are planar
flows in a
    disk centered at the origin \, $\Omega=\{ x^2+y^2 < R^2\}$. The velocity is
of the form:
$$
   \uu = V(r,t) \e_\phi.
$$
  The vorticity, which can be identified with a
scalar for planar flows, is also radially symmetric.

    \medskip

    \item[\textbf{  (PCF)}] {\bf Plane-parallel channel flows:} these are 3D
flows in an
    infinite channel, with periodicity imposed  in the $x$ and
    $y$-directions. The velocity takes the form:
        \begin{align*}
            \uu = (u_1 (z, t), \, u_2 (x, z, t), \, 0),
        \end{align*}
      and is defined on the domain
        \begin{align*}
            \Omega := (0, L)^2 \times (0, h).
        \end{align*}
  Here  $h$ is the width of the channel and $u_1$, $u_2$ satisfy periodic
       boundary conditions in  $x$ and $y$ with period $L$. The boundary is identified
with the set $\Gamma := \prt \Omega = [0,L]^2 \times \set{0, h}$

    \medskip

    \item[\textbf{ (PPF)}] {\bf Parallel pipe flows:}  these are 3D flows in an
      infinite straight, circular pipe, with periodicity imposed along the pipe
axis.
     The velocity is of the form
        \begin{align*}
            \uu=
                u_{\phi} (r, t) \e_{\phi}
                    + u_x (\phi, r, t) \e_x,
        \end{align*}
      in cylindrical coordinates on the domain
        \begin{align*}
            \Omega := \{(x, y, z) \in \R^3 \, \mid \, y^2 + z^2 < R^2, \; 0<x<L\}.
        \end{align*}
      Here $R$ is the radius of the circular cross-section of the pipe and $u_\phi$, $u_x$
    satisfy periodic boundary conditions in $x$.
   The boundary is identified with the set $\Gamma = [0, L] \times \{(y,z)\in
\RR^2 \,\mid \, y^2+z^2= R^2\}$.
\end{itemize}

CSF is a special case of PPF when the component of the velocity along the axis
is zero, that is, the flow can be identified with a
two-dimensional flow. In fact, the cross-sectional components of any PPF
can be identified with a CSF in the cross-section of the pipe.
In all three cases, the symmetry of the initial data is preserved in time for
both $\uu^\eps$ and $\uu^0$ as long as the forcing has the same spatial symmetry
as the initial velocity. Uniqueness holds not only in the class of strong
solutions, but also in the class of weak solutions (see \cite{BLNNT13} and
references therein).

For any initial velocity $\uu_0 \in H$, due to the energy inequality for
solutions of  NSE, weak sequential compactness  implies the existence of $\vv
\in L^\iny(0, T; H)$ and some subsequence of $(\uu^\eps)_{\ep>0}$ converging
weakly to $\vv$ in $L^\iny(0, T; H)$  (see \cite{Kel14Observations}).
Additional information is required to conclude that $\vv$ is a weak solution of
EE.

Let $\omega^{\eps} = \curl \uu^{\eps}$ be the vorticity. We will find that, in each of these examples,
\begin{align}\label{e:graduBounded}
    (\omega^\eps) \text{ is bounded in }
    L^\iny(0, T; L^1(\Omega))
    \text{ uniformly in } \eps.
\end{align}
 Except in the very special case when $\uu^0$ vanishes
on the boundary for all $t \in (0, T)$, it is not possible to have
$(\omega^\eps)$ uniformly bounded in $L^\iny(0, T; L^p(\Omega))$ for any $p > 1$
(see \cite{Kel14Observations}).
Hence,
\cref{e:graduBounded} is the strongest possible condition (in the class of
Lebesgue spaces) one could expect on $(\omega^\eps)$.

However, not even (\ref{e:graduBounded}) is enough to ensure that the classical vanishing viscosity limit,
\begin{align}\label{e:VVLimit}
    \uu^\eps \to \uu^0
        \text{ in } L^\iny(0, T; H),
\end{align}
holds true. In fact, a slightly stronger condition would be sufficient, namely that $\{\omega^{\eps}\}$ be bounded in $L^{\infty}(0,T;X)$, for some Banach space $X$ which is compactly imbedded in $H^{-1}$. This follows from an easy adaptation of Theorem 1.1 in \cite{LLT00}. Within the Lebesgue hierarchy, $L^1$ is critical for this imbedding. In fact, $L^p$ is compactly imbedded in $H^{-1}$, for any $p>1$.


It should be noted that, even for CSF, \cref{e:graduBounded} is not
straightforward to establish (in fact, the lack of an $L^p$ vorticity bound, for $p>1$, uniform in viscosity,
is a diffusive effect, present even when inertial terms
vanish). Here, it is a byproduct of establishing convergence in stronger norms
than the energy norm for
the corrected velocity. (For recent, related criteria on the validity of the
vanishing viscosity limit, see \cite{CEIVPreprint}.)


Another common property among all types of flows under study is the fact that
$\Delta \uu^\eps \cdot \n = 0$ on the boundary. Though we will not use this
property directly, we will exploit some related implications, in particular that
the Laplace and Stokes operators agree when applied to $\uu^\eps$ and that
$\grad p_\eps \cdot \n = 0$, providing a boundary condition for the pressure.
(The pressure will vanish entirely for CSF and PCF.)
By comparison, in \cite{BelloutNeustupa4}  (see also  \cite{BelloutNeustupa3})
the authors study NSE under boundary conditions of the form $\curl^k \uu^\eps
\cdot \n = 0$, $k = 0, 1, 2$, which can also be written as
\begin{align*}
    \uu^\eps \cdot \n
        = \curl \uu^\eps \cdot \n
        = \Delta \uu^\eps \cdot \n = 0.
\end{align*}
For the 3D examples of  PCF and PPF, the first and third of these boundary
conditions are satisfied. The second boundary condition is not satisfied, though
$\curl \uu^\eps \cdot \n$ is of an especially simple form, containing only a
tangential derivative of one of the components of the velocity.

\section{Lighthill principle for PCF} \label{S:Lighthill}

In this section we introduce the Lighthill principle, which we prove for the case of flow between two parallel planes, and we use it to derive an $L^1$ estimate on vorticity for PCF, independent of viscosity. The Lighthill principle is a property of viscous incompressible flow in a domain with a rigid boundary. Roughly speaking it is a way of expressing the flux, through the rigid boundary, of the vorticity components tangent to the boundary, in terms of tangential derivatives of pressure at the boundary. We will see that, for flow between two parallel planes, the vorticity vector is actually tangent to the boundary, so that Lighthill principle provides a complete set of boundary conditions for the viscous vorticity equation, provided that the pressure is known at the boundary. For a discussion of the Lighthill principle see \cite{Morton84}, and for the original source, see \cite{Lighthill63}.

The results covered in this section are not used in the remainder of the text. Our purpose in including this material is twofold. First, to present the argument which leads to the Lighthill principle. This is what will actually be used in the remainder of the article. Our second objective is to illustrate the use of Lighthill's idea in estimating vorticity production by the viscous friction between the fluid and the boundary in a rigorous form.

We are interested in solutions of the 3D Navier-Stokes equations \eqref{e:NSE} between two parallel planes, say $\{z=0\}$ and $\{z=h\}$. We will also assume the flow is periodic in the other two directions. Let $L>0$ and set $Q_L = [0,L] \times [0,L]$ to be the periodic box of sides $L$; in this section $\Omega = Q_L \times (0,h)$. In this section we will assume that $\f \equiv 0$.

We fix $\uu_0$ a smooth, divergence-free vector field in $\Omega$, tangent to $\Gamma=\partial \Omega \equiv Q_L \times \{0,h\}$, horizontally periodic, and consider $\uu^{\eps}$ the (smooth) solution of problem \eqref{e:NSE} with $\f = 0$.

\begin{proposition} \label{LHandtangPCF}
Let $\oo^{\eps} = \curl \uu^{\eps} = (\omega_1^{\eps},\omega_2^{\eps},\omega_3^{\eps})$. We have that, at any point on $\Gamma=\partial \Omega$, \[\omega_3^{\eps} = 0.\] We also have, at the boundary, that:
\begin{equation} \label{LF}
\begin{array}{l}
\displaystyle{\frac{\partial\omega_1^{\eps}}{\partial z} = - \frac{1}{\eps}\frac{\partial p^{\eps}}{\partial y}} \\
\\
\displaystyle{\frac{\partial\omega_2^{\eps}}{\partial z} = \frac{1}{\eps}\frac{\partial p^{\eps}}{\partial x}}. \\
\end{array}
\end{equation}
\end{proposition}

\begin{proof}
Set $\n_{\pm} = (0,0,\pm 1)$, so that $\n_+$ is the unit exterior normal to $Q_L \times \{z=h\}$  and $\n_-$ is the unit exterior normal to $Q_L \times \{z=0\}$. We claim that
\begin{equation} \label{vorttangtobdry}
\pm \omega_3^\eps \equiv \oo^{\eps}\cdot\n_{\pm} = 0 \mbox{ on } \partial \Omega \times (0,T).
\end{equation}

Indeed, it is immediate that $\pm \omega_3^\eps \equiv \oo^{\eps}\cdot\n_{\pm} $.
We write
\[\oo^{\eps}=(\partial_y u_3^\eps - \partial_z u_2^\eps, \, \partial_z u_1^\eps-\partial_x u_3^\eps, \, \partial_x u_2^\eps-\partial_y u_1^\eps).\]

Hence, at $\Gamma$, $\omega_3^\eps = \partial_x u_2^\eps-\partial_y u_1^\eps = 0$, because $\uu^\eps = 0$ at $\Gamma$ and both $\partial_x$ and $\partial_y$ are tangential  derivatives along the boundary. This establishes \eqref{vorttangtobdry}.

Next, we observe that, from the vector calculus identity below:
\[ \curl \curl \uu = \nabla \dv \uu - \Delta \uu\]
together with the fact that $\uu^{\eps}$ is divergence-free, it follows that
\begin{equation} \label{curlvort}
\Delta \uu^{\eps} = - \curl \oo^{\eps}.
\end{equation}

Assume that the Navier-Stokes equations \eqref{e:NSE} remain valid up to the boundary. Then, since $\uu^{\eps} = 0$ on $\partial \Omega \times (0,T)$, we find, using \eqref{curlvort},
\begin{equation} \label{gradpcurlvort}
\curl \oo^{\eps} = -\displaystyle{\frac{1}{\eps}}\nabla p^{\eps} \mbox{ on } \partial \Omega \times (0,T).
\end{equation}

We will take the cross-product of \eqref{gradpcurlvort} with $\n_{\pm}$.

We first compute $\curl\oo^\eps \times \n_{\pm}$ and we find:
\[\curl \oo^\eps \times \n_{\pm} = \pm(-\partial_x\omega_3^\eps + \partial_z\omega_1^\eps  , -\partial_y\omega_3^\eps + \partial_z\omega_2^\eps , 0).\]
However, on $\Gamma$ we now know that $\omega_3^\eps = 0$. Hence, since $\partial_x$ and $\partial_y$ are tangential derivatives, we find, on $\Gamma$, that
\begin{equation} \label{Doon}
\curl \oo^\eps \times \n_{\pm} = \pm(\partial_z\omega_1^\eps,\,\partial_z\omega_2^\eps,\,0).
\end{equation}

Next we compute $\nabla p^\eps \times \n_{\pm}$. We obtain:
\begin{equation} \label{nablapn}
\nabla p^\eps \times \n_\pm = \pm(\partial_y p^\eps, - \partial_x p^\eps, 0).
\end{equation}

We easily deduce, from \eqref{gradpcurlvort}, \eqref{Doon} and \eqref{nablapn}, the desired system of equations in the statement, \eqref{LF}.

\end{proof}

Lighthill principle, as expressed above, provides a complete set of boundary conditions for the vorticity form of the Navier-Stokes equations. The two tangential components of vorticity satisfy a non-homogeneous Neumann condition  and the normal component satisfies a homogeneous Dirichlet condition.

Next we will focus on the special case of plane-parallel channel flows in $\Omega$ (PCF). As discussed in the previous section, PCF have the form
\begin{equation} \label{PCF}
\uu^\eps = \uu^\eps (x,y,z,t) \equiv (u_1^\eps (z,t), \, u_2^\eps(x,z,t),\, 0).
\end{equation}
This symmetry is preserved by both the Euler and Navier-Stokes evolution. Note that the divergence-free condition for velocity is automatically satisfied.

We will use the following notation for the initial velocity:
\[\uu_0=\uu_0(x,y,z)=(g_1(z),\,g_2(x,z),\,0).\]

Under this symmetry the Navier-Stokes equations reduce to:
\begin{equation}\label{NSEsymreduced}
        \left\{\begin{array}{rll}
                             \dfrac{\pa u_1^\eps}{\pa t}
                             & \hspace{-2mm}
                             = - \dfrac{\pa p^\ep}{\pa x} + \ep \dfrac{\pa^2 u_1^\eps}{\pa z^2}, \, & \text{ in } \Omega\times(0,T),\\
             \\
              \dfrac{\pa u_2^\eps}{\pa t} + u_1^\eps \dfrac{\pa u_2^\eps}{\pa x}
                             & \hspace{-2mm}
                             = - \dfrac{\pa p^\ep}{\pa y} + \ep \Delta_{x,z} u_2^\eps, \, & \text{ in } \Omega\times(0,T),\\
                              \\
              0 & \hspace{-2mm} = - \dfrac{\pa p^\ep}{\pa z}, \, & \text{ in } \Omega\times(0,T),\\
                              \\
              u_1^\eps = u_2^\eps
                              & \hspace{-2mm}
                              = 0, \, & \text{ on } Q_L\times\{0,h\} \times (0, T),\\
              u_1^\eps,\; u_2^\eps
                              & \hspace{-2mm}
                              L-\text{periodic}\, & \text{ in } x,\,y, \, \text{ for each }z \in (0,h), t \in (0,T),\\
                              \\
                              u_1^\eps\big|_{t=0}
                              & \hspace{-2mm}
                              = g_1 (z), \, & \text{ in } \Omega, \\ \\
                              u_2^\eps\big|_{t=0}
                              & \hspace{-2mm}
                              = g_2(x,z), \, & \text{ in } \Omega.\\
        \end{array}\right.
\end{equation}
Above, the pressure $p^\eps$ may be chosen to vanish identically. Indeed, we deduce, from the evolution equations for $u_1^\eps $ and $u_2^\eps $, that $\pa^2 p^{\eps} /\pa x^2 = \pa^2 p^{\eps} /\pa y^2 = 0$. As we are assuming periodic boundary conditions on all the unknowns we find that $p^\eps = p^{\eps}(z,t)$ and, since $\pa p^\eps/\pa z = 0$, $p^\eps$ is constant in $z$; we choose $p^\eps = 0$.

In what follows we are interested primarily in the behavior of vorticity when the data is not compatible, i.e., when $g_1$ and $g_2$ are not necessarily vanishing at $z=0$ and $z=h$. The compatible case is much simpler to treat. We will begin our analysis with the observation that, taking the $\curl$ of the velocity equation \eqref{NSEsymreduced}  and using Theorem \ref{LHandtangPCF} and $p^\eps = 0$, we obtain the system of equations below, for $\oo^{\eps}=(\omega_1^\eps,\omega_2^\eps,\omega_3^\eps)\equiv \displaystyle{\left( -\dfrac{\pa u_2^\ep}{\pa z}, \dfrac{\pa u_1^\ep}{\pa z}, \dfrac{\pa u_2^\ep}{\pa x} \right)}$.
\begin{equation} \label{vortLH}
 \left\{\begin{array}{rll}
                             \dfrac{\pa \omega_1^\eps}{\pa t} + u_1^\eps \dfrac{\pa \omega_1^\ep}{\pa x} - \omega_2^\eps \omega_3^\eps
                             & \hspace{-2mm}
                             = \ep \Delta_{x,z}\omega_1^\eps, \, & \text{ in } \Omega\times(0,T),\\
             \\
              \dfrac{\pa \omega_2^\eps}{\pa t}
                             & \hspace{-2mm}
                             = \ep \dfrac{\pa^2 \omega_2^\eps}{\pa z^2}, \, & \text{ in } \Omega\times(0,T),\\
                              \\
              \dfrac{\pa \omega_3^\eps}{\pa t}+ u_1^\eps \dfrac{\pa \omega_3^\ep}{\pa x}  & \hspace{-2mm} =

              \ep \Delta_{x,z}\omega_3^\eps, \, & \text{ in } \Omega\times(0,T),\\
                              \\
              \dfrac{\pa \omega_1^\eps}{\pa z} = \dfrac{\pa \omega_2^\eps}{\pa z} = \omega_3^\eps
                              & \hspace{-2mm}
                              = 0, \, & \text{ on } \partial \Omega \times (0, T),\\
              \oo^\eps \;\;\; L-\text{periodic}, & \hspace{-2mm} \, & \text{ in } x,\,y, \, \text{ for each }z \in (0,h), t \in (0,T),\\
                              \\
                              \omega_1^\eps\big|_{t=0}
                              & \hspace{-2mm}
                              = -\dfrac{\pa g_2}{\pa z}, \, & \text{ in } \Omega, \\ \\
                              \omega_2^\eps\big|_{t=0}
                              & \hspace{-2mm}
                              =\dfrac{d g_1}{dz}, \, & \text{ in } \Omega.\\ \\
                              \omega_3^\eps\big|_{t=0}
                              & \hspace{-2mm}
                              = \dfrac{\pa g_2}{\pa x}, \, & \text{ in } \Omega.\\
        \end{array}\right.
\end{equation}

If the initial data $g_i$, $i=1,\,2$, were compatible then even the spatial derivatives of the solution $u_1^\eps$, $u_2^\eps$ to \eqref{NSEsymreduced} would be continuous in time, see  \cite{Evans} Chapter 7, \S 7.1, Theorem 5, allowing for energy methods to produce bounds on vorticity. As our main interest is non-compatible data, we will use a different approach.

We will obtain bounds for the vorticity $\oo^\ep$ in $L^\infty((0,T);L^1(\Omega))$, uniform with respect to $\ep$, by approximating the non-compatible problem for \eqref{NSEsymreduced} by a sequence of compatible problems. We will argue that the sequence of velocities converge, in the sense of distributions, to the solution of the non-compatible problem and we will derive estimates for the curl of the approximate velocities, uniform along the sequence, in $L^\infty((0,T);L^1(\Omega))$. It follows by the weak lower semicontinuity of the $L^1$-norm that these estimates remain true for the limit problem.

As stated in the beginning of this section, this material is independent from the remainder of the article and serves mostly a pedagogical purpose. Hence, in the theorem below, we choose to be rather loose regarding the precise regularity of the solutions involved. We point out that solutions of the heat equation are certainly as smooth as needed in the calculations performed in the proof.

\begin{theorem} \label{vortestimates}
Fix $h, L >0$ and $T>0$. Let $Q_L=[0,L]^2$ be the periodic box of sides $L$ and set $\Omega = Q_L \times (0,h)$. Let $g_1=g_1(z) \in C^{\infty}([0,h])$, $g_2=g_2(x,z) \in C^\infty ([0,L]\times[0,h])$, and suppose $g_2(0,z)=g_2(L,z)$. Assume that neither $g_1$ nor $g_2$ vanish for $z \in \{0,h\}$. Consider plane-parallel channel flow $\uu^\ep=\uu^{\ep}(x,\,y,\,z,\,t)\equiv (u_1^\ep(z,t),\,u_2^\ep(x,z,t),\,0)$. Then $u_1^\ep$, $u_2^\ep$ is the solution of
\begin{equation}\label{NSEsymreducedwopressure}
        \left\{\begin{array}{rll}
                             \dfrac{\pa u_1^\eps}{\pa t}
                             & \hspace{-2mm}
                             = \ep \dfrac{\pa^2 u_1^\eps}{\pa z^2}, \, & \text{ in } (0,h)\times(0,T),\\

             u_1^\eps
                             & \hspace{-2mm}
                              = 0, \, & \text{ at } \{z=0,h\} \times (0, T),\\

             u_1^\eps\big|_{t=0}
                              & \hspace{-2mm}
                              = g_1 (z), \, & \text{ in } (0,h), \\
                              \\
             \dfrac{\pa u_2^\eps}{\pa t} + u_1^\eps \dfrac{\pa u_2^\eps}{\pa x}
                             & \hspace{-2mm}
                             = \ep \Delta_{x,z} u_2^\eps, \, & \text{ in } [0,L]\times (0,h) \times(0,T),\\

             u_2^\eps
                              & \hspace{-2mm}
                              = 0, \, & \text{ on } [0,L]\times \{z=0,h\} \times (0, T),\\
             u_2^\eps
                              & \hspace{-2mm}
                              L-\text{periodic}\, & \text{ in } x, \, \text{ for each }z \in (0,h), t \in (0,T),\\

             u_2^\eps\big|_{t=0}
                              & \hspace{-2mm}
                              = g_2(x,z), \, & \text{ in } [0,L] \times (0,h).\\
        \end{array}\right.
\end{equation}
In addition, if $\curl \uu^\ep = \oo^\ep = (\omega_1^\ep,\,\omega_2^\ep,\,\omega_3^\ep)$ then
\[\omega_1^\ep = -\dfrac{\pa u_2^\ep}{\pa z}, \hspace{.5cm} \omega_2^\ep = \dfrac{\pa u_1^\ep}{\pa z}, \hspace{.5cm} \omega_3^\ep =
\dfrac{\pa u_2^\ep }{ \pa x}\]
 and

\begin{equation} \label{omega1est}
\begin{aligned}
\|\omega_1^\ep(\cdot,t)\|_{L^1(\Omega)} &\leq \left\| \dfrac{\pa g_2 }{\pa z}\right\|_{L^1(\Omega)} +
2L\left(\|g_2\|_{L^\infty(\Omega)} + \left\|\dfrac{\pa^2 g_2}{\pa x^2}\right\|_{L^{\infty}(\Omega)}\right) \\
 & + \, T\left(\left\|\dfrac{d g_1}{dz}\right\|_{L^1(\Omega)} + \|g_1\|_{L^\infty(\Omega)}\right)\left\|\dfrac{\pa g_2 }{ \pa x}\right\|_{L^{\infty}(\Omega)};
\end{aligned}
\end{equation}

\begin{equation} \label{omega2est}
\|\omega_2^\ep(\cdot,t)\|_{L^1(\Omega)} \leq \left\|\dfrac{d g_1}{dz}\right\|_{L^1(\Omega)} + 2\|g_1\|_{L^\infty(\Omega)};
\end{equation}

\begin{equation} \label{omega3est}
\|\omega_3^\ep(\cdot,t)\|_{L^{\infty}(\Omega)} \leq
\left\|\dfrac{\pa g_2 }{ \pa x}\right\|_{L^{\infty}(\Omega)},
\end{equation}
for all $0\leq t < T$.
\end{theorem}

\begin{proof}
We begin by noticing that, as $u_1^\ep$ is independent of $x,\,y$, $u_2^\ep$ is independent of $y$, and $p^\ep \equiv 0$, it follows from \eqref{NSEsymreduced} that $u_1^\ep$, $u_2^\ep$, satisfy \eqref{NSEsymreducedwopressure}.

Next, let us introduce $\alpha_n = \alpha_n(t)\in C^{\infty}([0,+\infty)) $ as below:
\begin{equation} \label{alphan}
\alpha_n=\alpha_n(t) \equiv 1 \text{ if }  t> \frac{1}{n},  \hspace{.1cm} \alpha_n=\alpha_n(t)  \equiv 0 \text{ if }  0 \leq t < \frac{1}{2n}, \hspace{.1cm} 0 \leq \alpha_n \leq 1, \hspace{.1cm} \alpha_n^{\prime}(t) \geq 0.
\end{equation}

We will start with an approximation to $u_1^\ep$, from which we will derive the bound \eqref{omega2est} for $\omega_2^\ep = \pa u_1^\ep / \pa z$.

We introduce $w_1^\eps$, the solution of
\begin{equation} \label{w1ep}
        \left\{\begin{array}{rll}
                             \dfrac{\pa w_1^\ep}{\pa t}
                             & \hspace{-2mm}
                             =   \ep \dfrac{\pa^2 w_1^\ep}{\pa z^2}, \, & \text{ in } (0,h)\times(0,T),\\
                             \\
                             w_1^\ep & \hspace{-2mm}
                             = -g_1, \, & \text{ at } \{z=0,\,h\} \times (0, T),\\
                             \\
                             w_1^\ep\big|_{t=0}
                             & \hspace{-2mm}
                             = 0, \, & \text{ in } (0,h).\\
        \end{array}\right.
\end{equation}
Observe that the data is not compatible.

Let us also introduce $v_1^\ep$ such that
\[u_1^\ep = v_1^\ep +  g_1 + w_1^\ep.\]

Now, $v_1^\ep$ satisfies a compatible problem -- both initial and boundary data vanish identically -- for a heat equation with smooth forcing, given by $\ep \partial_z^2 g_1$. We will use an approximation for $w_1^\ep$.

Set $w_1^{\ep,n}$ to be the solution of
\begin{equation} \label{w1ep}
        \left\{\begin{array}{rll}
                             \dfrac{\pa w_1^{\ep,n}}{\pa t}
                             & \hspace{-2mm}
                             =   \ep \dfrac{\pa^2 w_1^{\ep,n}}{\pa z^2}, \, & \text{ in } (0,h)\times(0,T),\\
                             \\
                             w_1^{\ep,n} & \hspace{-2mm}
                             = -\alpha_n(t)g_1, \, & \text{ at } \{z=0,\,h\} \times (0, T),\\
                             \\
                             w_1^{\ep,n}\big|_{t=0}
                             & \hspace{-2mm}
                             = 0, \, & \text{ in } (0,h).\\
        \end{array}\right.
\end{equation}
Since $\alpha_n(0)=0$ this problem is compatible.

Let $u_1^{\ep,n} \equiv v_1^\ep + g_1 + w_1^{\ep,n}.$

\begin{claim} \label{w1ntow1}
We have, passing to subsequences as needed,
\[u_1^{\ep,n} \rightharpoonup u_1^\ep,\]
in $\mathcal{D}^\prime ([0,T)\times(0,h))$, as $n\to \infty$.
\end{claim}

\begin{proof}[Proof of Claim:]

Clearly, it is enough to show that
\[w_1^{\ep,n} \rightharpoonup w_1^\ep.\]

Next, we lift the boundary data as a forcing term in the equation. Set $\ol{w_1^{\ep,n}} = w_1^{\ep,n} + \alpha_n(t)g_1$. Then $\ol{w_1^{\ep,n}}$ satisfies
\begin{equation} \label{olw1epn}
        \left\{\begin{array}{rll}
                             \dfrac{\pa  \ol{w_1^{\ep,n}}}{\pa t}
                             & \hspace{-2mm}
                             =   \ep \dfrac{\pa^2  \ol{w_1^{\ep,n}}}{\pa z^2} - \ep\alpha_n (t)\dfrac{d^2 g_1}{dz^2} + \alpha_n^\prime (t)g_1, \, & \text{ in } (0,h)\times(0,T),\\
                               \\
                        \ol{w_1^{\ep,n}} & \hspace{-2mm}
                             = 0, \, & \text{ at } \{z=0,\,h\} \times (0, T),\\
                          \\
                             \ol{w_1^{\ep,n}}\big|_{t=0}
                             & \hspace{-2mm}
                             = 0, \, & \text{ in } (0,h).\\
        \end{array}\right.
\end{equation}

We begin by observing that $\ol{w_1^{\ep,n}}$ is bounded, uniformly in $n$, in $L^{\infty}((0,T;L^2(0,h))$. Indeed, multiply the equation by $\ol{w_1^{\ep,n}}$, integrate over $(0,h)$, and divide by $\|\ol{w_1^{\ep,n}} \|_{L^2}$ to find
\[ \dfrac{d}{dt}\|\ol{w_1^{\ep,n}}\|_{L^2} \leq 2\ep \alpha_n (t)\left\|\dfrac{d^2 g_1}{dz^2} \right\|_{L^2} +
2\alpha_n^\prime (t)\| g_1 \|_{L^2},\]
where we used that $\alpha_n,\,\alpha_n^\prime \geq 0$. We obtain the uniform estimate upon integrating in time, using that $\|\ol{w_1^{\ep,n}} (0) \|_{L^2} = 0$ and
that $\alpha_n \leq 1$, $\displaystyle{\int_0^T \alpha_n^\prime (s) \, ds = 1}$.

It follows from the Banach-Alaoglu theorem that, passing to subsequences as needed, there exists $R \in L^{\infty}((0,T;L^2(0,h))$ such that $\ol{w_1^{\ep,n}} \rightharpoonup R$ weak-$\ast$ $L^{\infty}((0,T;L^2(0,h))$. Hence, we also have  $\ol{w_1^{\ep,n}} \rightharpoonup R$ in $\mathcal{D}^\prime([0,T)\times(0,h))$.

Let $\varphi \in C^{\infty}([0,T)\times [0,h])$. Assume $\varphi (\cdot,z) \in C^{\infty}_c([0,T))$ for every $z \in [0,h]$ and, additionally, $\varphi(t,0)=\varphi(t,h)=0$ for each $t \geq 0$. Multiply the equation for $\ol{w_1^{\ep,n}}$ by $\varphi$ and integrate in time and space, transferring all derivatives to $\varphi$, including the time-derivative of $\alpha_n$, to obtain a weak formulation for \eqref{olw1epn}. Since the equation is linear it follows, from weak convergence of $\ol{w_1^{\ep,n}}$ to $R$ and because $\alpha_n \to \chi_{(0,+\infty)}$ strongly in $L^1$, that
\[
-\int_0^T\int_0^h \pa_t \varphi R = \ep \int_0^T\int_0^h\pa_z^2\varphi R - \ep\int_0^T\int_0^h\varphi \dfrac{d^2g_1}{dz^2} +\int_0^h \varphi(0,z) g_1.
\]

Let us now introduce $ S \equiv R - g_1$. Clearly, it holds that:
\[ -\int_0^T\int_0^h \pa_t \varphi S = \ep \int_0^T\int_0^h\pa_z^2\varphi S + \ep \int_0^T\int_0^h\pa_z^2\varphi g_1- \varphi \dfrac{d^2g_1}{dz^2}.\]

Taking $\varphi \in C^\infty_c((0,T)\times(0,h))$ we obtain that $S$ is a distributional solution of the heat equation in $(0,T)\times(0,h)$. Taking now $\varphi \in C^\infty_c((0,T)\times[0,h])$, with $\varphi(t,0)=\varphi(t,h)=0$, we deduce that $S = - g_1$ at $z=0,\,h$. Finally, taking $\varphi \in C^\infty_c([0,T)\times(0,h))$ we deduce that $S = 0$ at $t=0$. Hence, by uniqueness for \eqref{w1ep}, it follows that $S = w_1^\ep$.

\end{proof}

Having established the claim, we now prove uniform estimates for $\omega_2^n\equiv \pa_z u_1^{\ep,n} $. Then \eqref{omega2est} will follow from these estimates, together with the weak convergence $u_1^{\ep,n} \rightharpoonup u_1^\ep$.

Start by observing that $u_1^{\ep,n}$ satisfies
\begin{equation} \label{u1epn}
      \left\{
             \begin{array}{rll}
                             \dfrac{\pa u_1^{\eps,n}}{\pa t}
                             & \hspace{-2mm}
                             = \ep \dfrac{\pa^2 u_1^{\eps,n}}{\pa z^2}, \, & \text{ in } (0,h)\times(0,T),\\

\\
             u_1^{\eps,n}
                             & \hspace{-2mm}
                              = (1-\alpha_n(t)) g_1, \, & \text{ at } \{z=0,h\} \times (0, T),\\
\\
             u_1^{\eps,n}\big|_{t=0}
                              & \hspace{-2mm}
                              = g_1 (z), \, & \text{ in } (0,h).
             \end{array}
      \right.
\end{equation}

Differentiate the equation for $u_1^{\ep,n}$ with respect to $z$ to find, easily,
\[\dfrac{\pa \omega_2^n}{\pa t} = \ep \dfrac{\pa^2 \omega_2^n}{\pa z^2}.\]
We proceed in the spirit of \eqref{LF}: evaluate the evolution equation for $u_1^{\ep,n}$, \eqref{u1epn} at the boundary $z=0$, $z=h$, to find
\begin{equation} \label{bdrycondomega2n}
\dfrac{\pa \omega_2^n}{\pa z}\big|_{z=0,h} = -\frac{1}{\ep} \alpha_n^\prime (t)g_1\big|_{z=0,h}.
\end{equation}
The initial condition for $\omega_2^n$ is clearly $\omega_2^n(z,t=0)=dg_1/dz$. Putting together the equation for $\omega_2^n$, the boundary condition \eqref{bdrycondomega2n},
and the initial data yields the Cauchy problem below:
\begin{equation} \label{omega2nsys}
\left\{
        \begin{array}{rll}
                             \dfrac{\pa \omega_2^n}{\pa t}
                             & \hspace{-2mm}
                             = \ep \dfrac{\pa^2 \omega_2^n}{\pa z^2}, \, & \text{ in } (0,h)\times(0,T),\\
                             \\
             \dfrac{\pa \omega_2^n}{\pa z}
                             & \hspace{-2mm}
                              = -\dfrac{1}{\ep} \alpha_n^\prime g_1, \, & \text{ at } \{z=0,h\} \times (0, T),\\
                             \\
             \omega_2^n\big|_{t=0}
                              & \hspace{-2mm}
                              = \dfrac{d g_1}{dz}, \, & \text{ in } (0,h).
        \end{array}
\right.
\end{equation}

Fix $\delta > 0$ and set $\varphi_\delta = \varphi_\delta (s) \equiv \sqrt{\delta^2 + s^2}$. Of course, $\varphi_\delta (s) \to |s|$ pointwise, as $\delta \to 0$. In addition,
\[|\varphi_\delta^\prime(s)| \leq 1; \hspace{1cm} \varphi_\delta^{\prime\prime}(s) \geq 0.\]
Multiply the equation for $\omega_2^n$ by $\varphi_\delta^\prime(\omega_2^n)$ and integrate on $(0,h)$ to find, upon integration by parts and using the Neumann boundary condition \eqref{bdrycondomega2n}:
\begin{equation} \label{varphideltaomega2n}
\begin{aligned}
\dfrac{d}{dt} \int_0^h \varphi_\delta(\omega_2^n) \, dz & = \ep \int_0^h \varphi_\delta^\prime (\omega_2^n) \dfrac{\pa^2 \omega_2^n}{\pa z^2} \, dz \\ \\
 & = - \ep \int_0^h \varphi_\delta^{\prime \prime}(\omega_2^n) \left(\dfrac{\pa \omega_2^n}{\pa z}\right)^2\,dz +
 \ep\int_0^h \pa_z [\varphi_\delta^\prime(\omega_2^n) \pa_z \omega_2^n]\, dz \\ \\
 & \leq -\alpha_n^\prime(t)[\varphi_\delta^\prime(\omega_2^n)g_1]\big|_{z=0}^{z=h} \\ \\
 & \leq |\alpha_n^\prime(t)|(|g_1(h)|+|g_1(0)|).
\end{aligned}
\end{equation}
Integrating \eqref{varphideltaomega2n} in time from $0$ to $t$ and taking the limit $\delta \to 0$ we obtain
\[\|\omega_2^n(t)\|_{L^1(0,h)} \leq \|\omega_2^n(t=0)\|_{L^1(0,h)} + |g_1(h)| + |g_1(0)| \leq \|dg_1/dz\|_{L^1(0,h)} + 2\|g_1\|_{L^{\infty}([0,h])}.\]
Estimate \eqref{omega2est} follows by taking $n \to \infty$, using the weak lower semicontinuity of $\|\cdot\|_{L^1}$, in view of the convergence $u_1^{\ep,n} \to u_1^\ep$ in the sense of distributions.

Next we treat $u_2^\ep$. The approximation is quite similar. We introduce $w_2^\eps$, the solution of
\begin{equation} \label{w2ep}
        \left\{\begin{array}{rll}
                             \dfrac{\pa w_2^\ep}{\pa t} + u_1^\ep \dfrac{\pa w_2^\ep}{\pa x}
                             & \hspace{-2mm}
                             =   \ep \Delta_{x,z} w_2^\ep, \, & \text{ in } (0,L)\times (0,h)\times(0,T),\\
                             \\
                             w_2^\ep & \hspace{-2mm}
                             = -g_2, \, & \text{ at } (0,L) \times \{0,\,h\} \times (0, T),\\
                             \\
                             w_2^\ep\big|_{t=0}
                             & \hspace{-2mm}
                             = 0, \, & \text{ in } (0,L)\times (0,h).\\
        \end{array}\right.
\end{equation}
We also require $w_2^\ep$ to be periodic in $x$ with period $L$. We note, as before, that the data is not compatible.

We introduce $v_2^\ep$ so that
\[u_2^\ep = v_2^\ep +  g_2 + w_2^\ep.\]

As before, $v_2^\ep$ satisfies a compatible problem -- both initial and boundary data vanish identically -- for a drift-diffusion equation, with drift $u_1^\ep$, and with smooth (in the interior of $(0,L)\times(0,h)\times (0,T)$) forcing, given by
$\ep \Delta_{x,z} g_2 - u_1^\ep \pa_x g_2$. Similarly to what we did for $u_1^\ep$, we will make use of an approximation for $w_2^\ep$.

Set $w_2^{\ep,n}$ to be the solution of
\begin{equation} \label{w2ep}
        \left\{\begin{array}{rll}
                             \dfrac{\pa w_2^{\ep,n}}{\pa t} + u_1^\ep\dfrac{\pa w_2^{\ep,n}}{\pa x}
                             & \hspace{-2mm}
                             =   \ep \Delta_{x,z} w_2^{\ep,n}, \, & \text{ in } (0,L) \times (0,h)\times(0,T),\\
                             \\
                             w_2^{\ep,n} & \hspace{-2mm}
                             = -\alpha_n(t)g_2, \, & \text{ at } (0,L)\times \{0,\,h\} \times (0, T),\\
                             \\
                             w_2^{\ep,n}\big|_{t=0}
                             & \hspace{-2mm}
                             = 0, \, & \text{ in } (0,L)\times (0,h).\\
        \end{array}\right.
\end{equation}
Impose periodic boundary conditions at $x=0$, $x=L$.
Since $\alpha_n(0)=0$ this problem is compatible.

Let $u_2^{\ep,n} \equiv v_2^\ep + g_2 + w_2^{\ep,n}.$

\begin{claim} \label{w2ntow2}
We have, passing to subsequences as needed,
\[u_2^{\ep,n} \rightharpoonup u_2^\ep,\]
in $\mathcal{D}^\prime ([0,T)\times(0,L) \times(0,h))$, periodic in $x$, as $n\to \infty$.
\end{claim}

\begin{proof}[Proof of Claim:]

As before, clearly, it is enough to show that
\[w_2^{\ep,n} \rightharpoonup w_2^\ep.\]

We lift the boundary data as a forcing term in the equation. Set $\ol{w_2^{\ep,n}} = w_2^{\ep,n} + \alpha_n(t)g_2$. Then $\ol{w_2^{\ep,n}}$ satisfies
\begin{equation} \label{olw2epn}
        \left\{\begin{array}{rll}
                             \dfrac{\pa  \ol{w_2^{\ep,n}}}{\pa t} + u_1^\ep\dfrac{\pa \ol{w_2^{\ep,n}}}{\pa x}
                             & \hspace{-2mm}
                             =   \ep \Delta_{x,z}\ol{w_2^{\ep,n}} - \ep\alpha_n (t)\Delta_{x,z} g_2  & \\
                             & \hspace{-2mm}
                             + u_1^\ep \alpha_n(t)\dfrac{\pa g_2}{\pa x} + \alpha_n^\prime (t)g_2, \, & \text{ in } (0,L) \times (0,h)\times(0,T),\\
                               \\
                        \ol{w_2^{\ep,n}} & \hspace{-2mm}
                             = 0, \, & \text{ at } (0,L) \times \{0,\,h\} \times (0, T),\\
                          \\
                             \ol{w_2^{\ep,n}}\big|_{t=0}
                             & \hspace{-2mm}
                             = 0, \, & \text{ in } (0,L) \times (0,h).\\
        \end{array}\right.
\end{equation}
Additionally, $\ol{w_2^{\ep,n}}$ is periodic in $x$.

We note that $\ol{w_2^{\ep,n}}$ is bounded, uniformly in $n$, in $L^{\infty}((0,T;L^2((0,L)\times(0,h)))$. Indeed, we have, easily,
\[ \dfrac{d}{dt}\|\ol{w_2^{\ep,n}}\|_{L^2} \leq 2\ep \alpha_n (t)\|\Delta_{x,z} g_2 \|_{L^2} + 2 \alpha_n(t) \|u_1^\ep \pa_x g_2\|_{L^2} +
2\alpha_n^\prime (t)\| g_2 \|_{L^2},\]
where we used, once again, that $\alpha_n,\,\alpha_n^\prime \geq 0$. We obtain the uniform estimate upon integrating in time, using that $\|\ol{w_2^{\ep,n}} (0) \|_{L^2} = 0$, that $\sup_{(0,T)}\|u_1^\ep g_2\|_{L^2} < \infty$, and
that $\alpha_n \leq 1$, $\displaystyle{\int_0^T \alpha_n^\prime (s) \, ds = 1}$.

The remainder of the argument used to establish Claim \ref{w1ntow1} can now be used, with the appropriate modifications, to conclude the proof of the present claim.

\end{proof}

Next, we use Claim \ref{w2ntow2} to establish \eqref{omega3est}.

Note that $u_2^{\ep,n}$ satisfies
\begin{equation} \label{u2epn}
      \left\{
             \begin{array}{rll}
                             \dfrac{\pa u_2^{\eps,n}}{\pa t} + u_1^\ep\dfrac{\pa u_2^{\ep,n}}{\pa x}
                             & \hspace{-2mm}
                             = \ep \Delta_{x,z} u_2^{\eps,n}, \, & \text{ in } (0,L)\times (0,h)\times(0,T),\\

\\
             u_2^{\eps,n}
                             & \hspace{-2mm}
                              = (1-\alpha_n(t)) g_2, \, & \text{ at } (0,L)\times \{0,h\} \times (0, T),\\
\\
             u_2^{\eps,n}\big|_{t=0}
                              & \hspace{-2mm}
                              = g_2 (z), \, & \text{ in } (0,L)\times (0,h).
             \end{array}
      \right.
\end{equation}
Moreover, $u_2^{\ep,n}$ is periodic in $x$.

Let $\omega_3^n\equiv \pa u_2^{\ep,n} / \pa x$. Differentiating the equation for $u_2^{\ep,n}$ with respect to $x$ yields, easily,
\[\dfrac{\pa \omega_3^n}{\pa t}  + u_1^\ep \dfrac{\pa \omega_3^n}{\pa x} = \ep \Delta_{x,z}\omega_3^n.\]
We now evaluate the evolution equation for $u_2^n$, \eqref{u2epn}, at the boundary $[0,L]\times\{0,\,h\}$, to obtain
\begin{equation} \label{bdrycondomega3n}
\omega_3^n(x,\cdot,t)\big|_{z=0,h} = (1-\alpha_n(t)\pa_x g_2(x,\cdot)\big|_{z=0,h}.
\end{equation}
In addition we have $\omega_3^n(x,z,t=0)=\pa_x g_2(x,z)$. Putting together the equation for $\omega_3^n$, the boundary condition \eqref{bdrycondomega3n},
and the initial data yields the Cauchy problem below:
\begin{equation} \label{omega3nsys}
\left\{
        \begin{array}{rll}
                             \dfrac{\pa \omega_3^n}{\pa t} + u_1^\ep \dfrac{\pa \omega_3^n}{\pa x}
                             & \hspace{-2mm}
                             = \ep \Delta_{x,z}\omega_3^n, \, & \text{ in } [0,L]\times (0,h)\times(0,T),\\
                             \\
             \omega_3^n(\cdot,t)
                             & \hspace{-2mm}
                              = (1-\alpha_n(t))\pa_x g_2(\cdot), \, & \text{ on } [0,L]\times \{z=0,h\} \times (0, T),\\
                             \\
             \omega_3^n\big|_{t=0}
                              & \hspace{-2mm}
                              = \pa_x g_2(x,z), \, & \text{ in } [0,L]\times (0,h).
        \end{array}
\right.
\end{equation}

Since all the coefficients and data are smooth ($u_1^\ep$ is smooth for $t>0$), we will have a smooth solution to which we can apply the maximum principle for the operator $\pa_t + u_1^\ep\pa_x - \ep\Delta_{x,z}$. We deduce that
\[\max_{[0,L]\times [0,h]\times[0,T]}|\omega_3^n(x,z,t)|  = \max\left\{\max_{[0,L]\times\{0,h\}\times [0,T]}|\omega_3^n|, \hspace{0.3cm} \max_{[0,L]\times[0,h]}|\omega_3^n(\cdot,\cdot,0)|\right\},\]
i.e.,
\[\|\omega_3^n\|_{L^\infty([0,L]\times[0,h]\times[0,T])} \leq \|\pa_x g_2\|_{L^\infty([0,L]\times[0,h])}.\]
Estimate \eqref{omega3est} follows by taking $n \to \infty$, given that $u_2^n \to u_2^\ep$ in the sense of distributions.

Finally, we analyze $\omega_1^\ep=-\pa_z u_2^\ep$. We note that the equation for $\omega_1^\ep$ is the most complicated because it is the only equation with a vorticity stretching term, namely, $-\omega_2^\ep \omega_3^\ep$. This will impact the analysis for the approximations as well.

Let $\omega_1^n\equiv -\pa u_2^{\ep,n} / \pa z$. Differentiate the equation for $u_2^{\ep,n}$ with respect to $z$ to obtain
\[\dfrac{\pa \omega_1^n}{\pa t}  + u_1^\ep \dfrac{\pa \omega_1^n}{\pa x} - \omega_2^n \omega_3^n = \ep \Delta_{x,z}\omega_1^n.\]
As in \eqref{LF}, assume that the evolution equation for $u_2^{\ep,n}$, \eqref{u2epn}, remains valid up to the boundary $[0,L]\times \{0,h\}$, so that, since $u_1^\ep$ vanishes at this boundary, for all $t>0$, we have
\[\dfrac{\pa u_2^n}{\pa_t}\big|_{z=0,h} = \ep (1-\alpha_n(t))\pa^2_x g_2 - \ep \dfrac{\pa \omega_1^n}{\pa z}\big|_{z=0,h},\]
hence
\begin{equation} \label{bdrycondomega1n}
\dfrac{\pa \omega_1^n}{\pa z}\big|_{z=0,h} = \frac{1}{\ep} \alpha_n^\prime g_2\big|_{z=0,h} + (1-\alpha_n(t))\pa^2_x g_2.
\end{equation}
In addition we have $\omega_1^n(z,t=0)=-\pa_z g_2$. Putting together the equation for $\omega_1^n$, the boundary conditions \eqref{bdrycondomega1n}, and the initial data yields the Cauchy problem below:
\begin{equation} \label{omega1nsys}
\left\{
        \begin{array}{rll}
                             \dfrac{\pa \omega_1^n}{\pa t}  + u_1^\ep \dfrac{\pa \omega_1^n}{\pa x} - \omega_2^n \omega_3^n
                             & \hspace{-2mm}
                             = \ep \Delta_{x,z}\omega_1^n, \, & \text{ in } [0,L]\times (0,h)\times(0,T),\\
                             \\
             \dfrac{\pa \omega_1^n}{\pa z}
                             & \hspace{-2mm}
                              = \dfrac{1}{\ep} \alpha_n^\prime g_2 + (1-\alpha_n(t))\pa^2_x g_2, \, & \text{ at } [0,L]\times \{z=0,h\} \times (0, T),\\
                             \\
             \omega_1^n\big|_{t=0}
                              & \hspace{-2mm}
                              = -\pa_z g_2 (x,z), \, & \text{ in } [0,L]\times (0,h).
        \end{array}
\right.
\end{equation}

Fix $\delta > 0$ and consider $\varphi_\delta = \varphi_\delta (s)$.
As we did for $\omega_2^n$, multiply the equation for $\omega_1^n$ by $\varphi_\delta^\prime(\omega_1^n)$ and integrate on $[0,L]\times (0,h)$ to find, upon integration by parts and using the Neumann boundary condition \eqref{bdrycondomega1n}:
\begin{equation} \label{varphideltaomega1n}
\begin{aligned}
\dfrac{d}{dt} \int_0^L \int_0^h \varphi_\delta(\omega_1^n) \,dx dz & = \ep \int_0^L \int_0^h \varphi_\delta^\prime (\omega_1^n)
\Delta_{x,z}\omega_1^n\,dx dz  + \int_0^L\int_0^h \varphi_\delta^\prime(\omega_1^n)\omega_2^n\omega_3^n \, dxdz\\ \\
 & = - \ep \int_0^L\int_0^h \varphi_\delta^{\prime \prime}(\omega_1^n) |\nabla_{x,z}\omega_1^n|^2\,dxdz +
 \ep\int_0^L\int_0^h \pa_z [\varphi_\delta^\prime(\omega_1^n) \pa_z \omega_1^n]\, dxdz \\ \\
 & + \int_0^L\int_0^h \varphi_\delta^\prime(\omega_1^n)\omega_2^n\omega_3^n \, dxdz\\ \\
 & \leq \alpha_n^\prime(t)\int_0^L [\varphi_\delta^\prime(\omega_1^n)g_2]\big|_{z=0}^{z=h} \,dx + \ep (1-\alpha_n(t))\int_0^L
[\varphi_\delta^\prime(\omega_1^n)\pa_x^2 g_2]\big|_{z=0}^{z=h} \,dx \\ \\
& + \|\omega_2^n(\cdot,t)\|_{L^1(0,h)}\|\omega_3^n(\cdot,t)\|_{L^\infty([0,L]\times[0,h])}.\\ \\
\end{aligned}
\end{equation}
Integrating \eqref{varphideltaomega1n} in time from $0$ to $t$ and taking the limit $\delta \to 0$ we obtain
\begin{equation*}
\begin{aligned}
\|\omega_2^n(t)\|_{L^1([0,L]\times (0,h))} & \leq \|\omega_1^n\big|_{t=0}\|_{L^1([0,L]\times(0,h))} + 2L(\|g_2\|_{L^\infty([0,L]\times (0,h))} +\|\pa_x^2 g_2\|_{L^{\infty}([0,L]\times(0,h))})\\ \\
 &+ T\|\omega_2^n\|_{L^\infty((0,T);L^1(0,h))}
\|\omega_3^n \|_{L^\infty([0,L]\times[0,h]\times(0,T))}\\ \\
& =  \|\pa_z g_2\|_{L^1([0,L]\times(0,h))} + 2L(\|g_2\|_{L^\infty([0,L]\times (0,h))} +\|\pa_x^2 g_2\|_{L^{\infty}([0,L]\times(0,h))})\\ \\
& + T(\|dg_1/dz\|_{L^1(0,h)} + 2\|g_1\|_{L^{\infty}([0,h])})
\|\pa_x g_2\|_{L^\infty([0,L]\times[0,h])}.\\
\end{aligned}
\end{equation*}
Estimate \eqref{omega1est} follows by taking $n \to \infty$, given that $u_2^n \rightharpoonup u_2^\ep$ in the sense of distributions.
This concludes the proof.

\end{proof}

Analogous results hold for flow in the pipe. More precisely, it is possible to obtain a version of Proposition \ref{LHandtangPCF} for flow in a pipe and a version of Theorem \ref{vortestimates} for PPF.

%
%
\section{Plane-parallel channel flows}\label{S:PPF}

In this section, we will present the asymptotic description of the vanishing viscosity limit for PCF, significantly extending the
analysis in \cite{MNW10,MT08}.

We consider NSE and EE in an infinite channel, but impose periodic boundary
conditions in the streamwise direction. The fluid domain is
\begin{align*}
    \Omega := (0, L)^2 \times (0, h)
        \text{ with } \Gamma := \prt \Omega = (0, L)^2 \times \set{0, h},
\end{align*}
for a fixed $h > 0$, and we consider flows which are periodic
in both the $x$ and $y$ directions, with period $L>0$.

We study plane-parallel solutions of the fluid equations of the form:
\begin{equation}\label{e:PP velocity_NSE}
\uu = (u_1 (z, t), \, u_2 (x, z, t), \, 0 ).
\end{equation}
The initial data and the forcing will be taken to satisfy the same symmetry,
that is:
\begin{equation*}\label{e:PP data}
            \f = (f_1 (z, t), \, f_2 (x, z, t), \, 0),
\qquad
            \uu_0 = (u_{0, \, 1} (z), \, u_{0, \, 2} (x, z), \, 0).
\end{equation*}
Under this symmetry, all vector fields are divergence free and automatically
satisfy the no-penetration condition. In addition, it is easy to see that the pressure can be taken to be
zero in both the NSE and EE and, therefore, will not enter the ensuing
calculations.

It can be shown that the forward evolution under the NSE and EE preserves the
symmetry, at least for strong solutions (cf. e.g. \cite{MT08}).
We hence consider the symmetry-reduced NSE (\ref{e:NSE}), which become the following weakly
non-linear system:
\begin{equation}\label{e:NSE_PP}
\left\{\begin{array}{rl}
                             \spacer
                                        \dfrac{\pa u^\eps_1}{\pa t} - \ep
\dfrac{\pa^2 u^\eps_1}{\pa z^2}
                                        & \hspace{-2mm}
                                        = f_1
                                            \text{ in } \Omega\times(0,T),\\
                             \spacer
                                        \dfrac{\pa u^\eps_2}{\pa t}
                                        - \ep \dfrac{\pa^2 u^\eps_2}{\pa x^2}
                                        - \ep \dfrac{\pa^2 u^\eps_2}{\pa z^2}
                                        + u^\eps_1 \, \dfrac{\pa u^\eps_2}{\pa
x}
                                        & \hspace{-2mm}
                                        = f_2
                                            \text{ in } \Omega\times(0,T),\\
                              \halfspacer
                                        u^\eps_2 \text{ is periodic}
                                        & \hspace{-2mm}
                                        \text{in $x$ with period $L$},\\
                              \halfspacer
                                        u^\eps_i
                                        & \hspace{-2mm}
                                        = 0, \text{ } i=1,2, \text{ on }
\Gamma,\\
                                        u^\eps_i \big|_{t=0}
                                        & \hspace{-2mm}
                                        = u_{0, \, i}, \text{ } i=1,2, \text{ in
} \Omega.
        \end{array}\right.
\end{equation}

Similarly, we consider the symmetry-reduced  EE (\ref{e:EE}):
\begin{equation}\label{e:EE_PP}
\left\{\begin{array}{rl}
                             \spacer
                                        \dfrac{\pa u^0_1}{\pa t}
                                        & \hspace{-2mm}
                                        = f_1
                                            \text{ in } \Omega\times(0,T),\\
                             \spacer
                                        \dfrac{\pa u^0_2}{\pa t} + u^0_1 \,
\dfrac{\pa u^0_2}{\pa x}
                                        & \hspace{-2mm}
                                        = f_2
                                            \text{ in } \Omega\times(0,T),\\
                              \spacer
                                        u^0_2
                                        & \hspace{-2mm}
                                        \text{is periodic in $x$ direction with
period $L$},\\
                                        u^0_i \big|_{t=0}
                                        & \hspace{-2mm}
                                        = u_{0, \, i}, \text{ } i=1,2, \text{ in
} \Omega.
        \end{array}\right.
\end{equation}
We assume that the data, $\uu_0$ and $\f$, are sufficiently regular, but
ill-prepared in the sense that
\begin{equation}\label{e:Data_reg_PPF}
        \uu_0 \in H \cap H^k(\Omega),
            \quad
        \f \in C(0, T ; H \cap H^k(\Omega)),
            \quad
        \text{for a sufficiently large } k \geq0.
\end{equation}
Note that we do not assume that $\uu_0$ vanishes at the boundary, nor does $\f$ have
to be compatible with $\uu_0$ at $t=0$.
Under these regularity assumptions, the NSE and EE have, both, global-in-time strong
solutions  (see e.g. \cite{MT08}).

Under the plane-parallel symmetry, the tangential components of
the EE velocity $\blds{u}^0$ need not vanish. Therefore, a viscous boundary layer is expected
to form  to account for the mismatch in the tangential components of the NSE and
EE velocities at the boundary. The fact that the data is ill prepared leads to
an initial layer for NSE, which also affects the zero-viscosity limit.

In the following subsections, we construct correctors for the Euler flow that
lead to an asymptotic expansion of $\uu^\eps$ at small viscosity $\ep$. This
expansion will be used to study the boundary layer and the vanishing viscosity
limit. This expansion is not a Prandtl-type expansion as that used in
\cite{MNW10} and lends itself somewhat naturally to study accumulation of
vorticity at the boundary in the limit.

\subsection{
Viscous approximation and
convergence result}\label{S:PPF_THM}




We postulate an approximation of the viscous solution of the form
\begin{equation}\label{e:assympt exp_PPF}
        \uu^\eps
                \cong
                    \uu^0 + \blds{\Tht},
\end{equation}
where $\blds{\Tht}$ is a corrector to the inviscid solution $\uu^0$.
The corrector depends on $\eps$, but for sake of notation, we will not
explicitly denote it.
The corrector will be assumed to satisfy the same
symmetry as the fluid velocities, that is,
\begin{equation}\label{e:Tht_PPF}
        \blds{\Tht}
                = \big(\Tht_1 (z, t), \, \Tht_2 (x, z, t), \, 0 \big).
\end{equation}
This assumption is justified by the fact that the flow remains laminar and there
is no boundary layer separation in this case (see e.g. \cite{MT08}).

In recent related work, see \cite{MNW10}, a viscous aproximation to the NS
solution similar to (\ref{e:assympt exp_PPF})was introduced, where the
corrector, which we  call $\blds{\Upsilon} = \big(\Upsilon_1 (z, t), \,
\Upsilon_2 (x, z, t), \, 0 \big)$ here, is defined as the solution of the
\textit{weakly coupled parabolic system},
\begin{equation}\label{e:Eq_MNW10}
        \left\{
                \begin{array}{rl}
                        \spacer \displaystyle
                                \dfrac{\pa \Upsilon_1}{\pa t} - \ep \dfrac{\pa^2
\Upsilon_1}{\pa z^2} = 0
                        & \displaystyle
                                \text{in } \Omega \times (0, T),\\
                        \spacer \displaystyle
                                \dfrac{\pa \Upsilon_2}{\pa t}
                                - \ep \dfrac{\pa^2 \Upsilon_2}{\pa z^2}
                                + \Upsilon_1 \dfrac{\pa \Upsilon_2}{\pa x}
                                + u^0_1 \dfrac{\pa \Upsilon_2}{\pa x}
                                + \Upsilon_1 \dfrac{\pa u^0_2}{\pa x}
                                = 0
                        & \displaystyle
                                \text{in } \Omega \times (0, T),\\
                        \spacer \displaystyle
                                \Upsilon_i = - u^0_i
                        &
                                \text{on } \Gamma \times (0,T), \text{ } i=1,2,\\
                        \blds{\Upsilon}|_{t=0} = 0.
                \end{array}
        \right.
\end{equation}
In fact, (\ref{e:Eq_MNW10}) is a reduced form of Prandtl's equations under the
plane-parallel symmetry given in (\ref{e:PP velocity_NSE}).
Assuming well-prepared data, bounds on $\Upsilon$ with an explicit dependency on
$\ep$ in the Sobolev space $H^k$,  for some $k\geq0$, were derived from energy
estimates. Using these bounds,  the authors verify the validity of the
vanishing viscosity limit. Certain higher-order expansions are discussed as
well.

In this article, we tackle the case of ill-prepared data. Since the second
component of the velocity is advected by the first component in
(\ref{e:NSE_PP}),  we first
construct $\Tht_1$; then, using $\Tht_1$, we construct the second component of the corrector,
$\Tht_2$.

The ingredients needed to construct both $\Tht_1$ and $\Tht_2$ are an explicit solution of
the heat equation on the half-line and a solution of a drift-diffusion
equation in a periodic channel.

Let $\Phi[g]=\Phi[g](\eta,t)$, $\eta>0$, $t>0$, denote the solution of \eqref{e:HEAT_eq} with boundary data $g=g(t)$. Set
\begin{equation}\label{e:Tht_LR}
\left\{
             \begin{array}{l}
                \spacer \displaystyle
                \Tht_{1, \, L} (z, t)
                            = \Phi[g_L](z,t)\\
                \displaystyle
                \Tht_{1, \, U} (z, t)
                            = \Phi[g_U](h-z,t),
             \end{array}
\right.
\end{equation}
where $g_L=g_L(t)\equiv u_1^0(0,t)$ and $g_U=g_U(t)\equiv u_1^0(h,t)$ are the boundary data at $z=0$ and $z=h$, respectively, for the first component of the EE solution.

Next, we introduce a smooth
cut-off function, $\sigma$ such that
\begin{equation}\label{e:sig_LR}
\begin{array}{ccc}

\sig \in C^{\infty}, & 0\leq \sig \leq 1, &
                \sig(z) =
                        \left\{
                        \begin{array}{l}
                                \spacer
                                        1, \text{ } 0 \leq z \leq h/4,\\
                                        0, \text{ } z  \geq h/2.
                        \end{array}
                        \right.
\end{array}
\end{equation}
The role of this cut-off is to localize the correctors near each of the
channel walls.

Using (\ref{e:Tht_LR}) and (\ref{e:sig_LR}), we define the first component
of the corrector as:
\begin{equation}\label{e:Tht_1}
        \Tht_1(z, t)
                = \sig(z) \, \Tht_{1, \, L} (z, t)
                    + \sig (h-z) \, \Tht_{1, \, U} (z, t).
\end{equation}
It is easy to see that $\Tht_1$ satisfies
\begin{equation}\label{e:Eq_Tht_1}
        \left\{
                \begin{array}{rl}
                        \spacer \displaystyle
                                \dfrac{\pa \Tht_1}{\pa t} - \ep \dfrac{\pa^2
\Tht_1}{\pa z^2}
                        & \hspace{-2mm}
                                =
                                   - \ep \Big\{
                                                2 \sig^{\pri} (z)\, \dfrac{\pa
\Tht_{1, \, L}}{\pa z}

                                                - 2 \sig^{\pri} (h-z)\, \dfrac{\pa
\Tht_{1, \, U}}{\pa z} \Big. \\

& \hspace{-2mm} + \Big. \sig^{\pri \pri} (z) \, \Tht_{1,
\, L}
                                                + \sig^{\pri \pri} (h-z) \, \Tht_{1,
\, U}
                                            \Big\}
                                \quad
                                \text{in } \Omega \times (0, T),\\
                        \spacer \displaystyle
                                \Tht_1
                        & \hspace{-2mm}
                                = - u^0_1
                                \quad
                                \text{on } \Gamma \times (0, T), \\
                        \Tht_1|_{t=0}
                        & \hspace{-2mm}
                                = 0.
                \end{array}
        \right.
\end{equation}
The right-hand side of (\ref{e:Eq_Tht_1})$_1$ is an $e.s.t.$, a fact that will be
verified later.


Next, we turn to the construction of the second component of the corrector. Let $\Psi[\mathcal{U},G,g]=\Psi[\mathcal{U},G,g](\tau,\eta,t)$, $0 < \tau < L$, $\eta > 0$, $t>0$, periodic in $\tau$, denote the solution of the drift-diffusion equation \eqref{e:HEAT_like_eq} with drift velocity $\mathcal{U}= \mathcal{U}(\eta,t)$, forcing $G=G(\tau,\eta,t)$ and boundary data $g=g(\tau,t)$, at $\eta=0$. We introduce the lower and upper drift velocity, forcing and boundary data, as follows:
\begin{equation}\label{e:UGg L}
\begin{array}{l}
\mathcal{U}_L = \mathcal{U}_L (\eta,t) \equiv u_1^0(\eta,t) + \Tht_1(\eta,t)
\\ \\
G_L = G_L(\tau,\eta,t) \equiv - \Tht_1(\eta,t)\partial_{\tau}u_2^0(\tau,\eta,t)
\\ \\
g_L = g_L(\tau,t)=-u_2^0(\tau,0,t)
\end{array}
\end{equation}

\vspace{.5cm}

and

\vspace{.5cm}

\begin{equation}\label{e:UGg R}
\begin{array}{l}
\mathcal{U}_U = \mathcal{U}_U (\eta,t) \equiv u_1^0(h-\eta,t) + \Tht_1(h-\eta,t)
\\ \\
G_U = G_U(\tau,\eta,t) \equiv - \Tht_1(h-\eta,t)\partial_{\tau}u_2^0(\tau,h-\eta,t)
\\ \\
g_U = g_U(\tau,t)=-u_2^0(\tau,h,t).
\end{array}
\end{equation}

Set
\begin{equation}\label{e:Tht2_LR}
\left\{
             \begin{array}{l}
                \spacer \displaystyle
                \Tht_{2, \, L} (x,z, t)
                            = \Psi[\mathcal{U}_L,G_L,g_L](x,z,t)\\
                \displaystyle
                \Tht_{2, \, U} (x,z, t)
                            = \Psi[\mathcal{U}_U,G_U,g_U](x, h-z,t),
             \end{array}
\right.
\end{equation}
and note that $\Tht_{2,\,L}$ is defined on $x \in (0,L)$, $z>0$, $t>0$, while $\Tht_{2,\,U}$ is defined on $x \in (0,L)$,
$z<h$, $t>0$.

We will use once more the cut-off $\sig$ (\ref{e:sig_LR}) to define the second component of our corrector:
\begin{equation}\label{e:Tht_2}
        \Tht_2(x, z, t)
                = \sig(z) \, \Tht_{2, \, L} (x, z, t)
                    + \sig(h-z) \, \Tht_{2, \, U} (x, z, t).
\end{equation}

It is a simple calculation to verify that $\Tht_2$ satisfies:
\begin{equation}\label{e:Eq_Tht_2}
        \left\{
                \begin{array}{l}
                        \spacer \displaystyle
                                \dfrac{\pa \Tht_2}{\pa t} - \ep \left(\dfrac{\pa^2
\Tht_2}{\pa z^2}
                                + \dfrac{\pa^2 \Tht_2}{\pa x^2} \right)
                                + (u^0_1 + \Tht_1) \dfrac{\pa \Tht_2}{\pa x}
                                \\
                        \spacer \qquad \qquad
                        		=
				- \Tht_1 \dfrac{\pa u^0_2}{\pa x}

- \ep \Big\{
                                                2 \sig^{\pri} (z)\, \dfrac{\pa
\Tht_{2, \, L}}{\pa z}

                                                - 2 \sig^{\pri} (h-z)\, \dfrac{\pa
\Tht_{2, \, U}}{\pa z} \Big. \\

 \spacer \qquad \qquad \qquad \qquad \qquad + \Big. \sig^{\pri \pri} (z) \, \Tht_{2,
\, L}
                                                + \sig^{\pri \pri} (h-z) \, \Tht_{2,
\, U}
                                            \Big\}
                                \quad
                                \text{in } \Omega \times (0, T),

				\\
                        \spacer \displaystyle
                                \Tht_2 = - u^0_2
                        \quad
                                \text{on } \Gamma  \times (0, T), \\
                        \Tht_2|_{t=0} = 0,
                \end{array}
        \right.
\end{equation}
where the second term on the right hand side of (\ref{e:Eq_Tht_2}) is again an
$e.s.t.$.


The needed bounds in $L^p$ on the corrector $\blds{\Tht}$  with an explicit
dependence
on $\ep$ are stated and proved separately in \cref{S:est_Tht_PPF} below.


Following a well-known approach \cite{Lions1973,VishikLyusternik1957}, we
define an approximation to the viscous solution combining the corrector with the
inviscid solution, and estimate the approximation error explicitly in terms of
$\ep$ in various norms.
We then set:
\begin{equation}\label{e:corrected velocity PPF}
        \blds{v}^\ep
                = \big(v^\ep_1 (z, t), \, v^\ep_2 (x, z, t), \, 0 \big)
                := \uu^\eps - \uu^0 -\blds{\Tht},
\end{equation}
and let $\blds\omega^\ep$ denote the associated vorticity,
\begin{equation}\label{e:omega_PPF}
        \blds{\omega}^\ep
                = \big(
                            \omega^\ep_1 (x,z,t), \, \omega^\ep_2 (z,t), \,
\omega^\ep_3 (x,z,t)
                    \big)
                :=
                    \curl \blds{v}^\ep
                =
                    \Big(
                            -\dfrac{\pa v^\ep_2}{\pa z}, \,
                            \dfrac{\pa v^\ep_1}{\pa z}, \,
                            \dfrac{\pa v^\ep_2}{\pa x}
                    \Big).
\end{equation}

With these definitions, we are ready to state and prove our main convergence
result for PCF.

\begin{theorem}\label{T:PPF}
        Under the assumptions (\ref{e:Data_reg_PPF}), we have:
        \begin{align}\label{e:conv_PPF}
            \begin{split}
                \| \blds{v}^\ep \|_{L^{\infty}(0, T; L^2(\Omega))}
                    + \ep^{\frac{1}{2}} \| \nabla \blds{v}^\ep \|_{L^{2}(0, T;
L^2(\Omega))}
                                &\leq \kap_T \, \ep,\\
                \| \omega^\ep_1 \|_{L^{\infty}(0, T; L^2(\Omega))}
                    + \ep^{\frac{1}{2}} \| \nabla\omega^\ep_1 \|_{L^{2}(0, T;
L^2(\Omega))}
                                &\leq \kap_T \, \ep^{\frac{1}{4}}, \\
                \| \omega^\ep_2 \|_{L^{\infty}(0, T; L^2(\Omega))}
                    + \ep^{\frac{1}{2}} \| \nabla\omega^\ep_2 \|_{L^{2}(0, T;
L^2(\Omega))}
                                &\leq \kap_T \, \ep^{\frac{3}{4}}, \\
                \| \omega^\ep_3 \|_{L^{\infty}(0, T; L^2(\Omega))}
                    + \ep^{\frac{1}{2}} \| \nabla\omega^\ep_3 \|_{L^{2}(0, T;
L^2(\Omega))}
                                &\leq \kap_T \, \ep, \\
                \norm{\blds{\omega}^\eps}_{L^\iny(0, T; L^2(\Omega))}
                    + \eps^{\frac{1}{2}} \norm{\grad \blds{\omega}^\epsilon}
                        _{L^2(0, T; L^2(\Omega))}
                    &\le \kap_T \eps^{\frac{1}{4}}, \\
                \norm{\vv^\eps}_{L^\iny(0, T; H^1(\Omega))}
                    + \eps^{\frac{1}{2}} \norm{\vv^\epsilon}
                        _{L^2(0, T; H^2(\Omega))}
                    &\le \kap_T \eps^{\frac{1}{4}}.
            \end{split}
        \end{align}
In particular, the vanishing viscosity limit holds with convergence rate:
\begin{equation}\label{e:VVL_PPF}
        \| \uu^\eps - \uu^0 \|_{L^{\infty}(0, T; L^2(\Omega))}
                \leq \kap_T \, \ep^{\frac{1}{4}}.
\end{equation}
In addition,
\begin{align}\label{e:conv_PPF_measure}
    \curl \uu^\eps \to \curl \uu^0 + (\uu_0 \times \n) \mu
        \quad \weak^* \text{ in } L^\iny(0, T; \Cal{M}(\ol{\Omega})),
\end{align}
where $\Cal{M}(\ol{\Omega})$ is the space of Radon measures on $\ol{\Omega}$ and
$\mu$ is a measure supported on $\Gamma$, on which $\mu|_\Gamma$ agrees with
normalized  surface area.
\end{theorem}

Since the bounds on the corrector in  \ref{S:est_Tht_PPF} allow to estimate its
size, by the triangle inequality it follows that the rates of convergence in
Theorem \ref{T:PPF} are optimal in viscosity. We postpone the proof of the
theorem until Section \ref{S:proof_PPF}.

\subsection{Estimates of the corrector \texorpdfstring{$\blds{\Tht}$}{}
}\label{S:est_Tht_PPF}

In this subsection, we derive bounds in Lebesgue and Sobolev spaces for the
corrector, $\blds{\Tht} = (\Tht_1(z, t), \, \Tht_2(x,z,t), \, 0)$,
using the estimates obtained in Appendix \ref{S.App}. These bounds, in turn,
will be crucial in establishing the convergence rates of Theorem \ref{T:PPF}.

We begin with bounds for the first component of the corrector $\Tht_1$, defined
in (\ref{e:Tht_1}), which satisfies (\ref{e:Eq_Tht_1}). Then, exploiting the
pointwise estimates in Lemma \ref{L:pointwise}, satisfied by  $\Tht_{1, L}$ and
$\Tht_{1, R}$ with $\eta$  replaced by $z$ and $h-z$ respectively, gives:
\begin{equation}\label{e:RHS_Eq_Tht_PPF}
        (\text{right-hand side of } (\ref{e:Eq_Tht_1}))
            =   e.s.t..
\end{equation}
Therefore, we can apply the $L^p$ estimates in Lemma \ref{L:L2} with $\eta$
replaced by $z$ and $h-z$ respectively,  and obtain:
\begin{equation}\label{e:Lp_est_Tht_1}
\Big\| \dfrac{\pa^{m} \Tht_1}{\pa z^m} \Big\|_{L^p(\Omega)}
                                            \leq
                                                    \kap_T \, (1 +
t^{\frac{1}{2p} - \frac{m}{2}}) \, \ep^{\frac{1}{2p} - \frac{m}{2}},
                                            \quad 1 \leq p \leq \infty,
                                            \quad 0 \leq m \leq 2,
                                            \quad 0 < t < T.
\end{equation}

Similarly, for the second component of the corrector $\Tht_2$, defined in
(\ref{e:Tht_2}), we can employ the estimates in Lemma
\ref{l:Lp_est_heat_like_sol}
on $\Tht_{2, \, L}$ or $\Tht_{2, \, R}$, as they both satisfy the parabolic
system (\ref{e:HEAT_like_eq})  with
$\tau$ and $\eta$  replaced by $x$ and $z$ (or $h-z$) in the domain $(0, L)
\times (0, \infty)$ (or $(0, L) \times (-\infty, h)$) respectively, and with
\begin{equation*}
	\left\{
		\begin{array}{l}
			\spacer
				\mathcal{U}
					= \Tht_1 + u^0_1,\\
			\spacer
				G
					= - \Tht_1 \dfrac{\pa u^0_2}{\pa x},\\
				g = - u^0_2 |_{z = 0 \text{ (or } h)},
		\end{array}
	\right.
\end{equation*}
in the notation of the Appendix.
Consequently, it holds:
\begin{equation}\label{e:Lp_est_Tht_2}
\left\{
\begin{array}{l}
	\spacer
	\Big\| \dfrac{\pa^{k} \Tht_2}{\pa x^k} \Big\|_{L^{\infty}(0, T;
L^p(\Omega))}
	+
	\ep^{\frac{1}{2p} + \frac{1}{4}}
	\Big\|
		\nabla \dfrac{\pa^{k} \Tht_2}{\pa x^k}
	\Big\|_{L^{2}(0, T; L^2(\Omega))}
                                                                \leq
                                                                        \kap_T
\ep^{\frac{1}{2p}},
                                                                        \quad
                                                                        1 \leq p
\leq \infty,\\
         \spacer
	\Big\| \dfrac{\pa^{k+1} \Tht_2}{\pa x^k \pa z} \Big\|_{L^{\infty}(0, T;
L^2(\Omega))}
	+
	\ep^{\frac{1}{2}}
	\Big\|
		\nabla \dfrac{\pa^{k+1} \Tht_2}{\pa x^k \pa z}
	\Big\|_{L^{2}(0, T; L^2(\Omega))}
                                                                \leq
                                                                        \kap_T
\ep^{-\frac{1}{4}},\\
	\Big\| \dfrac{\pa^{k+1} \Tht_2}{\pa x^k \pa z} \Big\|_{L^{\infty}(0, T;
L^1(\Omega))}
	\leq \kap_T,
\end{array}
\right.
\end{equation}
with $k \geq 0$.
Above, $\kap_T$ is a constant  depending on $T$ and the other data, but not  on
$\ep$.
Note that the continuity of $L^p$ norm in a bounded interval is used for the $L^1$ bound of $\pa^k \Tht_2/ \pa x^k$.
Moreover,  one can verify that
\begin{equation}\label{e:RHS_SIMPLE_FACT}
	(\text{the second term on the right hand side of } (\ref{e:Eq_Tht_2}))
            =   e.s.t.,
\end{equation}
by performing an energy estimate on $\Psi$ (or its derivative) in
(\ref{e:HEAT_like_eq}) multiplied by any derivative of the cut-off function in
(\ref{e:sig_LR}). We omit the proof of this last estimate for the sake of
exposition.

In particular, it follows from (\ref{e:Lp_est_Tht_1}) and (\ref{e:Lp_est_Tht_2})
that
\begin{align}\label{e:curlThtBound}
    \norm{\curl \blds{\Tht}}_{L^\iny(0, T; L^1(\Omega))}
        \le \kappa_T,
\end{align}
owing to the particular form of the curl of a victor field under plane-parallel
symmetry (cf. \cref{e:omega_PPF}).

\subsection{Proof of Theorem \ref{T:PPF}}\label{S:proof_PPF}

In this subsection, we sketch the proof of Theorem \ref{T:PPF}, using the
estimates for the corrector.

To prove (\ref{e:conv_PPF})$_1$, we will perform energy estimates on the
approximation remainder $\blds{v}^\ep$. Throughout,  we will use standard
inequalities (e.g. H\"older, Cauchy-Schwarz,
Young) and the fact that $\uu^0$
is a strong solution of the Euler equation under the assumptions
\eqref{e:Data_reg_PPF} on the data.

We first derive the initial-boundary problem for $\blds{v}^\ep$ from
the EE and the corresponding equations for the corrector:
\begin{equation}\label{e:Eq_v_PPF}
        \left\{
                \begin{array}{rl}
                        \spacer \displaystyle
                                \dfrac{\pa v^\ep_1}{\pa t} - \ep \dfrac{\pa^2
v^\ep_1}{\pa z^2}
                                =
                                \ep \dfrac{\pa^2 u^0_1}{\pa z^2}
                                +
                                e.s.t.,
                        & \displaystyle
                                \text{in } \Omega \times (0, T),\\
                        \spacer \displaystyle
                                \dfrac{\pa v^\ep_2}{\pa t} - \ep \Delta v^\ep_2
                                +
                                u^\eps_1 \dfrac{\pa v^\ep_2}{\pa x}
                                + v^\ep_1 \dfrac{\pa \Tht_2}{\pa x}
                                + v^\ep_1 \dfrac{\pa u^0_2}{\pa x}
                                = \ep 
                                                \Delta u^0_2 
                                   +
                                   e.s.t.,
                        & \displaystyle
                                \text{in } \Omega \times (0, T),\\
                        \spacer \displaystyle
                                \blds{v}^\ep = 0,
                        &
                                \text{on } \Gamma, \\
                        \blds{v}^\ep|_{t=0} = 0.
                \end{array}
        \right.
\end{equation}

We multiply (\ref{e:Eq_v_PPF})$_1$ by $v^\ep_1$ and integrate over $\Omega$ to
establish the following estimate:
\begin{equation*}\label{e:engergy_v1_PPL}
        \dfrac{1}{2}\dfrac{d}{dt} \| v^\ep_1 \|_{L^2(\Omega)}^2
                            + \ep \Big\| \dfrac{\pa v^\ep_1}{\pa z}
\Big\|_{L^2(\Omega)}^2
                                    \leq \kap \ep^2 + \kap \| v^\ep_1
\|_{L^2(\Omega)}^2.
\end{equation*}
Thanks to  Gr\"onwall's inequality, we then obtain:
\begin{equation}\label{e:engergy_conv_v1_PPL}
        \| v^\ep_1 \|_{L^{\infty}(0, T; L^2(\Omega))}
        + \ep^{\frac{1}{2}}
            \Big\| \dfrac{\pa v^\ep_1}{\pa z} \Big\|_{L^{2}(0, T; L^2(\Omega))}
        \leq
            \kap_T \ep.
\end{equation}

For the second component of $\blds{v}^\ep$, $v^\ep_2$, we multiply
(\ref{e:Eq_v_PPF})$_2$ by $v^\ep_2$ and integrate over $\Omega$.
Then, using (\ref{e:Lp_est_Tht_2}), the regularity of $\uu^0$, and
(\ref{e:engergy_conv_v1_PPL}) as well, we find that
\begin{equation}\label{e:engergy_v2_PPL-1}
    \begin{array}{rl}
        \spacer
                \dfrac{1}{2}\dfrac{d}{dt} \| v^\ep_2 \|_{L^2(\Omega)}^2
                + \ep \| \nabla v^\ep_2 \|_{L^2(\Omega)}^2
                        & \hspace{-2mm}
                        \leq
                                    \kap \ep^2
                                    + \kap \| v^\ep_2 \|_{L^2(\Omega)}^2
                                    + \Big\|
                                                v^\ep_1 \Big( \dfrac{\pa
\Tht_2}{\pa x}
                                                                            +
\dfrac{\pa u^0_2}{\pa x}
                                                                    \Big)
v^\ep_2
                                        \Big\|_{L^1(\Omega)}\\
                        \spacer
                        & \hspace{-2mm}
                        \leq
                                    \kap \ep^2
                                    + \kap \| v^\ep_2 \|_{L^2(\Omega)}^2
                                    + \| v^\ep_1 \|_{L^2(\Omega)}
                                        \Big\|
                                                \dfrac{\pa \Tht_2}{\pa x}
                                                                            +
\dfrac{\pa u^0_2}{\pa x}
                                        \Big\|_{L^{\infty}(\Omega)}
                                        \| v^\ep_2 \|_{L^2(\Omega)}\\
                        & \hspace{-2mm}
                        \leq
                                    \kap_T \ep^2
                                    + \kap_T \| v^\ep_2 \|_{L^2(\Omega)}^2.
    \end{array}
\end{equation}
Above, we have used that, after integrating by parts,  the third term on the
left-hand side of (\ref{e:Eq_v_PPF})$_2$ integrates to zero, because $u^\eps_1$
does not depend on the variable $x$.

Using Gr\"onwall's inequality again, we obtain from (\ref{e:engergy_v2_PPL-1})
that
\begin{equation}\label{e:engergy_conv_v2_PPL}
        \| v^\ep_2 \|_{L^{\infty}(0, T; L^2(\Omega))}
        + \ep^{\frac{1}{2}}
            \| \nabla v^\ep_2 \|_{L^{2}(0, T; L^2(\Omega))}
        \leq
            \kap_T \ep.
\end{equation}
By combining the estimates above (\ref{e:conv_PPF})$_1$ follows.
Given the bounds on the corrector,   (\ref{e:VVL_PPF}) is a direct
consequence of (\ref{e:conv_PPF})$_{1}$.

To establish error bounds on the vorticity,  (\ref{e:conv_PPF})$_{2, 3, 4}$,
 we derived an initial-boundary value problem for
$\blds{\omega}^\ep$ and use energy estimates. In doing so, we are confronted
with the problem of deriving a usable boundary condition for
the vorticity. We follow here the approach of  Lighthill.
First, using the expression for $\blds{\omega}^\ep$ given in
(\ref{e:omega_PPF}), we derive the equations satisfied  by $\blds{\omega}^\ep$
from those for $\blds{v}^\ep$, by differentiating (\ref{e:Eq_v_PPF})$_2$ in $x$
and $z$, and (\ref{e:Eq_v_PPF})$_1$ in $z$.:
\begin{equation}\label{e:Eq_vorticity PPF}
        \left\{
                \begin{array}{l}
                        \spacer \displaystyle
                                \dfrac{\pa \omega^\ep_1}{\pa t}
                                - \ep \Delta \omega^\ep_1
                                + u^\eps_1 \, \dfrac{\pa \omega^\ep_1}{\pa x}\\
                        \spacer \qquad
                                =
                                - \ep \dfrac{\pa (\Delta u^0_2)}{\pa z}
                                +
                                \omega^\ep_2 \, \dfrac{\pa (u^0_2 + \Tht_2)}{\pa
x}
                                +
                                v^\ep_1 \, \dfrac{\pa^2 (u^0_2 + \Tht_2)}{\pa x
\pa z}
                                +
                                \dfrac{\pa u^\eps_1}{\pa z} \, \omega^\ep_3
                                +
                                e.s.t.
                        \quad
                                \text{in } \Omega \times (0, T),\\
                        \spacer \displaystyle
                                \dfrac{\pa \omega^\ep_2}{\pa t}
                                - \ep \dfrac{\pa^2 \omega^\ep_2}{\pa z^2}
                                =
                                \ep \dfrac{\pa^3 u^0_1}{\pa z^3}
                                +
                                e.s.t.
                        \quad
                                \text{in } \Omega \times (0, T),\\
                        \displaystyle
                                \dfrac{\pa \omega^\ep_3}{\pa t}
                                - \ep \Delta \omega^\ep_3
                                +
                                u^\eps_1 \, \dfrac{\pa \omega^\ep_3}{\pa x}
                                =
                                \ep \dfrac{\pa (\Delta u^0_2)}{\pa x}
                                - v^\ep_1\, \dfrac{\pa^2 (u^0_2 + \Tht_2)}{\pa
x^2}
                                +
                                e.s.t.
                        \quad
                                \text{in } \Omega \times (0, T).
                \end{array}
        \right.
\end{equation}

Next, by restricting (\ref{e:Eq_v_PPF})$_{1,2}$ on $\Gamma$, and using
(\ref{e:Eq_v_PPF})$_{3, 4}$, we find the boundary and initial conditions for
$\blds{\omega}^\ep$:
\begin{equation}\label{e:BC_IC_vorticity PPF}
        \left\{
                \begin{array}{rl}
                        \spacer \displaystyle
                                \dfrac{\pa \omega^\ep_1}{\pa z}
                                =
                                \Delta u^0_2,
                        &
                                \text{on } \Gamma, \\
                        \spacer \displaystyle
                                \dfrac{\pa \omega^\ep_2}{\pa z}
                                =
                                - \dfrac{\pa^2 u^0_1}{\pa z^2},
                        &
                                \text{on } \Gamma, \\
                        \spacer \displaystyle
                                \omega^\ep_3 = 0,
                        &
                                \text{on } \Gamma, \\
                        \blds{\omega}^\ep|_{t=0} = 0.
                \end{array}
        \right.
\end{equation}
Above,  we have used the fact that the terms of order $e.s.t.$ in
(\ref{e:Eq_v_PPF}) vanishes on $\Gamma$ (the explicit expression of these terms
is given in (\ref{e:Eq_Tht_1}) and (\ref{e:Eq_Tht_2})).

A standard energy estimate, using that $u^\ep_1$ is independent of $x$, the
regularity and decay of the corrector $\blds{\Tht}$, the regularity of the EE
solution $\uu^0$, and the bounds on $\blds{v}^\ep$, gives that
\begin{equation*}\label{e:conv_w_3_PPL-1}
    \begin{array}{rl}
        \spacer
                \dfrac{1}{2}\dfrac{d}{dt} \| \omega^\ep_3 \|_{L^2(\Omega)}^2
                + \ep \| \nabla \omega^\ep_3 \|_{L^2(\Omega)}^2
                        & \hspace{-2mm}
                        \leq
                                    \kap \ep^2
                                    + \kap \| \omega^\ep_3 \|_{L^2(\Omega)}^2
                                    + \Big\|
                                                v^\ep_1 \Big( \dfrac{\pa^2
\Tht_2}{\pa x^2}
                                                                            +
\dfrac{\pa^2 u^0_2}{\pa x^2}
                                                                    \Big)
\omega^\ep_3
                                        \Big\|_{L^1(\Omega)}\\
                        \spacer
                        & \hspace{-2mm}
                        \leq
                                    \kap \ep^2
                                    + \kap \| \omega^\ep_3 \|_{L^2(\Omega)}^2
                                    + \| v^\ep_1 \|_{L^2(\Omega)}
                                        \Big\|
                                                \dfrac{\pa \Tht_2}{\pa x}
                                                                            +
\dfrac{\pa u^0_2}{\pa x}
                                        \Big\|_{L^{\infty}(\Omega)}
                                        \| \omega^\ep_3 \|_{L^2(\Omega)}\\
                        & \hspace{-2mm}
                        \leq
                                    \kap_T \ep^2
                                    + \kap_T \| \omega^\ep_3 \|_{L^2(\Omega)}^2.
    \end{array}
\end{equation*}
Estimate (\ref{e:conv_PPF})$_4$ now follows  from (\ref{e:conv_PPF})$_{1}$ by
applying Gr\"onwall's inequality again.

We proceed similarly for $\omega^\ep_2$ and obtain:
\begin{equation}\label{e:conv_w_2_PPL-1}
    \begin{array}{rl}
        \spacer
                \dfrac{1}{2}\dfrac{d}{dt} \| \omega^\ep_2 \|_{L^2(\Omega)}^2
                + \ep \Big\| \dfrac{\pa \omega^\ep_2}{\pa z}
\Big\|_{L^2(\Omega)}^2
                        & \hspace{-2mm}
                        \leq
                                    \kap \ep^2
                                    + \kap \| \omega^\ep_2 \|_{L^2(\Omega)}^2
                                    + \ep
                                        \Big\|
                                                \dfrac{\pa^2 u^0_1}{\pa z^2} \,
                                                \omega^\ep_2
                                        \Big\|_{L^1(\Gamma)}\\
                        & \hspace{-2mm}
                        \leq
                                    \kap \ep^2
                                    + \kap \| \omega^\ep_2 \|_{L^2(\Omega)}^2
                                    + \ep
                                        \Big\|
                                                \dfrac{\pa^2 u^0_1}{\pa z^2}
                                        \Big\|_{L^2(\Gamma)}
                                        \|
                                                \omega^\ep_2
                                        \|_{L^2(\Gamma)}.
    \end{array}
\end{equation}
To apply Gr\"onwall's inequality one more time, we need to estimate the third
term
on the right-hand side of (\ref{e:conv_w_2_PPL-1}). This estimate follows in
turn from the regularity of $\uu^0$, and the trace theorem, noticing
$\omega^\ep_2$ depends on $z$ alone:
\begin{equation}\label{e:conv_w_2_PPL-2}
    \begin{array}{rl}
        \ep
        \Big\|\dfrac{\pa^2 u^0_1}{\pa z^2}\Big\|_{L^2(\Gamma)}
        \|\omega^\ep_2\|_{L^2(\Gamma)}
                & \spacer \hspace{-2mm}
                        \leq \kap \ep
                                    \|\omega^\ep_2\|_{L^2(\Omega)}^{\frac{1}{2}}

\|\omega^\ep_2\|_{H^1(\Omega)}^{\frac{1}{2}}\\
                & \spacer \hspace{-2mm}
                        \leq \kap \ep
                                    \|\omega^\ep_2\|_{L^2(\Omega)}
                                    +
                                    \kap \ep
                                    \|\omega^\ep_2\|_{L^2(\Omega)}^{\frac{1}{2}}
                                    \Big\| \dfrac{\pa \omega^\ep_2}{\pa z}
\Big\|_{L^2(\Omega)}^{\frac{1}{2}}\\
                 & \spacer \hspace{-2mm}
                        \leq \kap \ep ^2
                                    +
                                    \kap \|\omega^\ep_2\|_{L^2(\Omega)}^2
                                    +
                                    \kap \ep^{\frac{3}{4}}
                                    \|\omega^\ep_2\|_{L^2(\Omega)}
                                    +
                                    \kap \ep^{\frac{5}{4}}
                                    \Big\| \dfrac{\pa \omega^\ep_2}{\pa z}
\Big\|_{L^2(\Omega)}\\
                 & \hspace{-2mm}
                        \leq \kap \ep ^{\frac{3}{2}}
                                    +
                                    \kap \|\omega^\ep_2\|_{L^2(\Omega)}^2
                                    +
                                    \dfrac{1}{2}
                                    \ep
                                    \Big\| \dfrac{\pa \omega^\ep_2}{\pa z}
\Big\|_{L^2(\Omega)}^2.
    \end{array}
\end{equation}
Combining (\ref{e:conv_w_2_PPL-1}) and (\ref{e:conv_w_2_PPL-2}), we obtain
\begin{equation}\label{e:conv_w_2_PPL-3}
                \dfrac{d}{dt} \| \omega^\ep_2 \|_{L^2(\Omega)}^2
                + \ep \Big\| \dfrac{\pa \omega^\ep_2}{\pa z}
\Big\|_{L^2(\Omega)}^2
                        \leq
                                    \kap \ep ^{\frac{3}{2}}
                                    +
                                    \kap \|\omega^\ep_2\|_{L^2(\Omega)}^2.
\end{equation}
By Gr\"onwall's inequality again,  (\ref{e:conv_PPF})$_3$ then follows from
(\ref{e:conv_w_2_PPL-3}).

Again, similarly an energy estimate gives an {\em a priori} bound for
$\omega^\ep_1$:
\begin{equation}\label{e:conv_w_1_PPL-1}
    \begin{array}{rl}
        \spacer
                \dfrac{1}{2}\dfrac{d}{dt} \| \omega^\ep_1 \|_{L^2(\Omega)}^2
                + \ep \| \nabla \omega^\ep_1 \|_{L^2(\Omega)}^2
                        \leq
                        & \hspace{-2mm}
                                    \kap \ep^2
                                    + \kap \| \omega^\ep_1 \|_{L^2(\Omega)}^2
                                    + \ep
                                        \|
                                                \Delta u^0_2
                                        \|_{L^2(\Gamma)}
                                        \|
                                                \omega^\ep_1
                                        \|_{L^2(\Gamma)}\\
                        & \hspace{-2mm} \displaystyle \spacer
                                    + \Big\|
                                            \omega^\ep_2 \,
                                            \dfrac{\pa (u^0_2 + \Tht_2)}{\pa x}
\,
                                            \omega^\ep_1
                                        \Big\|_{L^1(\Omega)}
                                    + \Big|
                                        \int_{\Omega} v^\ep_1 \, \dfrac{\pa^2
(u^0_2 + \Tht_2)}{\pa x \pa z} \, \omega^\ep_1
                                    \Big|\\
                        & \hspace{-2mm} \displaystyle
                                    + \Big\|
                                        \dfrac{\pa u^\eps_1}{\pa z}
\,\omega^\ep_3 \, \omega^\ep_1
                                    \Big\|_{L^1(\Omega)}
                                    .
    \end{array}
\end{equation}
We need to bound the last three terms on the right-hand side of the above
expression.

 We can estimate the third term on the right-hand side,
by using again the regularity of $\uu^0$, and the
trace theorem:
\begin{equation}\label{e:conv_w_1_PPL-2}
    \begin{array}{rl}
        \ep
        \|\Delta u^0_2\|_{L^2(\Gamma)}
        \|\omega^\ep_1\|_{L^2(\Gamma)}
                & \spacer \hspace{-2mm}
                        \leq \kap \ep
                                    \|\omega^\ep_1\|_{L^2(\Omega)}^{\frac{1}{2}}

\|\omega^\ep_1\|_{H^1(\Omega)}^{\frac{1}{2}}\\
                & \spacer \hspace{-2mm}
                        \leq \kap \ep
                                    \|\omega^\ep_1\|_{L^2(\Omega)}
                                    +
                                    \kap \ep
                                    \|\omega^\ep_1\|_{L^2(\Omega)}^{\frac{1}{2}}
                                    \| \nabla \omega^\ep_1
\|_{L^2(\Omega)}^{\frac{1}{2}}\\
                 & \spacer \hspace{-2mm}
                        \leq \kap \ep ^2
                                    +
                                    \kap \|\omega^\ep_1\|_{L^2(\Omega)}^2
                                    +
                                    \kap \ep^{\frac{3}{4}}
                                    \|\omega^\ep_1\|_{L^2(\Omega)}
                                    +
                                    \kap \ep^{\frac{5}{4}}
                                    \| \nabla \omega^\ep_1 \|_{L^2(\Omega)}\\
                 & \hspace{-2mm}
                        \leq \kap \ep ^{\frac{3}{2}}
                                    +
                                    \kap \|\omega^\ep_1\|_{L^2(\Omega)}^2
                                    +
                                    \dfrac{1}{4}
                                    \ep
                                    \| \nabla \omega^\ep_1 \|_{L^2(\Omega)}^2.
    \end{array}
\end{equation}

Using the regularity of $\uu^0$, and the bounds for the corrector and
$\omega^\ep_2$,
we estimate the fourth term
on the right-hand side of (\ref{e:conv_w_1_PPL-1}) as follows:
\begin{equation}\label{e:conv_w_1_PPL-3}
    \begin{array}{rl}
            \Big\|
                \omega^\ep_2 \,
                \dfrac{\pa (u^0_2 + \Tht_2)}{\pa x} \,
                \omega^\ep_1
            \Big\|_{L^1(\Omega)}
                    & \spacer \hspace{-2mm}
                        \leq
                                \|
                                    \omega^\ep_2
                                \|_{L^2(\Omega)}
                                \Big\|
                                    \dfrac{\pa (u^0_2 + \Tht_2)}{\pa x}
                                \Big\|_{L^{\infty}(\Omega)}
                                \|
                                    \omega^\ep_1
                                \|_{L^2(\Omega)}\\
                    & \spacer \hspace{-2mm}
                        \leq
                                \kap_T \ep^{\frac{3}{4}}
                                \|
                                    \omega^\ep_1
                                \|_{L^2(\Omega)}\\
                    &  \hspace{-2mm}
                        \leq
                                \kap_T \ep^{\frac{3}{2}}
                                +
                                \kap_T\|
                                    \omega^\ep_1
                                \|_{L^2(\Omega)}^2.
    \end{array}
\end{equation}

Next, we tackle the fifth term on the right-hand side of
(\ref{e:conv_w_1_PPL-1}).
Since $v^\ep_1 = 0$ on $\Gamma$, by integrating by parts we can
write
\begin{equation}\label{e:conv_w_1_PPL-4}
            \Big|
            \int_{\Omega} v^\ep_1 \, \dfrac{\pa^2 (u^0_2 + \Tht_2)}{\pa x \pa z}
\, \omega^\ep_1 \, d \blds{x}
            \Big|
                    \leq
            \Big\|
                    \dfrac{\pa v^\ep_1}{\pa z} \,
                    \dfrac{\pa (u^0_2 + \Tht_2)}{\pa x} \,
                    \omega^\ep_1
            \Big\|_{L^1(\Omega)}
            +
            \Big\|
                    v^\ep_1 \,
                    \dfrac{\pa (u^0_2 + \Tht_2)}{\pa x} \,
                    \dfrac{\pa \omega^\ep_1}{\pa z}
            \Big\|_{L^1(\Omega)}.
\end{equation}
Then, we can proceed as for the fourth term, and obtain:
\begin{equation}\label{e:conv_w_1_PPL-5}
    \begin{array}{l}
            \spacer \displaystyle
            \Big|
            \int_{\Omega} v^\ep_1 \, \dfrac{\pa^2 (u^0_2 + \Tht_2)}{\pa x \pa z}
\, \omega^\ep_1 \, d \blds{x}
            \Big|\\
                    \spacer
                    \leq
                            \Big\|
                                \dfrac{\pa v^\ep_1}{\pa z}
                            \Big\|_{L^2(\Omega)}
                            \Big\|
                                \dfrac{\pa (u^0_2 + \Tht_2)}{\pa x}
                            \Big\|_{L^{\infty}(\Omega)}
                            \|
                                \omega^\ep_1
                            \|_{L^2(\Omega)}
                            +
                            \|
                                v^\ep_1
                            \|_{L^2(\Omega)}
                            \Big\|
                                \dfrac{\pa (u^0_2 + \Tht_2)}{\pa x}
                            \Big\|_{L^{\infty}(\Omega)}
                            \Big\|
                                \dfrac{\pa \omega^\ep_1}{\pa z}
                            \Big\|_{L^2(\Omega)}
                            \\
                    \spacer
                    \leq
                            \kap
                            \Big\|
                                \dfrac{\pa v^\ep_1}{\pa z}
                            \Big\|_{L^2(\Omega)}
                            \|
                                \omega^\ep_1
                            \|_{L^2(\Omega)}
                            +
                            \kap_T
                            \ep
                            \Big\|
                                \dfrac{\pa \omega^\ep_1}{\pa z}
                            \Big\|_{L^2(\Omega)}
                            \\
                    \leq
                            \kap
                            \Big\|
                                \dfrac{\pa v^\ep_1}{\pa z}
                            \Big\|_{L^2(\Omega)}^2
                            +
                            \kap
                            \|
                                \omega^\ep_1
                            \|_{L^2(\Omega)}^2
                            +
                            \kap_T \ep
                            +
                            \dfrac{1}{4}
                            \ep
                            \|
                                \nabla \omega^\ep_1
                            \|^2_{L^2(\Omega)}.
    \end{array}
\end{equation}

Finally, we estimate the last term on the right-hand side of
(\ref{e:conv_w_1_PPL-1}) as
follows, once again utilizing the bounds for $\omega^\ep_3$ and for the
corrector $\Theta_1$:
\begin{equation}\label{e:extra_PPF 1}
    \begin{array}{rl}
        \Big\|
            \dfrac{\pa u^\eps_1}{\pa z} \,\omega^\ep_3 \, \omega^\ep_1
        \Big\|_{L^1(\Omega)}
            & \spacer \hspace{-2mm}
                \leq
                    \|
                         \omega^\ep_2 \,\omega^\ep_3 \, \omega^\ep_1
                    \|_{L^1(\Omega)}
                    +
                    \Big\|
                        \dfrac{\pa u^0_1}{\pa z} \,\omega^\ep_3 \, \omega^\ep_1
                    \Big\|_{L^1(\Omega)}
                    +
                    \Big\|
                        \dfrac{\pa \Tht_1}{\pa z} \,\omega^\ep_3 \, \omega^\ep_1
                    \Big\|_{L^1(\Omega)}\\
            & \spacer \hspace{-2mm}
                \leq
                    \| \omega^\ep_2 \|_{L^{\infty}(\Omega)}
                    \| \omega^\ep_3 \|_{L^2(\Omega)}
                    \| \omega^\ep_1 \|_{L^2(\Omega)}
                    +
                    \Big\| \dfrac{\pa u^0_1}{\pa z} \Big\|_{L^{\infty}(\Omega)}
                    \| \omega^\ep_3 \|_{L^2(\Omega)}
                    \| \omega^\ep_1 \|_{L^2(\Omega)}\\
            & \spacer \quad
                    +
                    \Big\| \dfrac{\pa \Tht_1}{\pa z} \Big\|_{L^{\infty}(\Omega)}
                    \| \omega^\ep_3 \|_{L^2(\Omega)}
                    \| \omega^\ep_1 \|_{L^2(\Omega)}\\
            & \spacer \hspace{-2mm}
                \leq
                    \kap_T \big(
                                    \ep
                                    \| \omega^\ep_2 \|_{L^{\infty}(\Omega)}
                                    +
                                    \ep
                                    +
                                    (1 + t^{- \frac{1}{2}}) \ep^{\frac{1}{2}}
                                \big)
                                \| \omega^\ep_1 \|_{L^2(\Omega)}\\
            & \spacer \hspace{-2mm}
                \leq
                    \kap_T \ep \| \omega^\ep_2 \|_{L^{\infty}(\Omega)}^2
                    +
                    \kap_T (1 + t^{- \frac{1}{2}}) \ep
                    +
                    \kap_T (1 + t^{- \frac{1}{2}}) \| \omega^\ep_1
\|_{L^2(\Omega)}^2.
    \end{array}
\end{equation}

Combining all the estimates above into (\ref{e:conv_w_1_PPL-1}),
 we deduce that
\begin{equation}\label{e:conv_w_1_PPL-6}
\begin{array}{l}
            \spacer
                \dfrac{1}{2}\dfrac{d}{dt} \| \omega^\ep_1 \|_{L^2(\Omega)}^2
                + \ep \| \nabla \omega^\ep_1 \|_{L^2(\Omega)}^2\\
                \qquad \quad
                        \leq
                                    \kap_T \ep \, \big(1 + t^{-\frac{1}{2}} +
\| \omega^\ep_2 \|_{L^{\infty}(\Omega)}^2\big)
                                    +
                                    \kap
                                    \Big\|
                                        \dfrac{\pa v^\ep_1}{\pa z}
                                    \Big\|_{L^2(\Omega)}^2
                                    +
                                    \kap_T (1 + t^{-\frac{1}{2}}) \|
\omega^\ep_1 \|_{L^2(\Omega)}^2.
\end{array}
\end{equation}
We need to estimate the
$L^\infty$ norm of $\omega^\ep_2$. This estimate, in turn, follows by applying
the one-dimensional Agmon's inequality and the $L^2$ bounds on $\omega^\ep_2$,
observing that  $\omega^\ep_2$ is a function of $z$ only, which
represents the distance to the boundary:
\begin{equation}\label{e:extra_PPF 2}
            \int_0^T \| \omega^\ep_2 \|_{L^{\infty}(\Omega)}^2 \, dt
            \leq
            \kap
            \int_0^T \| \omega^\ep_2 \|_{L^{2}(\Omega)}
                            \Big\| \dfrac{ \pa \omega^\ep_2}{\pa z}
\Big\|_{L^{2}(\Omega)} \, dt
            \leq
            \kap_T \ep.
\end{equation}
Estimate (\ref{e:conv_PPF})$_2$ now follows
by an applycation of Gr\"onwall's inequality one more time, with the integrating
factor $\exp(-\kap_T t - 2 \kap_T t^{1/2})$.

Since the approximate NS solution $\blds{v}^\ep$ vanishes on the boundary of
$\Omega$, $\Gamma$, we have that:
\[
      \|\blds{u}^\ep\|_{H^k} \leq C_k \, \|\blds{\omega}^\ep\|_{H^{k-1}},
\qquad k\geq 0.
\]
Therefore, estimate \cref{e:conv_PPF}$_6$ on the velocity follows from the
corresponding bounds on the vorticity.

Finally, \cref{e:conv_PPF}$_5$ and \cref{e:curlThtBound}  imply  that
$\norm{\curl(\uu^\eps - \uu^0)}_{L^\iny(0, T; L^1(\Omega))} \le \kappa_T$.
Weak convergence of the vorticity in \cref{e:conv_PPF_measure} can then be
established using
\cref{C:ConvKelMeasure} in the Appendix.

$\Box$

\section{Parallel pipe flows}\label{S:IPF}


In this section, we turn to the most complex class of symmetric flows, the
so-called {\em parallel pipe flows} in an infinite, straight pipe.

Since our focus is on the behavior of the flow near walls,  we consider the case
of a pipe with annular cross section:
$$
\Omega := \{ (x, y, z) \in \mathbb{R}^3 | \, r_L < y^2 + z^2 < r_R \},
\quad
0 < r_L <  r_R.
$$
The case of a pipe with circular cross section is technically more difficult,
because the estimates in cylindrical coordinates are weaker at the pipe axis.
It can be treated, as in \cite{HMNW11}, by a two-step localization, near the
boundary utilizing cylindrical coordinates, near the axis using Cartesian
coordinates, or by applying a suitable form of Hardy's inequality,
but we will not address this case here.
We also periodize the channel in the direction of its axis, and denote the
period by
$L$. As in the case of the channel geometry, together with symmetry, periodicity
ensures uniqueness for the solutions of
the fluid equations, in particular by ensuring that the only steady
pressure-driven flow is the trivial flow (see \cite{MT11} for a more detailed
discussion of this point).

We introduce cylindrical coordinates, $\blds{\xi} = (\phi, x, r)$, with
associated orthonormal frame $\{\e_\phi, \e_x, \e_r\}$, so that the domain
$\Omega$ corresponds in these coordinates to the parallepiped
\begin{equation*}\label{e:domina in cyl_co}
    \tilde\Omega
            = \{ \blds{\xi} = (\phi,x,r) | \, 0
\leq \phi < 2\pi, \,0<x<L, \, r_L \leq r < r_R \},
\end{equation*}
hence the notation of $r_L$ and $r_R$ for the inner and outer radii of the
pipe. We employ this notation to emphasize comparison with the channel case.

We will call any three-dimensional flow a {\em parallel pipe flow} (PPF) if the
velocity field is a vector field of the form:
\begin{equation}\label{e:IPF_general}
        {\bf F} =
                        F_{\phi} (r) \e_{\phi}
                        +
                        F_x (\phi, r) \e_x.
\end{equation}
In the special case of a planar PPF, i.e., when $F_x\equiv 0$, the flow reduces
to a linear, 2D circularly symmetric flow in any cross-sections of the pipe.

Standard calculus equalities (see e.g \cite{Bat99}) imply that
the gradient of each component of ${\bf F}$, $F_{\phi}$ or $F_x$, can be written
as:\begin{equation}\label{e:IPF_grad_comp}
        \nabla F_{\phi}
            = \dfrac{\pa F_{\phi}}{\pa r} \e_r,
\qquad
        \nabla F_x
            = \dfrac{1}{r} \dfrac{\pa F_x}{\pa \phi} \e_{\phi}
                +
                \dfrac{\pa F_x}{\pa r} \e_r.
\end{equation}
The derivative of ${\bf F}$ in the
direction of the vector
$\n = n_{\phi} \e_{\phi} + n_x \e_x + n_r \e_r$ is then given by:
\begin{equation}\label{e:IPF_grad}
\n \cdot \nabla {\bf F}
        =
            \big(
                \n \cdot \nabla F_{\phi}
            \big) \e_{\phi}
            +
            \big(
                \n \cdot \nabla F_x
            \big) \e_x
            +
            \Big(
                - \dfrac{1}{r} n_{\phi} F_{\phi}
            \Big) \e_r.
\end{equation}
It is easy to verify that the divergence of ${\bf F}$ is zero:
\begin{equation}\label{e:divergence_IPF}
        \dv {\bf F} = 0,
\end{equation}
and that the curl of ${\bf F}$ takes the form:
\begin{equation}\label{e:curl_IPF}
        \curl {\bf F}
                =
                    \Big(
                            - \dfrac{\pa F_x}{\pa r}
                    \Big) \e_{\phi}
                    +
                    \Big(
                            \dfrac{1}{r} \dfrac{\pa (r F_{\phi})}{\pa r}
                    \Big) \e_x
                    +
                    \Big(
                            \dfrac{1}{r} \dfrac{\pa F_x}{\pa \phi}
                    \Big) \e_r.
\end{equation}

It is also straightforward, though tedious, to verify that
the (vector) Laplacian of ${\bf F}$ can be written as:
\begin{equation}\label{e:Laplacian_IPF}
\Delta {\bf F}
        =
            \Big(
                    \Delta F_{\phi}
                    -
                    \dfrac{1}{r^2} F_{\phi}
            \Big) \e_{\phi}
            +
            \big(
                    \Delta F_x
            \big) \e_{x},
\end{equation}
where Laplacian of each
component, $F_{\phi}$ or $F_x$, is given by:
\begin{equation}\label{e:Laplacian_IPF_comp}
        \Delta F_{\phi}
                = \dfrac{1}{r}
                    \dfrac{\pa}{\pa r}\Big(
                                                    r \dfrac{\pa F_{\phi}}{\pa
r}
                                                \Big),
\qquad
        \Delta F_x
                = \dfrac{1}{r}
                    \dfrac{\pa}{\pa r}\Big(
                                                    r \dfrac{\pa F_x}{\pa r}
                                                \Big)
                    +
                    \dfrac{1}{r^2} \dfrac{\pa^2 F_x}{\pa \phi^2}.
\end{equation}
These formulas will be used in deriving the equations satisfied by the NS and
EE velocity fields and vorticites.

As in \cite{HMNW11,MT11}, we  will consider solutions of NSE and EE under this
symmetry; that is, we assume the following form for the velocities, $\uu^\ep$
and $\uu^0$,  and for the associated pressures, $p^\ep$ and $p^0$, respectively:
\begin{subequations}
\begin{equation}\label{e:C velocity_NSE and p}
\uu^\eps =
                                    u^\eps_{\phi} (r, t) \e_{\phi}
                                    + u^\eps_x (\phi, r, t) \e_x,
\qquad
p^\ep = p^\ep(r, t),
\end{equation}
\begin{equation}\label{e:C velocity_EE and P}
\uu^0 =
                                    u^0_{\phi} (r, t) \e_{\phi}
                                    + u^0_x (\phi, r, t) \e_x,\\
\qquad
p^0 = p^0(r, t).
\end{equation}
\end{subequations}
We remark that any PPF is automatically incompressible,
and that the components of its velocity field in the cross-section of the pipe
have circular symmetry, and evolve independently from the third.

It can be shown (see \cite{HMNW11,MT11} for a proof), that  the
parallel pipe symmetry is preserved under both the NSE and EE evolution,
provided that the data, the forcing $\f$ and initial velocity $\uu^0$,
are regular enough and have the same symmetry,
i.\,e.,
\begin{equation*}\label{e:C data}
            \f = f_{\phi} (r, t) \e_{\phi}
                                    + f_x (\phi, r, t) \e_x,
\qquad
            \uu_0 = u_{0, \, \phi} (r) \e_{\phi}
                                    + u_{0, \, x} (\phi, r) \e_x.
\end{equation*}
Then,
the NSE (\ref{e:NSE}) become the weakly coupled system:
\begin{equation}\label{e:NSE_C}
\left\{\begin{array}{rl}
                            \spacer
                                        \dfrac{\pa u^\eps_{\phi}}{\pa t}
                                        - \ep \Delta u^\eps_{\phi}
                                        + \ep \dfrac{1}{r^2} u^\eps_{\phi}
                                        = f_{\phi},
                                        &
                                        \text{in } \Omega \times(0,T),\\
                             \spacer
                                        \dfrac{\pa u^\eps_{x}}{\pa t}
                                        - \ep \Delta u^\eps_x
                                        + \dfrac{1}{r} u^\eps_{\phi} \dfrac{\pa
u^\eps_x}{\pa \phi}
                                        = f_{x},
                                        &
                                        \text{in } \Omega \times(0,T),\\
                             \spacer
                                        - \dfrac{1}{r} (u^\eps_{\phi})^2
                                        + \dfrac{\pa p^\ep}{\pa r}
                                        = 0,
                                        &
                                        \text{in } \Omega \times(0,T),\\
                              \spacer
                                        u^\eps_x \text{ is periodic}
                                        &
                                        \text{in $\phi$ with period $2 \pi$},\\
                              \spacer
                                        u^\eps_i
                                        = 0, \text{ } i=\phi, x,
                                        &
                                        \text{on } \Gamma, \text{ i.e., at }
r=r_L, r_R,\\
                                        u^\eps_i \big|_{t=0}
                                        = u_{0, \, i}, \text{ } i=\phi, x,
                                        &
                                        \text{in } \Omega ,
        \end{array}\right.
\end{equation}
and, similarly, the EE \eqref{e:EE} become the weakly coupled system:
\begin{equation}\label{e:EE_C}
\left\{\begin{array}{rl}
                            \spacer
                                        \dfrac{\pa u^0_{\phi}}{\pa t}
                                        = f_{\phi},
                                        &
                                        \text{in } \Omega \times(0,T),\\
                             \spacer
                                        \dfrac{\pa u^0_{x}}{\pa t}
                                        + \dfrac{1}{r} u^0_{\phi} \dfrac{\pa
u^0_x}{\pa \phi}
                                        = f_{x},
                                        &
                                        \text{in } \Omega \times(0,T),\\
                              \spacer
                                        - \dfrac{1}{r} (u^0_{\phi})^2
                                        + \dfrac{\pa p^0}{\pa r}
                                        = 0,
                                        &
                                        \text{in } \Omega \times(0,T),\\
                             \spacer
                                        u^0_x \text{ is periodic}
                                        &
                                        \text{in $\phi$ with period $2 \pi$},\\
                                        u^0_i \big|_{t=0}
                                       = u_{0, \, i}, \text{ } i=\phi, x,
                                        &
                                        \text{in } \Omega .
        \end{array}\right.
\end{equation}
Differently than for parallel channel flows, the pressure cannot be taken
constant here, but the pressure can be recovered up to a constant from
knowledge of $u^\phi$.

As in the case of a channel,  we assume that the data, $\uu_0$ and $\f$, are
sufficiently regular, but ill-prepared in the sense that
\begin{equation}\label{e:Data_reg_IPF}
        \uu_0 \in H \cap H^k(\Omega),
            \quad
        \f \in C(0, T ; H \cap H^k(\Omega)),
            \quad
        \text{for a sufficiently large } k \geq0,
\end{equation}
but $\uu_0$ is not assumed to vanish at the boundary nor $\f$  is
assumed necessarily compatible with $\uu_0$ at $t=0$. Again, this choice leads
to the formation of an initial layer for NSE. Since the cross-sectional
component of the inviscid solution is time independent, this initial layer
affects the zero-viscosity limit, in particular vorticity
production at the boundary in the limit.

Under the assumption (\ref{e:Data_reg_IPF}) one can verify that there exists a
unique global strong  solution, which is also classical for positive time.
Similarly, the limit problem (\ref{e:EE_C}) possesses a unique global, strong
solution.
%

Since the tangential component of $\uu^0$ need not vanish on $\Gamma$,  boundary
layers are expected to form on both components of the boundary, the inner and
outer cylinders of the pipe.


In what follows, to highlight the effect of the curvature, we will leave
the explicit dependence on $r$ in many expressions, even though $r$ will be a
bounded function, bounded away from zero. Constants may depend on the geometry
of the pipe through the pipe inner and outer radii, $r_L$, $r_R$.

\subsection{
Viscous approximation and convergence
result}\label{S:IPF_THM}
We will make the following ansatz for the approximate NS solution $\uu^\eps$:
\begin{equation}\label{e:assympt exp_IPF}
        \uu^\eps
                \cong
                    \uu^0 + \blds{\Tht},
\end{equation}
where $\blds{\Tht}$ is a corrector to the inviscid solution $\uu^0$. As for the
channel case, this corrector depends on $\ep$, but for ease of notation, we do
not explicitly show it.

We assume that the corrector satisfies the same symmetry as the NSE and EE
solutions, that is:
\begin{equation}\label{e:Tht_IPF}
        \blds{\Tht}
                =
                    \Tht_{\phi} (r, t) \e_{\phi}
                    +
                    \Tht_{x} (\phi, r, t) \e_x.
\end{equation}

Effective equations for the corrector lead to the \textit{weakly
coupled parabolic system} below
\cite{HMNW11}:
\begin{equation}\label{e:Eq_HMNW11}
        \left\{
                \begin{array}{rl}
                        \spacer \displaystyle
                                \dfrac{\pa \Tht_{\phi}}{\pa t} - \ep
\dfrac{\pa^2 \Tht_{\phi}}{\pa r^2} = 0
                        & \displaystyle
                                \text{in } \Omega \times (0, T),\\
                        \spacer \displaystyle
                                \dfrac{\pa \Tht_x}{\pa t}
                                - \ep \dfrac{\pa^2 \Tht_x}{\pa r^2}
                                + \Tht_{\phi} \dfrac{\pa \Tht_x}{\pa \phi}
                                + u^0_\phi \big|_{\Gamma} \dfrac{\pa \Tht_x}{\pa
\phi}
                                + \dfrac{\pa u^0_x}{\pa \phi}\Big|_{\Gamma}
\Tht_{\phi}
                                = 0
                        & \displaystyle
                                \text{in } \Omega \times (0, T),\\
                        \spacer \displaystyle
                                \Tht_i = - u^0_i
                        &
                                \text{on } \Gamma, \text{ } i= \phi, x,\\
                        \blds{\Tht}|_{t=0} = 0.
                \end{array}
        \right.
\end{equation}
Again, for well-prepared data, these equations allow directly to establish
good regularity and decay propertyes for the corrector, which in turn allow to
study the zero-viscosity limit.

To handle the case of ill-prepared initial data,
we follow the methodology introduced in Section \ref{S:PPF}.
We first define $\Tht_\phi$ as an explicit approximate solution of
the 1D heat equation (\ref{e:Eq_HMNW11})$_1$ on a segment with the proposed
boundary and initial conditions.
Then, we define $\Tht_x$ as a solution of
a drift-diffusion equation similar to (\ref{e:Eq_HMNW11})$_2$ with the proper
boundary and initial conditions.
The analysis proceeds parallel to that of channel flows. However, the curvature
of the boundary has an effect here on the estimates for the approximate
viscous solution.

As in the case of the channel, the corrector $\Tht_\phi$ will be constructed
from the exact solutions of the following one-dimensional heat equations on a
half line,  with boundary and initial data as in (\ref{e:Eq_HMNW11}):
\begin{equation}\label{e:Tht_app_IPF}
\left\{
	\begin{array}{l}
		\displaystyle \spacer
		\Tht_{\phi, L} (r, t)
                            =
                                -2 \, u^0_\phi (r_L, 0 ) \, \erfc\Big(
\dfrac{r-r_L}{\sqrt{2 \ep t}} \Big)
                               -2 \int_0^{t} \dfrac{\pa \, u^0_\phi}{\pa t}
(r_L, s) \, \erfc \Big( \dfrac{r-r_L}{\sqrt{2 \ep (t-s)}} \Big) \, ds,\\
                 \displaystyle
                 \Tht_{\phi, R} (r, t)
                            =
                                -2 \, u^0_\phi (r_R, 0 ) \, \erfc\Big(
\dfrac{r_R-r}{\sqrt{2 \ep t}} \Big)
                               -2 \int_0^{t} \dfrac{\pa \, u^0_\phi}{\pa t}
(r_R, s) \, \erfc \Big( \dfrac{r_R-r}{\sqrt{2 \ep (t-s)}} \Big) \, ds.
        \end{array}
\right.
\end{equation}
Again, these heat solutions satisfy the pointwise and $L^p$ estimates for $\Phi$
in Lemmas \ref{L:pointwise} and \ref{L:L2} in the Appendix, with $\eta$ replaced
respectively by $r-r_L$ and $r_R-r$.

As before, we introduce a smooth cut-off function so as to localize the correctors  near the inner and outer boundaries
$r=r_L, r_R$, in order to enforce the boundary conditions exactly. Let
\begin{equation}\label{e:sig_IPF}
\begin{array}{cc}
                \sig (r) =
                        \left\{
                        \begin{array}{l}
                                \spacer
                                        1, \text{ } r_L \leq r \leq r_L + a_0,\\
                                        0, \text{ } r  \geq r_L + 2 a_0,
                        \end{array}
                        \right.
\end{array}
\end{equation}
for a fixed $ 0 < a_0 \ll r_R - r_L$.

Finally,  we define $\Tht_\phi$ as follows:
\begin{equation}\label{e:Tht_phi}
        \Tht_\phi (r, t)
                = \sig (r) \, \Tht_{\phi, L} (r, t)
                +
                \sig (r_R-r) \, \Tht_{\phi, R} (r, t).
\end{equation}
The equations for the corrector become then:
\begin{equation}\label{e:Eq_Tht_phi}
        \left\{
                \begin{array}{rl}
                        \spacer \displaystyle
                                \dfrac{\pa \Tht_\phi}{\pa t} - \ep \dfrac{\pa^2
\Tht_\phi}{\pa r^2}
                        & \hspace{-2mm}
                                =
                                    -\ep \Big\{
                                                2 \sig^{\pri}(r) \,
\dfrac{\pa \Tht_{\phi, \, L}}{\pa r}
                                                + \sig^{\pri \pri}(r) \,
\Tht_{\phi, \, L}
                                                - 2 \sig^{\pri} (r_R-r)\,
\dfrac{\pa \Tht_{\phi, \, R}}{\pa r}
                                                + \sig^{\pri \pri}(r_R-r) \,
\Tht_{\phi, \, R}
                                            \Big\}
                                \quad
                                \text{in } \Omega,\\
                        \spacer \displaystyle
                                \Tht_\phi
                        & \hspace{-2mm}
                                = - u^0_\phi
                                \quad
                                \text{on } \Gamma, \\
                        \Tht_\phi|_{t=0}
                        & \hspace{-2mm}
                                = 0.
                \end{array}
        \right.
\end{equation}
The right-hand side of (\ref{e:Eq_Tht_phi})$_1$ is an $e.s.t.$  when considered
as a term in (\ref{e:RHS_Eq_Tht_IPF}) later in the section.

Having  $\Tht_{\phi}$ at hand, we can define the last component of the corrector
$\Tht_x$ as follows.  We first utilize the solutions, $\Tht_{x, \, L}$ and
$\Tht_{x, \, R}$, to the following parabolic systems in periodic half spaces:
\begin{equation}\label{e:Eq_Tht_x_L}
        \left\{
                \begin{array}{rl}
                        \spacer \displaystyle
                                \dfrac{\pa \Tht_{x, \, L}}{\pa t} - \ep
\dfrac{\pa^2 \Tht_{x, \, L}}{\pa r^2}
                                - \ep \dfrac{\pa^2 \Tht_{x, \, L}}{\pa \phi^2}
                                +
                                \dfrac{1}{r}
                                \big( \Tht_{\phi} + u^0_\phi \big)
                                		\dfrac{\pa \Tht_{x, \, L}}{\pa
\phi}
                                =
                                - \dfrac{1}{r} \Tht_\phi \dfrac{\pa u^0_x}{\pa
\phi},
                        & \displaystyle
                                (0, 2 \pi) \times (r_L, \infty),\\
                        \spacer
			\Tht_{x, \, L} \text{ is periodic}
			& \hspace{-2mm} \text{ in } \phi,\\
                        \spacer \displaystyle
                                \Tht_{x, \, L} = - u^0_x,
                        &
                                r = r_L, \\
                        \spacer
			\Tht_{x, \, L} \rightarrow 0
			& \hspace{-2mm} \text{ as } r \rightarrow \infty,\\
                        \Tht_{x, \, L}|_{t=0} = 0,
                \end{array}
        \right.
\end{equation}
and
\begin{equation}\label{e:Eq_Tht_x_R}
        \left\{
                \begin{array}{rl}
                        \spacer \displaystyle
                                \dfrac{\pa \Tht_{x, \, R}}{\pa t} - \ep
\dfrac{\pa^2 \Tht_{x, \, R}}{\pa r^2}
                                - \ep \dfrac{\pa^2 \Tht_{x, \, R}}{\pa \phi^2}
                                +
                                \dfrac{1}{r}
                                \big( \Tht_{\phi} + u^0_\phi \big)
                                		\dfrac{\pa \Tht_{x, \,R}}{\pa
\phi}
                                =
                                - \dfrac{1}{r} \Tht_\phi \dfrac{\pa u^0_x}{\pa
\phi},
                        & \displaystyle
                                (0, 2 \pi) \times (-\infty, r_R),\\
                        \spacer
			\Tht_{x, \, R} \text{ is periodic}
			& \hspace{-2mm} \text{ in } \phi,\\
                        \spacer \displaystyle
                                \Tht_{x, \, R} = - u^0_x,
                        &
                                r = r_R, \\
                        \spacer
			\Tht_{x, \, R} \rightarrow 0
			& \hspace{-2mm} \text{ as } r \rightarrow -\infty,\\
                        \Tht_{x, \, R}|_{t=0} = 0,
                \end{array}
        \right.
\end{equation}
Then, using the cut-off function in (\ref{e:sig_IPF}), we define the
second component of the corrector in the form,
\begin{equation}\label{e:Tht_x}
        \Tht_x(\phi, r, t)
                = \sig(r) \, \Tht_{x, \, L} (\phi, r, t)
                    + \sig(r_R-r) \, \Tht_{x, \, R} (\phi, r, t),
\end{equation}
which then satisfies:
\begin{equation}\label{e:Eq_Tht_x}
        \left\{
                \begin{array}{l}
                        \spacer \displaystyle
                                \dfrac{\pa \Tht_{x}}{\pa t} - \ep \dfrac{\pa^2
\Tht_{x}}{\pa r^2}
                                - \ep \dfrac{\pa^2 \Tht_{x}}{\pa \phi^2}
                                +
                                \dfrac{1}{r}
                                \big( \Tht_{\phi} + u^0_\phi \big)
                                		\dfrac{\pa \Tht_{x}}{\pa \phi}\\
                        \spacer \quad 
                        		=
				- \dfrac{1}{r} \Tht_\phi \dfrac{\pa u^0_x}{\pa
\phi}
				-
				\ep \Big\{
                                                2 \sig^{\pri}(r) \,
\dfrac{\pa \Tht_{x, \, L}}{\pa r}
                                                + \sig^{\pri \pri}(r) \,
\Tht_{x, \, L}
                                                - 2 \sig^{\pri}(r_R-r) \,
\dfrac{\pa \Tht_{x, \, R}}{\pa r}
                                                + \sig^{\pri \pri}(r_R-r) \,
\Tht_{x, \, R}
                                            \Big\}
                        \quad
                                \text{in } \Omega,\\
                        \spacer \displaystyle
                                \Tht_x = - u^0_x
                        \quad
                                \text{on } \Gamma, \\
                        \Tht_x|_{t=0} = 0,
                \end{array}
        \right.
\end{equation}
where the second term on the right hand side of (\ref{e:Eq_Tht_x})$_1$ is an
$e.s.t.$ as appearing in (\ref{e:RHS_SIMPLE_FACT_IPF}).

Some needed $L^p$-estimates on the corrector in terms of the viscosity $\eps$
are stated and proved separately in \cref{S:est_Tht_IPF} below.

We next turn to estimating the error between the true NS solution and the
approximate solution. We define again the error term as:
\begin{equation}\label{e:corrected velocity IPF}
        \blds{v}^\ep
                =
                    v^\ep_{\phi} (r, t) \e_{\phi}
                    +
                    v^\ep_x (\phi, r, t) \e_x
                := \uu^\eps - \uu^0 -\blds{\Tht},
\end{equation}
and, using (\ref{e:curl_IPF}), we compute its curl:
\begin{equation}\label{e:omega_IPF}
\begin{array}{rl}
\blds{\omega}^\ep
        & \spacer \hspace{-2mm}
            =
                \omega^\ep_{\phi} (\phi, r, t) \e_{\phi}
                +
                \omega^\ep_x (r, t) \e_{x}
                +
                \omega^\ep_{r} (\phi, r, t) \e_{r}\\
        & \spacer \hspace{-2mm}
            : = \curl \uu^\eps \\
        & \hspace{-2mm}
            =
                -\dfrac{\pa v^\ep_x}{\pa r} \e_{\phi}
                + \dfrac{1}{r} \dfrac{\pa (r v^\ep_{\phi})}{\pa r} \e_{x}
                + \dfrac{1}{r} \dfrac{\pa v^\ep_x}{\pa \phi} \e_r.
\end{array}
\end{equation}

We give some convergence results for PPF below. These are proved in
\cref{S:proof_IPF}.

\begin{theorem}\label{T:IPF}
        Under the regularity assumptions (\ref{e:Data_reg_IPF}) on the data,
for any fixed $0<T<\infty$, the following estimates hold:
        \begin{equation}\label{e:conv_IPF}
        \left\{
                \begin{array}{l}
                    \spacer
                        \| \blds{v}^\ep \|_{L^{\infty}(0, T; L^2(\Omega))}
                        + \ep^{\frac{1}{2}} \| \nabla \blds{v}^\ep \|_{L^{2}(0,
T; L^2(\Omega))}
                                \leq \kap_T \, \ep^{\frac{3}{4}},\\
                    \spacer
                         \| \omega^\ep_\phi \|_{L^{\infty}(0, T; L^2(\Omega))}
                         + \ep^{\frac{1}{2}} \| \nabla \omega^\ep_\phi
\|_{L^{2}(0, T; L^2(\Omega))}
                                \leq \kap_T \,
                                        \ep^{\frac{1}{4}}, \\
                    \spacer
                          \| \omega^\ep_x \|_{L^{\infty}(0, T; L^2(\Omega))}
                         + \ep^{\frac{1}{2}} \| \nabla \omega^\ep_x \|_{L^{2}(0,
T; L^2(\Omega))}
                                \leq \kap_T \, \ep^{\frac{1 }{4}},\\
                         \| \omega^\ep_r \|_{L^{\infty}(0, T; L^2(\Omega))}
                         + \ep^{\frac{1}{2}} \| \nabla\omega^\ep_r \|_{L^{2}(0,
T; L^2(\Omega))}
                                \leq
                                        \kap_T \,
                                        \ep^{\frac{3}{4}}, \\
    \norm{\blds{\omega}^\eps}_{L^\iny(0, T; L^2(\Omega))}
                    + \eps^{\frac{1}{2}} \norm{\grad \blds{\omega}^\epsilon}
                        _{L^2(0, T; L^2(\Omega))}
                    \le \kap_T \eps^{\frac{1}{4}}, \\
                \norm{\vv^\eps}_{L^\iny(0, T; H^1(\Omega))}
                    + \eps^{\frac{1}{2}} \norm{\vv^\epsilon}
                        _{L^2(0, T; H^2(\Omega))}
                    \le \kap_T \eps^{\frac{1}{4}}.
               \end{array}
        \right.
        \end{equation}
where $\ep$ is the viscosity coefficient, $\blds{v}^\ep$ is the approximate NS
solution, $\blds{\omega}^\ep$ its vorticity, and $\kappa_T$ is a constant
independent of $\ep$.

In particular, the parallel pipe viscous solution, $\uu^\eps$, converges to
the corresponding inviscid solution, $\uu^0$, as viscosity vanishes, with
rate:
\begin{equation}\label{e:VVL_IPF}
        \| \uu^\eps - \uu^0 \|_{L^{\infty}(0, T; L^2(\Omega))}
                \leq \kap_T \, \ep^{\frac{1}{4}}.
\end{equation}
Accumulation of vorticity at the boundary occurs in the
limit as in \cref{e:conv_PPF_measure}.
\end{theorem}

As for the channel case, since the size of the corrector can be estimated, by
the triangle inequality the rate of convergence to the inviscid solution is
shown to be optimal.

\begin{rem}\label{r:Conv_P}
\textnormal{
We observe that the rates of convergence for pipe flows is slower than that for
channel flows, which can be ascribed to the presence of a non-zero
boundary curvature in the pipe geometry.}

\textnormal{
We also note that, since the gradient of the pressure is a function of the
$\phi$-component of the velocity (see (\ref{e:NSE_C})$_3$,
(\ref{e:EE_C})$_3$), from (\ref{e:VVL_IPF}) we obtain convergence of the
gradient of the pressure with rate:
\begin{equation*}\label{e:conv_IPF_p}
            \bigg\|
                r \, \dfrac{\pa \big(p^\ep - p^0 \big) }{\pa r}
            \bigg\|_{L^{\infty}(0, T; L^1(\Omega))}
        \leq
        \kap_T \ep^{\frac{1}{4}}.
\end{equation*}
}
\end{rem}



\begin{rem}\label{r:IPF using RSF}
\textnormal{
By considering the evolution of the cross-sectional component of the velocity,
$\uu^\ep_\phi$ and $\uu^0_\phi$,
and its vorticity, $\omega^\ep_x$, we recover the known bounds for CSF.
Accumulation of vorticity at the boundary in the limit takes a particularly
simple form in this case, i.e.,
\begin{equation}\label{e:IPF_RRF_2}
 \lim_{\ep \rightarrow 0}
        \big(
                \curl \uu^\eps \cdot \e_x, \, \varphi
        \big)_{L^2(\Omega)}
                =
                         \big(
                                \curl \uu^0 \cdot \e_x, \, \varphi
                         \big)_{L^2(\Omega)}
                         +
                         \big(
                                u^0_x, \, \varphi
                         \big)_{L^2(\Gamma)},
\end{equation}
for all $\varphi \in C(\ov{\Omega})$, uniformly in $0 < t <T$ (cf.
\cite{LMNT08}).
directions.
}
\end{rem}

\subsection{Estimates of the corrector \texorpdfstring{$\blds{\Theta}$}{}
}\label{S:est_Tht_IPF}
In this section,
we discuss needed estimates on the corrector
$\blds{\Tht} = \Tht_{\phi} (r, t) \e_{\phi} + \Tht_{x} (\phi, r, t) \e_x$,
again using the results in Appendix \ref{S.App}. Since the analysis is very
similar to that for channel flows in \cref{S:est_Tht_PPF}, we omit the details
of the proofs. We recall that each component of the corrector is a combination
of correctors near each wall of the pipe using cut-off functions.

Bounds on the derivatives of $\Tht_\phi$ follow by first observing that
\begin{equation}\label{e:RHS_Eq_Tht_IPF}
        (\text{the right-hand side of } (\ref{e:Eq_Tht_phi})_1)
            =   e.s.t.,
\end{equation}
since $\Tht_{\phi, L}$ and $\Tht_{\phi, R}$ satisfy the pointwise estimates
in Lemma \ref{L:pointwise} with $\eta$  replaced by $r- r_L$ and $r_R -
r$,respectively.
Then, the $L^p$ estimates in Lemma \ref{L:L2} give that
\begin{equation}\label{e:Lp_est_Tht_phi}
\Big\| \dfrac{\pa^{m} \Tht_\phi}{\pa r^m} \Big\|_{L^p(\Omega)}
                                            \leq
                                                    \kap_T \, (1 +
t^{\frac{1}{2p} - \frac{m}{2}}) \, \ep^{\frac{1}{2p} - \frac{m}{2}},
                                            \text{ } 1 \leq p \leq \infty,
                                            \text{ } 0 \leq m \leq 2,
                                            \text{ } 0 < t < T.
\end{equation}

To derive  estimates on the second component $\Tht_x$ of the corrector,
we use the fact that
$\Tht_{x, \, L}$  satisfies the parabolic system
(\ref{e:HEAT_like_eq}) in the domain $(0, L) \times (r_L, \infty)$ with
$\tau=\phi$, $\eta=r-r_L$ , and
\begin{equation*}
	\left\{
		\begin{array}{l}
			\spacer
				\mathcal{U}
					= \dfrac{1}{r} \Big( \Tht_\phi +
u^0_\phi\Big),\\
			\spacer
				G
					= - \dfrac{1}{r} \Tht_\phi \dfrac{\pa
u^0_x}{\pa \phi},\\
				g = - u^0_x |_{r = r_L \text{ (or } r_R)}.
		\end{array}
	\right.
\end{equation*}
(And similarly for $\Tht_{x, \, R}$
with $r-r_L$ replaced by $r_R-r$).
Consequently, both  $\Tht_{x, \, L}$ and $\Tht_{x, \, R}$ satisfies
the estimates in Lemma \ref{l:Lp_est_heat_like_sol}, which give:
\begin{equation}\label{e:Lp_est_Tht_x}
\left\{
\begin{array}{l}
	\spacer
	\Big\| \dfrac{\pa^{k} \Tht_x}{\pa \phi^k} \Big\|_{L^{\infty}(0, T;
L^p(\Omega))}
	+
	\ep^{\frac{1}{2p} + \frac{1}{4}}
	\Big\|
		\nabla \dfrac{\pa^{k} \Tht_x}{\pa \phi^k}
	\Big\|_{L^{2}(0, T; L^2(\Omega))}
                                                                \leq
                                                                        \kap_T
\ep^{\frac{1}{2p}},
                                                                        \quad
                                                                        1 \leq p
\leq \infty,\\
         \spacer
	\Big\| \dfrac{\pa^{k+1} \Tht_x}{\pa \phi^k \pa r} \Big\|_{L^{\infty}(0,
T; L^2(\Omega))}
	+
	\ep^{\frac{1}{2}}
	\Big\|
		\nabla \dfrac{\pa^{k+1} \Tht_x}{\pa \phi^k \pa r}
	\Big\|_{L^{2}(0, T; L^2(\Omega))}
                                                                \leq
                                                                        \kap_T
\ep^{-\frac{1}{4}},\\
	\Big\| \dfrac{\pa^{k+1} \Tht_x}{\pa \phi^k \pa r} \Big\|_{L^{\infty}(0,
T; L^1(\Omega))}
	\leq \kap_T,
\end{array}
\right.
\end{equation}
with $k \geq 0$,
for a constant, $\kap_T$, depending on $T$ and the other data, but independent
of $\ep$.
The $L^1$ bound of $\pa^k \Tht_x/ \pa \phi^k$ is obtained by
the estimates in Lemma \ref{l:Lp_est_heat_like_sol} and
the continuity of $L^p$ norm in a bounded interval.
Moreover, performing an energy estimate on $\Tht_x$ using, as test function, any
derivative of the cut-off function $\sigma$, yields
\begin{equation}\label{e:RHS_SIMPLE_FACT_IPF}
	(\text{the second term on the right hand side of } (\ref{e:Eq_Tht_x})_1)
            =   e.s.t..
\end{equation}

Finally, the bounds on the derivatives of the corrector plus the explicit form
of $\curl \blds{\Tht}$ (see \cref{e:omega_IPF}) imply that
\begin{align}\label{e:curlThtBound_IPF}
    \norm{\curl \blds{\Tht}}_{L^\iny(0, T; L^1(\Omega))}
        \le \kappa_T.
\end{align}

\subsection{Proof of Theorem \ref{T:IPF}}\label{S:proof_IPF}

We recall the notation from the Introduction: $\kappa$, $\kappa_T$  will denote
generic constants that may depend on the data or time, but not on $\ep$, and
may change from line to line.

To prove (\ref{e:conv_IPF})$_1$, using (\ref{e:Laplacian_IPF}), (\ref{e:NSE_C}),
(\ref{e:EE_C}), (\ref{e:Eq_Tht_phi}), (\ref{e:Eq_Tht_x}), and (\ref{e:corrected
velocity IPF}),
we write the equation for $\blds{v}^\ep$:

We begin by deriving the equations satisfied by the components of
$\blds{v}^\ep$, $v^\ep_\phi$ and $v^\ep_x$. We  utilize the the form
(\ref{e:Laplacian_IPF}) for the  Laplacean of a parallel pipe velocity field,
and isolate terms that are negligible in $\ep$, exploiting the estimates for
the corrector obtained above. We therefore have:
\begin{equation}\label{e:Eq_v_IPF}
        \left\{
                \begin{array}{rl}
                        \spacer \displaystyle
                                \dfrac{\pa v^\ep_{\phi}}{\pa t}
                                - \ep \Delta v^\ep_{\phi}
                                + \ep  \dfrac{1}{r^2} \, v^\ep_{\phi}
                                =
                                \ep \Delta u^0_{\phi}
                                - \ep \dfrac{1}{r^2} \, u^0_{\phi}
                                + R^\ep (\Tht_{\phi})
                                +e.s.t.,
                        & \displaystyle
                                \text{in } \Omega,\\
                        \halfspacer \displaystyle
                                \dfrac{\pa v^\ep_x}{\pa t}
                                -
                                \ep \Delta v^\ep_x
                                +
                                \dfrac{1}{r} \,
                                u^\eps_{\phi} \, \dfrac{\pa v^\ep_x}{\pa \phi}
                                + v^\ep_{\phi}
                                    \Big( \dfrac{\pa \Tht_x}{\pa \phi} +
\dfrac{\pa u^0_x}{\pa \phi} \Big)
                                = \ep 
                                                \Delta u^0_x
                                   + R^\ep (\Tht_{x})
                                  +e.s.t.,
                        & \displaystyle
                                \text{in } \Omega,\\
                        \halfspacer \displaystyle
                                \blds{v}^\ep = 0,
                        &
                                \text{on } \Gamma, \\
                        \blds{v}^\ep|_{t=0} = 0,
                \end{array}
        \right.
\end{equation}
where
\begin{equation}\label{e:R_IPF}
\left\{
\begin{array}{l}
\spacer
    R^\ep (\Tht_{\phi})
            = \dfrac{\pa \Tht_\phi}{\pa t}
                                              +\ep \Big\{
                                                \dfrac{1}{r} \dfrac{\pa}{\pa r}
\Big( r \dfrac{\pa
                                                \Tht_\phi}{\pa r}\Big)
                                                - \dfrac{1}{r^2} \Tht_\phi
                                                \Big\} \\
    \displaystyle
    R^\ep (\Tht_{x})
            = -\ep \dfrac{1}{r} \dfrac{\pa \Tht_x}{\pa r} - \ep \dfrac{1}{r^2}
\dfrac{\pa \Tht_x}{\pa \phi^2}.
\end{array}
\right.
\end{equation}

The bound for $v_\phi$ can be easily obtained via standard energy
estimates. In fact, by multiplying the first equation in
(\ref{e:Eq_v_IPF}) by $ v^\ep_\phi$, and integrating over $\Omega$ by
parts,
we deduce that
\begin{equation}\label{e:EST_RSF_2}
\begin{array}{l}
\spacer \displaystyle
        \dfrac{d}{dt} \| v^\ep_\phi \|_{L^2(\Omega)}^2
        + \ep \| \nabla v^\ep_\phi \|_{L^2(\Omega)}^2
        \leq
                        \kap \ep
                        \Big|
                                \big( \Delta \uu_0\cdot \e_\phi, \, v^\ep_\phi
\big)_{L^2(\Omega)}
                        \Big|
                        +
                        \kap
                        \Big|
                                \int_{r_L}^{r_R} R^\ep(\Tht_\phi) v^\ep_\phi r\, dr
                        \Big|,
\end{array}
\end{equation}
for some positive constant $\kappa$, where we used that $\uu^0\cdot \e_\phi$ is
time independent, and that both $R(\tht_\phi)$ and $v^\ep_\phi$ are radial.
The estimates on the corrector $\Tht_\phi$ of the previous section give:
\begin{equation}\label{e:R_RSF_est_3}
        \Big|
        \int_{r_L}^{r_R} R^\ep(\Tht_\phi) v^\ep_\phi r \, dr
        \Big|
                \leq \kap_T \, t^{- \frac{1}{4}} \ep^{\frac{3}{4}} \| \sqrt{r}
\, v^\ep_\phi \|_{L^2(r_L, r_R)}
                \leq \kap_T t^{- \frac{1}{2}} \ep^{\frac{3}{2}} + \|
\blds{v}^\ep \|_{L^2(\Omega)}^2,
\end{equation}
from which it follows, by applying Cauchy's inequality on $\big( \Delta
\uu_0\cdot \e_\phi, \, v^\ep_\phi \big)_{L^2(\Omega)}$, that
\begin{equation}\label{e:EST_RSF_3}
\begin{array}{l}
\spacer \displaystyle
        \dfrac{d}{dt} \| \blds{v}^\ep \|_{L^2(\Omega)}^2
        + \ep \| \nabla \blds{v}^\ep \|_{L^2(\Omega)}^2
                \leq
                    \kap_T \, (1 + t^{-\frac{1}{2}} ) \ep^{\frac{3}{2}} + 2 \|
\blds{v}^\ep \|_{L^2(\Omega)}^2,
\end{array}
\end{equation}
where $\kappa_T$ depends on the data, but not on $\ep$.
Gr\"onwall's inequality then gives:
\begin{equation}\label{e:IPF_RRF_11}
                    \| v^\ep_{\phi} \|_{L^{\infty}(0, T; L^2(\Omega))}
                    + \ep^{\frac{1}{2}} \| \nabla v^\ep_{\phi} \|_{L^{2}(0, T;
L^2(\Omega))}
                                \leq \kap_T \, \ep^{\frac{3 }{4}}.
\end{equation}



We proceed similarly for the component of the velocity
along the axis.  We multiply (\ref{e:Eq_v_IPF})$_2$ by  $v^\ep_x$ and
integrate it over $\Omega$.
Using the fact that $v^\eps_{\phi}$ is independent of the variable
$\phi$, we obtain:
\begin{equation*}\label{e:engergy_v_x_IPF-1}
    \begin{array}{l}
        \spacer
                \dfrac{1}{2}\dfrac{d}{dt} \| v^\ep_x \|_{L^2(\Omega)}^2
                + \ep \| \nabla v^\ep_x \|_{L^2(\Omega)}^2\\
        \qquad
                \leq
                        \| v^\ep_{\phi} \|_{L^2(\Omega)}
                        \Big\|
                            \dfrac{\pa \Tht_x}{\pa \phi}
                            + \dfrac{\pa u^0_x}{\pa \phi}
                        \Big\|_{L^{\infty}(\Omega)}
                        \| v^\ep_x \|_{L^2(\Omega)}
                        +
                        \big(
                        \ep
                        \| \Delta u^0_x \|_{L^2(\Omega)}
                        +
                        \| R^\ep (\Tht_{x}) \|_{L^2(\Omega)}
                        \big)
                        \| v^\ep_x \|_{L^2(\Omega)}.
    \end{array}
\end{equation*}
Again, Cauchy's inequality, the estimates on the corrector, and the bound on
$v^\ep_\phi$ in (\ref{e:IPF_RRF_11}) yield:
\begin{equation*}\label{e:engergy_v_x_IPF-2}
                \dfrac{1}{2}\dfrac{d}{dt} \| v^\ep_x \|_{L^2(\Omega)}^2
                + \ep \| \nabla v^\ep_x \|_{L^2(\Omega)}^2
                \leq
                        \kap \ep^{\frac{3}{4}} \| v^\ep_{x} \|_{L^2(\Omega)}
                \leq
                        \kap \ep^{\frac{3}{2}}
                        +
                        \| v^\ep_{x} \|_{L^2(\Omega)}^2.
\end{equation*}
and by Gr\"onwall's:
\begin{equation}\label{e:engergy_v_x_IPF-3}
        \| v^\ep_x \|_{L^{\infty}(0, T; L^2(\Omega))}
        + \ep^{\frac{1}{2}}
            \| \nabla v^\ep_x \|_{L^{2}(0, T; L^2(\Omega))}
        \leq
            \kap_T \ep^{\frac{3}{4}},
\end{equation}
which together with  (\ref{e:IPF_RRF_11}) gives (\ref{e:conv_IPF})$_1$.
Estimate (\ref{e:VVL_IPF}) follows from (\ref{e:conv_IPF})$_1$ by the triangle
inequality, given the estimates on the corrector $\Tht$.
%
%

To verify (\ref{e:conv_IPF})$_{2, 3, 4}$, we once gain follow the idea of Lighthill.
We differentiate (\ref{e:Eq_v_IPF})$_i$, $i=1,2$, in $r$, and divide (\ref{e:Eq_v_IPF})$_2$ by $r$ after differentiating it in $\phi$.
We then find the equations for $\blds{\omega}^\ep$, utilizing the explicit form for the pipe vorticity given in (\ref{e:omega_IPF}):
\begin{equation}\label{e:Eq_vorticity IPF}
        \left\{
                \begin{array}{l}
                        \spacer \displaystyle
                                \dfrac{\pa \omega^\ep_{\phi}}{\pa t}
                                - \ep \Delta \omega^\ep_{\phi}
                                +
                                \dfrac{1}{r} \,
                                u^\eps_{\phi} \, \dfrac{\pa \omega^\ep_{\phi}}{\pa {\phi}}\\
                        \spacer \qquad
                                =
                                - \ep \dfrac{\pa (\Delta u^0_x)}{\pa r}
                                - \dfrac{\pa (R^\ep (\Tht_{x}))}{\pa r}
                                +
                                \Big(
                                    \omega^\ep_x - \dfrac{1}{r} \, v^\ep_{\phi}
                                \Big)
                                \dfrac{\pa (u^0_x + \Tht_x)}{\pa {\phi}}
                                +
                                v^\ep_{\phi} \, \dfrac{\pa^2 (u^0_x + \Tht_x)}{\pa \phi \, \pa r}\\
                        \spacer \qquad \quad
                                -
                                \dfrac{1}{r} \, u^\eps_{\phi} \, \omega^\ep_r
                                +
                                \dfrac{\pa u^\eps_{\phi}}{\pa r} \, \omega^\ep_r
                                + e.s.t.
                        \quad
                                \text{in } \Omega,\\
                        \spacer \displaystyle
                                \dfrac{\pa \omega^\ep_x}{\pa t}
                                - \ep \dfrac{1}{r} \dfrac{\pa}{\pa r} \Big( r \, \dfrac{\pa \omega^\ep_x}{\pa r} \Big)
                                =
                                \dfrac{1}{r} \dfrac{\pa}{\pa r} \Big(
                                                                                r \, \big(\text{right-hand side of } (\ref{e:Eq_v_IPF})_1 \big)
                                                                            \Big)
                        \quad
                                \text{in } \Omega,\\
                        \displaystyle \spacer
                                \dfrac{\pa \omega^\ep_r}{\pa t}
                                - \ep \Delta \omega^\ep_r
                                - \ep \dfrac{1}{r^2} \, \omega^\ep_r
                                - 2 \ep \dfrac{1}{r} \dfrac{\pa \omega^\ep_r}{\pa r}
                                +
                                \dfrac{1}{r} \,
                                u^\eps_{\phi} \, \dfrac{\pa \omega^\ep_r}{\pa \phi} \\
                        \qquad
                                =
                                \ep \, \dfrac{1}{r} \,
                                \dfrac{\pa (\Delta u^0_x)}{\pa \phi}
                                +
                                \dfrac{1}{r} \,
                                \dfrac{\pa (R^\ep (\Tht_{x}))}{\pa \phi}
                                -
                                \dfrac{1}{r} \,
                                v^\ep_{\phi} \, \dfrac{\pa^2 (u^0_x + \Tht_x)}{\pa \phi^2}
                        \quad
                                \text{in } \Omega.
                \end{array}
        \right.
\end{equation}
The boundary and initial conditions for $\blds{\omega}^\ep$ can be obtained by
restricting (\ref{e:Eq_v_IPF})$_{1,2}$ on $\Gamma$ and using
(\ref{e:Eq_v_IPF})$_{3, 4}$:
 \begin{equation}\label{e:BC_IC_vorticity IPF}
        \left\{
                \begin{array}{l}
                        \spacer \displaystyle
                                \dfrac{\pa \omega^\ep_{\phi}}{\pa r}
                                +
                                \omega^\ep_{\phi}
                                =
                                \Delta u^0_x
                        \quad
                                \text{on } \Gamma, \\
                        \spacer \displaystyle
                                \dfrac{\pa \omega^\ep_x}{\pa r}
                                =
                                - \ep^{-1} \, \big(\text{right-hand side of } (\ref{e:Eq_v_IPF})_1 \text{ up to  $e.s.t.$ terms} \big)
                        \quad 
                                \text{on } \Gamma, \\
                        \spacer \displaystyle
                                \omega^\ep_r = 0
                        \quad 
                                \text{on } \Gamma, \\
                        \blds{\omega}^\ep|_{t=0} = 0.
                \end{array}
        \right.
\end{equation}
Above, we have used the fact that the  $e.s.t.$ terms  in (\ref{e:Eq_v_IPF}) vanish on $\Gamma$ (see the right-hand side of (\ref{e:Eq_Tht_phi}) and (\ref{e:Eq_Tht_x})) .

We first estimate $\omega^\eps_x$, which is precisely the scalar vorticity of
the cross-sectional velocity in the pipe. Taking $\omega_x^\ep$ as test
function and integrating by parts gives:
\begin{multline}\label{e:H1_est_RSF 1-1}
        \dfrac{1}{2} \dfrac{d}{dt} \| \omega^\ep_x \|_{L^2(\Omega)}^2
        + \ep \|
                            \nabla \omega^\ep_x
                    \|_{L^2(\Omega)}^2
        =
            \ep \Big(
                        \omega^\ep \dfrac{\pa \omega^\ep_x }{\pa r}
                    \Big)\Big|_{\Gamma} \\
            +
            \int_{r_L}^{r_R} \Big\{
                                        \curl \big( \ep (\Delta \uu_0\cdot
\e_\phi +
R^\ep(\Tht_\phi)) \e_{\phi}\big)
                                \Big\} r \, \omega^\ep_x
                                \, dr + \; e.s.t,
\end{multline}
identifying the two-dimensional curl with a scalar function, and exploiting
again that the cross-sectional inviscid velocity is stationary, and that
$u_{0,\phi}$, $\omega^\ep_x$, $R^\ep(\Tht_\phi)$ are radial.
We can bound the first term on the right-hand side above as follows:
\begin{equation*}\label{e:EST_V_RSF_1}
    \ep
        \left|
        \Big(
              \omega^\ep_x(r_*,t) \, \dfrac{\pa \omega^\ep_x(r_*,t) }{\pa r}
        \Big)\Big|
        \right|
            \leq
                \ep
                    \left|
                        \dfrac{\pa \omega^\ep_x }{\pa r}(r_*,t)
                        \right|
                        \,
                            \big|
                                \omega^\ep_x(r_*,t)
                             \big|
            \leq
                \kap (1 + t^{- \frac{1}{2}})\ep^{\frac{1}{2}}
                \abs{\omega^\eps_x(r_*,t)},
\quad
	* = L, R,
\end{equation*}
where we have used the Neumann boundary condition (\ref{e:BC_IC_vorticity
IPF})$_2$ for $\omega^\ep_x$ and the
estimates for the corrector to bound the radial derivative of $\omega^\ep_x$
on $\Gamma$. But
$\abs{\omega^\eps_x(r_*, t)} = (2 \pi)^{-1/2} \|\omega^\ep_x(t)\|_{L^2(\Gamma)}$,
$*=L, R$,
which is estimated using the bounds on the trace in \cref{C:TraceCor} with $p =
2$.
We then apply a modified version of Young's inequality with three factors
(more precisely we apply \cref{L:Youngs} with $p = 4/3$, $q = 4$, $p_1 = 3/2$,
$p_2 =3$) and obtain:
\begin{align}\label{e:EST_V_RSF_3}
    \begin{split}
    \eps \abs{\pr{\pdx{\omega^\eps_x \omega^\eps_x}{r}|_\Gamma}}
        &\le \kappa (1 + t^{-\frac{1}{2}}) \eps^{\frac{1}{2}}
            \norm{\omega^\eps}_{L^2(\Omega)}^{\frac{1}{2}}
            \norm{\grad \omega^\eps_x}_{L^2(\Omega)}^{\frac{1}{2}} \\
        &= \kappa
            \brac{(1 + t^{-\frac{1}{2}})^\frac{2}{3} \eps^{\frac{1}{4}}}
            \brac{(1 + t^{-\frac{1}{2}})^\frac{1}{3}
                \norm{\omega^\eps_x}_{L^2(\Omega)}^{\frac{1}{2}}}
            \brac{2^{-\frac{3}{4}} \eps^{\frac{1}{4}}
            \norm{\grad \omega^\eps_x}_{L^2(\Omega)}^{\frac{1}{2}}} \\
        &\le \kappa (1 + t^{-\frac{2}{3}})
            \eps^{\frac{1}{2}}
            +
            \kappa (1 + t^{-\frac{2}{3}}) \norm{\omega^\eps_x}_{L^2(\Omega)}^2
            +
            \frac{\eps}{2}
            \norm{\grad \omega^\eps_x}_{L^2(\Omega)}^2.
    \end{split}
\end{align}

Similarly, the regularity of $\uu_0$, the
estimates for the corrector, and
standard inequalities give for the second term on the right-hand side of
(\ref{e:H1_est_RSF 1-1}):
\begin{equation}\label{e:H1_est_RSF 2}
        \begin{array}{l}
                \displaystyle 
                \Big|
                    \int_{r_L}^{r_R}
                        \curl \big( \ep (\Delta \uu_0\cdot \e_\phi +
R^\ep(\Tht_\phi)) \e_{\phi}\big) \,
                        r \omega^\ep_x
                    \, dr
                \Big|\\
                \displaystyle \spacer \hspace{15mm}
                                \leq
                                        \ep \Big\{
                                                    \kap \| \curl \Delta \uu_0
\|_{L^2(\Omega)}
                                                    +
                                                    \sum_{m=0}^2 \Big\|

       \sqrt{r} \dfrac{\pa^m \Tht}{\pa r^m}

   \Big\|_{L^2(0,1)}
                                                \Big\}
                                                \kap
                                                \|  \omega^\ep_x
\|_{L^2(\Omega)}\\
                \displaystyle \spacer \hspace{15mm}
                                \leq
                                        \kap (1 + t^{-\frac{3}{4}})
\ep^{\frac{1}{4}}  \| \omega^\ep_x \|_{L^2(\Omega)}\\
                \displaystyle \hspace{15mm}
                                \leq
                                        \kap (1 + t^{-\frac{3}{4}})
\ep^{\frac{1}{2}}
                                        +
                                        \kap (1 + t^{-\frac{3}{4}})\|
\omega^\ep_x \|_{L^2(\Omega)}^2.
        \end{array}
\end{equation}
Combining the estimates above into (\ref{e:H1_est_RSF 1-1}), we have for
$\ep$ small enough  that
\begin{equation}\label{e:H1_est_RSF 3}
        \dfrac{d}{dt} \| \omega^\ep_x \|_{L^2(\Omega)}^2
        + \ep \| \nabla \omega^\ep_x \|_{L^2(\Omega)}^2
                    \leq
                            \kap (1 + t^{-\frac{3}{4}}) \ep^{\frac{1}{2}}
                                        +
                            \kap (1 + t^{-\frac{3}{4}})\| \omega^\ep_x
\|_{L^2(\Omega)}^2,
\end{equation}
from which the bound on $\omega^\ep_x$ follows from Gr\"onwall's lemma with the
integrating factor $\textnormal{exp}(-\kap t -4 \kap t^{1/4})$:
\begin{equation} \label{e:IPF_RRF_12}
   \| \omega^\ep_x \|_{L^{\infty}(0, T; L^2(\Omega))}
                   + \ep^{\frac{1}{2}} \| \nabla \omega^\ep_x \|_{L^{2}(0, T;
L^2(\Omega))}
                                \leq \kap_T \, \ep^{\frac{1 }{4}}.
\end{equation}


We proceed in an entirely analogous fashion for the radial component of the
vorticity, multiplying (\ref{e:Eq_vorticity IPF})$_3$ by $\omega^\ep_r$ and
integrating over $\Omega$ by parts using the homogeneous Dirichlet condition on
$\omega^\ep_r$:
 \begin{equation}\label{e:conv_w_r_IPL-1}
    \begin{array}{l}
        \spacer
                \dfrac{1}{2}\dfrac{d}{dt} \| \omega^\ep_r \|_{L^2(\Omega)}^2
                + \ep \| \nabla \omega^\ep_r \|_{L^2(\Omega)}^2\\
        \qquad \spacer 
                        \leq
                                    \ep \Big\|
                                                \dfrac{1}{r} \, \omega^\ep_r
                                            \Big\|_{L^2(\Omega)}^2
                                    +
                                    \ep
                                            \Big\|
                                                \dfrac{1}{r}
                                                \dfrac{\pa (\Delta u^0_x )}{\pa \phi}
                                            \Big\|_{L^2(\Omega)}
                                            \|
                                                \omega^\ep_r
                                    	\|_{L^2(\Omega)}
                                    +
                                    \Big\|\dfrac{\pa(
                                    R^\ep (\Tht_{x}))}{\pa \phi}
                                    \Big\|_{L^2(\Omega)}
                                            \|
                                                \omega^\ep_r
                                    	\|_{L^2(\Omega)}\\
	\qquad \quad \spacer \displaystyle
                                    +
                                    \| v^\ep_{\phi} \|_{L^2(\Omega)}
                                    \Big\|
                                            \dfrac{1}{r}
                                         \dfrac{\pa^2 (u^0_x + \Tht_x)}{\pa \phi^2}
                                    \Big\|_{L^{\infty}(\Omega)}
                                    \|
                                                \omega^\ep_r
                                    \|_{L^2(\Omega)}\\

        \qquad  
                        \leq
                                     \kap_T \ep^{\frac{3}{2}}
                                    +
                                    \kap_T
                                    \|
                                                \omega^\ep_r
                                    \|_{L^2(\Omega)}^2.
    \end{array}
\end{equation}
As before, (\ref{e:conv_IPF})$_4$ follows by applying Gr\"onwall's inequality.

A similar integration by parts leads to the following energy estimate for $
\omega^\ep_{\phi}$, using  again that  $u^\eps_{\phi}$ is independent of $\phi$:
\begin{equation}\label{e:omega_phi_IPF 1}
    \begin{array}{l}
        \spacer
            \dfrac{1}{2} \dfrac{d}{dt} \| \omega^\ep_{\phi} \|_{L^2(\Omega)}^2
            +
            \ep \| \nabla \omega^\ep_{\phi} \|_{L^2(\Omega)}^2\\
        \spacer \qquad
            \leq
            \ep \Big\| \dfrac{\pa \omega^\ep_{\phi} }{\pa r} \, \omega^\ep_{\phi} \Big\|_{L^1(\Gamma)}
            +
            \ep \Big\| \dfrac{\pa (\Delta u^0_x) }{\pa r} \Big\|_{L^2(\Omega)}
            \| \omega^\ep_{\phi} \|_{L^2(\Omega)}
            +
            \Big\|\dfrac{\pa(
                                    R^\ep (\Tht_{x}))}{\pa r}
                                    \Big\|_{L^2(\Omega)}
                                            \|
                                                \omega^\ep_\phi
                                    	\|_{L^2(\Omega)}
            \\
        \spacer \qquad \quad
            +
            \Big\{
                    \| \omega^\ep_x \|_{L^2(\Omega)}
                    \Big\| \dfrac{\pa \big( u^0_x + \Tht_x \big) }{\pa \phi} \Big\|_{L^{\infty}(\Omega)}
                    +
                    \| v^\ep_{\phi} \|_{L^2(\Omega)}
                    \Big\| \dfrac{1}{r} \dfrac{\pa \big( u^0_x + \Tht_x \big) }{\pa \phi} \Big\|_{L^{\infty}(\Omega)}
            \Big\}
            \| \omega^\ep_{\phi} \|_{L^2(\Omega)}\\
        \spacer \qquad \quad \displaystyle
            +
            \bigg|
                \int_0^{2\pi} \int_0^L \int_{r_L}^{r_R}
                    v^\ep_{\phi} \, \dfrac{\pa^2 \big( u^0_x + \Tht_x \big) }{\pa \phi \, \pa r} \, \omega^\ep_{\phi} \,
                    r \, dr dx d\phi
            \bigg|\\
        \spacer \qquad \quad \displaystyle
            +
            \Big\|
                -\dfrac{1}{r} \, u^\eps_{\phi} + \dfrac{\pa u^\eps_{\phi}}{\pa r}
            \Big\|_{L^{\infty}(\Omega)}
            \| \omega^\ep_r \|_{L^2(\Omega)}
            \| \omega^\ep_{\phi} \|_{L^2(\Omega)}.
    \end{array}
\end{equation}
Above, we have also used that the integrals containing  the fourth and
fifth term on the left-hand side of $(\text{\ref{e:Eq_vorticity IPF}})_3$
vanish.

We next estimate each term on the right-hand side of (\ref{e:omega_phi_IPF 1})
in order to apply Gr\"onwall's Lemma. The second term is already in the
appropriate form.
The first term can be bounded as follows, using the boundary condition
for $\omega^\ep_\phi$ (\ref{e:BC_IC_vorticity IPF})$_1$:
\begin{equation}\label{e:omega_phi_IPF 2}
        \ep \Big\| \dfrac{\pa \omega^\ep_{\phi} }{\pa r} \, \omega^\ep_{\phi} \Big\|_{L^1(\Gamma)}
                \leq
                    \ep
                    \Big\| \dfrac{\pa \omega^\ep_{\phi} }{\pa r} \Big\|_{L^2(\Gamma)}
                    \| \omega^\ep_{\phi} \|_{L^2(\Gamma)}
                \leq
                    \ep
                    \| \omega^\ep_{\phi} \|_{L^2(\Gamma)}^2
                    +
                    \kap \ep
                    \| \omega^\ep_{\phi} \|_{L^2(\Gamma)}.
\end{equation}
Standard estimates then give:
\begin{equation}\label{e:omega_phi_IPF 3}
\begin{array}{rl}
    \ep
    \| \omega^\ep_{\phi} \|_{L^2(\Gamma)}^2
        & \spacer \hspace{-2mm}
        \leq
            \kap \ep
            \| \omega^\ep_{\phi} \|_{L^2(\Omega)}
            \| \omega^\ep_{\phi} \|_{H^1(\Omega)}\\
        & \spacer \hspace{-2mm}
        \leq
            \kap \ep
            \| \omega^\ep_{\phi} \|_{L^2(\Omega)}^2
            +
            \kap \ep
            \| \omega^\ep_{\phi} \|_{L^2(\Omega)}
            \| \nabla \omega^\ep_{\phi} \|_{L^2(\Omega)}\\
        &\hspace{-2mm}
        \leq
            \kap \ep
            \| \omega^\ep_{\phi} \|_{L^2(\Omega)}^2
            +
            \dfrac{1}{8} \ep
            \| \nabla \omega^\ep_{\phi} \|_{L^2(\Omega)}^2,
\end{array}
\end{equation}
\begin{equation}\label{e:omega_phi_IPF 4}
\begin{array}{rl}
    \ep
    \| \omega^\ep_{\phi} \|_{L^2(\Gamma)}
        & \spacer \hspace{-2mm}
        \leq
            \kap \ep
            \| \omega^\ep_{\phi} \|_{L^2(\Omega)}^{\frac{1}{2}}
            \| \omega^\ep_{\phi} \|_{H^1(\Omega)}^{\frac{1}{2}}\\
        & \spacer \hspace{-2mm}
        \leq
            \kap \ep
            \| \omega^\ep_{\phi} \|_{L^2(\Omega)}
            +
            \kap \ep
            \| \omega^\ep_{\phi} \|_{L^2(\Omega)}^{\frac{1}{2}}
            \| \nabla \omega^\ep_{\phi} \|_{L^2(\Omega)}^{\frac{1}{2}}\\
        & \spacer \hspace{-2mm}
        \leq
            \kap \ep
            \| \omega^\ep_{\phi} \|_{L^2(\Omega)}
            +
            \kap \ep
            \| \nabla \omega^\ep_{\phi} \|_{L^2(\Omega)}\\
        &\hspace{-2mm}
        \leq
            \kap \ep
            +
            \kap \ep
            \| \omega^\ep_{\phi} \|_{L^2(\Omega)}^2
            +
            \dfrac{1}{8} \ep
            \| \nabla \omega^\ep_{\phi} \|_{L^2(\Omega)}^2,
\end{array}
\end{equation}
from which it follows  that
\begin{equation}\label{e:omega_phi_IPF 5}
        \ep \Big\| \dfrac{\pa \omega^\ep_{\phi} }{\pa r} \, \omega^\ep_{\phi} \Big\|_{L^1(\Gamma)}
                \leq
                    \kap \ep
                    +
                    \kap \| \omega^\ep_{\phi} \|_{L^2(\Omega)}^2
                    +
                    \dfrac{1}{4} \ep \| \nabla \omega^\ep_{\phi} \|_{L^2(\Omega)}^2.
\end{equation}
We bound the third term on the right-hand side of (\ref{e:omega_phi_IPF
1}), employing the explicit form of $R^\ep(\Tht_x)$ and the estimates on the
corrector:
\begin{equation}\label{e:ADDITIONAL_IPF}
\begin{array}{rl}
	\spacer \displaystyle
	\Big\|\dfrac{\pa(R^\ep (\Tht_{x}))}{\pa r}\Big\|_{L^2(\Omega)}
	\| \omega^\ep_\phi \|_{L^2(\Omega)}
		&
		\leq
			\Big\{
				\kap_T \ep^{\frac{3}{4}} + \ep \Big\|\dfrac{\pa^2 \Tht_{x}}{\pa r^2}\Big\|_{L^2(\Omega)}
			\Big\}
			\| \omega^\ep_\phi \|_{L^2(\Omega)}\\
		&
		\leq
			\kap_T \ep^{\frac{3}{2}} +
			\kap \ep \Big\|\dfrac{\pa^2 \Tht_{x}}{\pa r^2}\Big\|_{L^2(\Omega)}^2
			+
			\| \omega^\ep_\phi \|_{L^2(\Omega)}^2.
\end{array}
\end{equation}
The fourth and fifth terms on the right-hand side of (\ref{e:omega_phi_IPF 1})
are readily  estimated, thanks to the bounds on
$\blds{v}^\ep$ and $\omega^\ep_x$ already established, (\ref{e:conv_IPF})$_1${}
and (\ref{e:conv_IPF})$_3$:
\begin{multline}\label{e:omega_phi_IPF 6}
    \spacer
    \Big\{
                    \| \omega^\ep_x \|_{L^2(\Omega)}
                    \Big\| \dfrac{\pa \big( u^0_x + \Tht_x \big) }{\pa \phi} \Big\|_{L^{\infty}(\Omega)}
                    +
                    \| v^\ep_{\phi} \|_{L^2(\Omega)}
                    \Big\| \dfrac{1}{r} \dfrac{\pa \big( u^0_x + \Tht_x \big) }{\pa \phi} \Big\|_{L^{\infty}(\Omega)}
    \Big\}
    \| \omega^\ep_{\phi} \|_{L^2(\Omega)}\\
            \leq
                    \kap\ep^{\frac{1}{4}}\| \omega^\ep_{\phi} \|_{L^2(\Omega)}
            \leq
                    \kap\ep^{\frac{1}{2}} + \kap \| \omega^\ep_{\phi} \|_{L^2(\Omega)}^2.
\end{multline}

After integrating by parts in the $r$ direction, we write the sixth term on the
right-hand side of (\ref{e:omega_phi_IPF 1}) in the form, using that
$v^\ep_\phi$ vanishes on the boundary:
\begin{equation}\label{e:omega_phi_IPF 7}
\begin{array}{l}
    \spacer \displaystyle
        \bigg|
        \int_0^{2\pi} \int_0^L \int_{r_L}^{r_R}
            v^\ep_{\phi} \, \dfrac{\pa^2 \big( u^0_x + \Tht_x \big) }{\pa \phi \, \pa r} \, \omega^\ep_{\phi} \,
            r \, dr dx d\phi
        \bigg|\\
            \spacer \displaystyle
            =
                \bigg|
                    \int_0^{2\pi} \int_0^L \int_{r_L}^{r_R}
                     \dfrac{\pa \big( u^0_x + \Tht_x \big) }{\pa \phi } \,
                     \dfrac{\pa \big( v^\ep_{\phi} \, \omega^\ep_{\phi} \, r \big) }{\pa r} \, dr dx d\phi
                \bigg|\\
            \spacer \displaystyle
            \leq
                \Big\|
                    \dfrac{\pa \big( u^0_x + \Tht_x \big) }{\pa \phi }
                    \dfrac{\pa v^\ep_{\phi}}{\pa r} \, \omega^\ep_{\phi}
                \Big\|_{L^1(\Omega)}
                +
                \Big\|
                    \dfrac{\pa \big( u^0_x + \Tht_x \big) }{\pa \phi } \,
                    v^\ep_{\phi} \,
                    \dfrac{\pa \omega^\ep_{\phi}}{\pa r}
                \Big\|_{L^1(\Omega)}
                +
                \Big\|
                    \dfrac{\pa \big( u^0_x + \Tht_x \big) }{\pa \phi } \,
                    v^\ep_{\phi} \, \omega^\ep_{\phi} \,\dfrac{1}{r}
                \Big\|_{L^1(\Omega)}\\
            \spacer \displaystyle
            \leq
                \Big\|
                    \dfrac{\pa \big( u^0_x + \Tht_x \big) }{\pa \phi }
                \Big\|_{L^{\infty}(\Omega)}
                \bigg\{
                        \Big\|
                            \dfrac{\pa v^\ep_{\phi}}{\pa r}
                        \Big\|_{L^2(\Omega)}
                        +
                        \Big\|
                            \dfrac{v^\ep_{\phi}}{r}
                        \Big\|_{L^2(\Omega)}
                \bigg\}
                        \|
                            \omega^\ep_{\phi}
                        \|_{L^2(\Omega)} \\
            \spacer \displaystyle
            \quad
                +
                \Big\|
                    \dfrac{\pa \big( u^0_x + \Tht_x \big) }{\pa \phi }
                \Big\|_{L^{\infty}(\Omega)}
                \|
                    v^\ep_{\phi}
                \|_{L^2(\Omega)}
                \Big\|
                    \dfrac{\pa \omega^\ep_{\phi}}{\pa r}
                \Big\|_{L^2(\Omega)}\\
            \spacer \displaystyle
            \leq
                \kap
                \Big\|
                    \dfrac{\pa v^\ep_{\phi}}{\pa r}
                \Big\|_{L^2(\Omega)}
                \|
                    \omega^\ep_{\phi}
                \|_{L^2(\Omega)}
                +
                \kap \ep^{\frac{1}{2}}\,  \|
                    \omega^\ep_{\phi}
                \|_{L^2(\Omega)} +
                \kap \ep^{\frac{3}{4}}
                \Big\|
                    \dfrac{\pa \omega^\ep_{\phi}}{\pa r}
                \Big\|_{L^2(\Omega)}\\
            \leq
                \kap
                \Big\|
                    \dfrac{\pa v^\ep_{\phi}}{\pa r}
                \Big\|_{L^2(\Omega)}^2
                +
                \kap
                \|
                    \omega^\ep_{\phi}
                \|_{L^2(\Omega)}^2
                +
                \kap \ep^{\frac{1}{2}}
                +
                \dfrac{1}{4} \ep
                \|
                    \nabla
                    \omega^\ep_{\phi}
                \|_{L^2(\Omega)}^2,
\end{array}
\end{equation}
where  we employed estimate (\ref{e:IPF_RRF_11}) and the fact that
$0<r_L<r<r_R<\infty$ to bound $\|v^\ep_{\phi}/r\|_{L^2(\Omega)}$.

We bound the last term on the right-hand side of (\ref{e:omega_phi_IPF 1}), by
first writing it in terms of $\blds{v}^\ep$ and $\blds{\omega}^\ep$, exploiting
the explicit form of the curl for a parallel pipe flow, and then using the
estimates on the corrector:
\begin{equation}\label{e:omega_phi_IPF 8}
\begin{array}{l}
\spacer
\Big\|
-\dfrac{1}{r} \, u^\eps_{\phi} + \dfrac{\pa u^\eps_{\phi}}{\pa r}
\Big\|_{L^{\infty}(\Omega)}\\
        \spacer \qquad
        \leq
                2
                \Big\|
                    \dfrac{1}{r} \, v^\ep_{\phi}
                \Big\|_{L^{\infty}(\Omega)}
                +
                \Big\|
                    \dfrac{1}{r} \, \big(u^0_{\phi} + \Tht_{\phi}\big)
                \Big\|_{L^{\infty}(\Omega)}
                +
                \|
                    \omega^\ep_x
                \|_{L^{\infty}(\Omega)}
                +
                \Big\|
                    \dfrac{\pa \big(u^0_{\phi} + \Tht_{\phi}\big)}{ \pa r}
                \Big\|_{L^{\infty}(\Omega)}\\
	\qquad
        \leq
                \kap
                \|
                    v^\ep_{\phi}
                \|_{L^{\infty}(\Omega)}
                +
                \|
                    \omega^\ep_x
                \|_{L^{\infty}(\Omega)}
                +
                \kap_T(1 + t^{-\frac{1}{2}}) \ep^{-\frac{1}{2}}.
\end{array}
\end{equation}
Then, Poincar\'e's, the one-dimensional Agmon's inequalities, and the bounds
on $v^\ep_{\phi}$ and $\omega^\ep_{x}$, (\ref{e:conv_IPF})$_1$ and
(\ref{e:conv_IPF})$_3$, yield:
\begin{align}
\| v^\ep_{\phi} \|_{L^{\infty}(\Omega)}
    &\leq \kap
            \| v^\ep_{\phi} \|_{L^{2}(\Omega)}^{\frac{1}{2}}
            \| v^\ep_{\phi} \|_{H^{1}(\Omega)}^{\frac{1}{2}}
    \leq \kap
            \ep^{\frac{3}{8}} \| \nabla v^\ep_{\phi}
\|_{L^{2}(\Omega)}^{\frac{1}{2}},
 \label{e:omega_phi_IPF 9}  \\
    \|
        \omega^\ep_x
    \|_{L^{\infty}(\Omega)}
        & \spacer \hspace{-2mm}
        \leq \kap
            \| \omega^\ep_x \|_{L^{2}(\Omega)}^{\frac{1}{2}}
            \| \omega^\ep_x \|_{H^{1}(\Omega)}^{\frac{1}{2}}
        \leq
            \kap
            \| \omega^\ep_x \|_{L^{2}(\Omega)}
            +
            \kap
            \| \omega^\ep_x \|_{L^{2}(\Omega)}^{\frac{1}{2}}
            \| \nabla \omega^\ep_x \|_{L^{2}(\Omega)}^{\frac{1}{2}} \nonumber\\
        & \qquad \qquad \qquad \qquad \qquad \hspace{-2mm}
        \leq
            \kap \ep^{\frac{1}{4}}
            +
            \kap \ep^{\frac{1}{8}}
            \| \nabla \omega^\ep_x \|_{L^{2}(\Omega)}^{\frac{1}{2}}.
\label{e:omega_phi_IPF 11}
\end{align}

Putting together these estimates finally gives the following bound for the
last term on the right-hand side of (\ref{e:omega_phi_IPF 1}):
\begin{equation}\label{e:omega_phi_IPF 12}
    \begin{array}{l}
            \spacer
            \Big\|
                -\dfrac{1}{r} \, u^\eps_{\phi} + \dfrac{\pa u^\eps_{\phi}}{\pa r}
            \Big\|_{L^{\infty}(\Omega)}
            \| \omega^\ep_r \|_{L^2(\Omega)}
            \| \omega^\ep_{\phi} \|_{L^2(\Omega)}\\
                    \spacer
                    \leq \kap_T \Big\{
                                            1 +
                                            \ep^{\frac{3}{8} } \| \nabla v^\ep_{\phi} \|_{L^{2}(\Omega)}^{\frac{1}{2}}
                                            + (1 + t^{-\frac{1}{2}}) \, \ep^{- \frac{1}{2}} + \ep^{\frac{1}{4}}
                                            + \ep^{\frac{1}{8} } \| \nabla \omega^\ep_{x} \|_{L^{2}(\Omega)}^{\frac{1}{2}}
                                        \Big\} \,
                                            \ep^{\frac{3}{4}}
                                            \,
                                            \| \omega^\ep_{\phi} \|_{L^2(\Omega)}\\
                    \spacer
                    \leq
                            \kap_T (1 + t^{- \frac{1}{2}}) \| \omega^\ep_{\phi} \|_{L^2(\Omega)}^2
                            + \kap_T
                            \Big(
                            t^{- \frac{1}{2}}
                            + \ep^{\frac{3}{2}}
                            + \ep^{\frac{7}{4} } \| \nabla v^\ep_{\phi} \|_{L^{2}(\Omega)}
                            + \ep^{\frac{5}{4} } \| \nabla \omega^\ep_{x} \|_{L^{2}(\Omega)}
                            \Big)
                            \ep^{\frac{1}{2}} \\
                    \leq
                            \kap_T (1 + t^{- \frac{1}{2}}) \| \omega^\ep_{\phi} \|_{L^2(\Omega)}^2
                            + \kap_T (1 + t^{- \frac{1}{2}})
                                \ep^{\frac{1}{2}}
                            + \kap_T \ep^{\frac{7}{2}} \| \nabla v^\ep_{\phi} \|_{L^{2}(\Omega)}^2
                            + \kap_T \ep^{\frac{5}{2}} \| \nabla \omega^\ep_{x} \|_{L^{2}(\Omega)}^2.
    \end{array}
\end{equation}

Combining all previous bounds, we deduce from (\ref{e:omega_phi_IPF 1}) that
\begin{equation}\label{e:omega_phi_IPF 13}
    \begin{array}{rl}
        \spacer
            \dfrac{d}{dt} \| \omega^\ep_{\phi} \|_{L^2(\Omega)}^2
            +
            \ep \| \nabla \omega^\ep_{\phi} \|_{L^2(\Omega)}^2
        \spacer
            & \hspace{-2mm}
            \leq
            \kap_T (1 + t^{- \frac{1}{2}}) \| \omega^\ep_{\phi} \|_{L^2(\Omega)}^2
                            + \kap_T (1 + t^{- \frac{1}{2}})
                                \ep^{\frac{1}{2}}\\
            &
                            + \kap \ep \Big\|\dfrac{\pa^2 \Tht_{x}}{\pa r^2}\Big\|_{L^2(\Omega)}^2
                            + \kap_T \ep^{\frac{7}{2}} \| \nabla v^\ep_{\phi} \|_{L^{2}(\Omega)}^2
                            + \kap_T \ep^{\frac{5}{2}} \| \nabla \omega^\ep_{x} \|_{L^{2}(\Omega)}^2,
    \end{array}
\end{equation}
 from which (\ref{e:conv_IPF})$_2$  follows by applying Gr\"onwall's Lemma once
again, since
\begin{equation*}\label{e:omega_phi_IPF 14}
    \kap_T
    \int_0^T
    \Big(
    	\ep \Big\|\dfrac{\pa^2 \Tht_{x}}{\pa r^2}\Big\|_{L^2(\Omega)}^2
    	+
        \ep^{\frac{7}{2}}
        \| \nabla v^\ep_{\phi} \|_{L^{2}(\Omega)}^2
        +
        \ep^{\frac{5}{2}}
        \| \nabla \omega^\ep_{x} \|_{L^{2}(\Omega)}^2
    \Big) \, dt
            \leq \kap_T \ep^{\frac{1}{2}}.
\end{equation*}
This last inequality in turns follows from the estimates on the correctors and
the bounds already established on $\blds{v}^\ep$ and $\blds{\omega}^\ep$.

As for the case of channel flows, the bounds on the vorticity in
(\ref{e:conv_IPF}) imply that:
$$\norm{\curl(\uu^\eps - \uu^0)}_{L^\iny(0, T; L^1(\Omega)} \le \kappa_T.$$
Then, weak convergence of the vorticity with accumulation at the boundary as
a vortex sheet as in  \cref{e:conv_PPF_measure} follows again
\cref{C:ConvKelMeasure} in the Appendix.

The proof of Theorem  \ref{T:IPF} is complete.

\appendix
\Ignore{ 
\setcounter{section}{0}
\renewcommand\thesection{\Alph{section}}
\setcounter{equation}{0}
\renewcommand{\theequation}{A.\arabic{equation}}
\newtheorem{Alem}{Lemma}\renewcommand{\theAlem}{A.\arabic{Alem}}
\newtheorem{Arem}{Remark}\renewcommand{\theArem}{A.\arabic{Arem}}
} 

\section{One-dimensional heat equations}\label{S.App}

In this Appendix we discuss mostly known results on 1D and 2D heat equations
with posible drift. These results, in turn, are used throughtout the paper to
derive decay and regularity estimates for the boundary layer correctors. In
fact, due to the weakly non-linear nature of the flows considered here, the
corrector can be taken to be linear (cf. the approach using layer potentials
for a heat equation with drift in \cite{MT08}).

\subsection{On the one-dimensional heat equation with small
diffusivity}\label{S.App_heat}

We consider the following boundary-value problem for the heat
equation on a half line:
\begin{equation}\label{e:HEAT_eq}
\left\{
        \begin{array}{rl}
                \spacer
                        \dfrac{\pa \Phi}{\pa t} - \ep \dfrac{\pa^2 \Phi}{\pa \eta^2} = 0,
                        & \eta, \,t > 0,\\
                \spacer
                        \Phi = g(t) ,
                        & \eta = 0,\\
                \spacer
                        \Phi \rightarrow 0,
                        & \text{as } \eta \rightarrow \infty,\\
                        \Phi = 0,
                        & t= 0.
        \end{array}
\right.
\end{equation}
Above, the diffusivity $\ep$ is a fixed, strictly positive parameter, and $g(t)$
is the boundary data, assumed sufficiently smooth. The
incompatibility between the boundary data, which need not vanish at $t=0$, and
the initial data leads to the formation of
an initial layer.

The solution of (\ref{e:HEAT_eq}) is explicitly given by the following
formula (see e.g. the classical reference \cite{Can84}):
\begin{equation}\label{e:Phi}
\begin{array}{rl}
\Phi (\eta, t)
            =
            & \displaystyle \spacer \hspace{-2mm}
                2 \, g(0) \, \erfc\Big( \dfrac{\eta}{\sqrt{2 \ep t}} \Big)
               + 2 \int_0^{t} \dfrac{\pa g}{\pa t} (s) \, \erfc \Big( \dfrac{\eta}{\sqrt{2 \ep (t-s)}} \Big) \, ds,
\end{array}
\end{equation}
where $\erfc$ is the complimentary error function on $\mathbb{R}_+$,
\begin{equation}\label{e:erfc}
\erfc (z) := \dfrac{1}{\sqrt{2\pi}} \int_{z}^{\infty} e^{-y^2/2} \, dy,
\end{equation}
which satisfies
\begin{equation*}\label{e:erfc at 0 and infty}
    \erfc (0) = \dfrac{1}{2},
\quad
    \erfc (\infty) = 0.
\end{equation*}


We recall known estimates on $\Phi$ (for a proof in the context of boundary
layer analysis, we refer to \cite{GHT10, Gie13}).

\begin{lemma}\label{L:pointwise}
    Let $g \in W^{1,  \infty}\big(0, T \big)$, $0<T<\infty$. Then,
    the following  pointwise estimates hold  for $\eta>0$ and $0< t <T$,
        \begin{equation*}\label{e:HEAT_pointwise}
                \left\{
                        \begin{array}{l}
                                \spacer
                                        | \Phi(t,\eta) |
                                            \leq \kap_T \, e^{- \eta^2 / (4 \ep t)}, \\
                                \spacer
                                        \Big| \dfrac{\pa \Phi}{\pa \eta}(t,\eta)
\Big|
                                            \leq \kap_T \, \ep^{-\frac{1}{2}} (1 + t^{-\frac{1}{2}})e^{- \eta^2 / (4 \ep t)}, \\
                                \displaystyle
                                        \Big| \dfrac{\pa^2
\Phi}{\pa \eta^2}(t,\eta) \Big|
                                            \leq \kap_T \, ( \ep  t)^{-1}
\dfrac{\eta}{\sqrt{\ep t}} \, e^{- \eta^2 / (4 \ep t)}
                                                + \kap_T \, \ep^{-1} \int_0^t s^{-1} \dfrac{\eta}{\sqrt{\ep s}} \, e^{- \eta^2 / (4 \ep s)} \, ds,
                        \end{array}
                \right.
        \end{equation*}
where $\kap_T$ depends  on $T$ and the data $g$, but is independent of
$\ep$.
\end{lemma}

\begin{lemma}\label{L:L2}
 Assume again that $g \in W^{1,  \infty}\big(0, T \big)$, $0<T<\infty$. Then,
for $1 \leq
p \leq \infty$ and   $0 \leq m \leq 2$,
        \begin{equation*}\label{e:HEAT_L2}
                                        \Big\| \dfrac{\pa^m \Phi}{\pa
\eta^m}(t)) \Big\|_{L^p_\eta(0, \infty)}
                                            \leq \kap_T \, (1 + t^{\frac{1}{2 p}
- \frac{m}{2}}) \, \ep^{\frac{1}{2 p} - \frac{m}{2}} ,
        \end{equation*}
for $0< t<T$, with a constant $\kap_T$ depending on $T$ and the data $g$, but
independent of
$\ep$.
\end{lemma}

The next result is utilized in particular in establishing concentration of
vorticity at the boundary in the vanishing viscosity limit. We include a
proof for the reader's convenience.

\begin{lemma}\label{L:erfcDeltaFunction}
Under the hypotheses of Lemma \ref{L:pointwise} and \ref{L:L2},
\[
            \pdx{\Phi(\cdot, t)}{\eta} \, \underset{\eps\to 0}{\rightharpoonup}
\, 2 g(t) \delta_0,
\]
weakly-$\ast$ in the space of Radon measures on $\RR$,  pointwise in
$0<t<T$. That is,
    \begin{align*}
        \lim_{\eps \to 0} \bigppr{\pdx{\Phi(\cdot, t)}{\eta},
                \varphi}_{L^2_\eta(\RR)}
            = 2 g(t) \varphi(0),
    \end{align*}
for all $\varphi \in C_C(\R)$, the space of continuous, compactly supported
functions on $\R$.
\end{lemma}

\begin{proof}
We observe that we can write $\pdx{\Phi}{\eta}$ as
\begin{align}\label{e:pdxPhi}
    \pdx{\Phi(\eta, t)}{\eta}
        = 2 g(0) K_\eps(\eta, t)
            + 2 \int_0^t \pdx{g(s)}{s} K_\eps(\eta, t - s) \, ds,
\end{align}
where
\begin{equation}\label{e:K}
        K_\eps(\eta, t)
                        := \dfrac{\pa \, \erfc(\eta/ \sqrt{2 \ep t})}{\pa \eta}
                        = - \dfrac{1}{2 \sqrt{\pi}} \dfrac{1}{ \sqrt{\ep t}} \,
e^{- \eta^2 / (4 \ep t)}.
\end{equation}
Since  the family $\{K_\eps(\cdot, t)\}_{\eps}$ is a classical approximation
of the identity in the variable $\eta$, the first term on the right-hand side of
(\ref{e:K}) converges in the sense of distributions, and also weakly-$\ast$ in
the space of Radon measures, to $2 g(0)\varphi(0)$. It is therefore enough to
identify the weak-$\ast$ limit of the second term on the right-hand side of
(\ref{e:K}) as  $2 \pr{g(t) - g(0)} \varphi(0)$, as $\eps\to 0$.

To this end, we fix $\varphi \in C_C(\RR)$, and observe that the regularity
of $g$ implies
\begin{align*}
    \bigppr{2 \int_0^t &\pdx{g(s)}{s} K_\eps(\eta, t - s) \, ds,
                \varphi(\eta)}_{L^2_\eta(\Omega)}
        = 2 \int_\R \int_0^t \pdx{g(s)}{s} K_\eps(\eta, t - s) \, ds
                \, \varphi(\eta) \, d \eta \\
        &= 2 \int_0^t \pdx{g(s)}{s} \int_\R K_\eps(\eta, t - s)
                \varphi(\eta) \, d \eta \, ds.
\end{align*}
We can conclude the proof if we can bring the limit $\eps\to 0$ inside the
integrals. To justify this step, we first assume that $\varphi \in C^1_C(\R)$
and  note that, for each $\eps >0$,
\begin{align*}
    &\abs{\int_\R K_\eps(\eta, t - s) \varphi(\eta) \, d \eta}
        = {\int_\R \pdx{\erfc(\eta/\sqrt{2 \eps (t - s)})}{\eta} \varphi(\eta) \, d \eta} \\
        &\qquad
        = {\int_\R \erfc\pr{\frac{\eta}{\sqrt{2 \eps (t - s)}}} \varphi'(\eta) \, d \eta}
        \le \int_\R \abs{\varphi'(\eta)} \, d \eta
        \le \kappa_\varphi,
\end{align*}
since $\abs{\erfc} \le 1$. hence, we are justified in writing:
\begin{align*}
    \lim_{\eps \to 0}
        \bigppr{2 \int_0^t &\pdx{g(s)}{s} K_\eps(\eta, t - s) \, ds,
                \varphi(\eta)}_{L^2_\eta(\Omega)}
        = 2 \int_0^t \pdx{g(s)}{s} \lim_{\eps \to 0} \int_\R K_\eps(\eta, t -s)
                \varphi(\eta) \, d \eta \, ds \\
        &= 2 \pr{g(t) - g(0)} \varphi(0),
\end{align*}
which gives the desired result for $\varphi \in C^1_C(\R)$.
Using the density of
$C^1_C(\R)$ in $C_C(\R)$
then completes the proof.
\end{proof}

\subsection{On a drift-diffusion equation with small
diffusivity}\label{S.App_heat-like}

In this subsection, we derive various estimates for a drift-diffusion equation
in a periodic channel, uniformly in the diffusivity $\eps$.

We consider the following initial-boundary value problem:
\begin{equation}\label{e:HEAT_like_eq}
\left\{
        \begin{array}{rl}
                \spacer
                        \dfrac{\pa \Psi}{\pa t}
                        - \ep \dfrac{\pa^2 \Psi}{\pa \tau^2} - \ep \dfrac{\pa^2 \Psi}{\pa \eta^2}
                        + \mathcal{U}(\eta,t) \dfrac{\pa \Psi}{\pa \tau}
                        =
                        G(\tau, \eta, t),
                        &
                        0 < \tau < L_{\tau}, \,
                        \eta, \,t > 0,\\
                \spacer
			\Psi \text{ is periodic} & \hspace{-2mm} \text{ in } \tau \text{ with period } L_{\tau},\\
                \spacer
                        \Psi = g(\tau, t) ,
                        & \eta = 0,\\
                \spacer
                        \Psi \rightarrow 0,
                        & \text{as } \eta \rightarrow \infty,\\
                        \Psi = 0,
                        & t= 0.
        \end{array}
\right.
\end{equation}
Above, $\ep>0$ represents again the diffusivity,
and $\mathcal{U}(\eta)$, $G(\tau, \eta, t)$, and $g(\tau, t)$ are sufficiently
smooth data in the indicated variables such that, for a given $0<T<\infty$ and
for all $0<t<T$,
\begin{equation}\label{e:heat_like_data}
\left\{
	\begin{array}{l}
	\spacer
	\mathcal{U}|_{\eta = 0} = 0,
		\qquad
	\|
	G |_{ \eta=0}
	\|_{L^{\infty}( (0, L_{\tau})\times (0, T) )}
		\leq \kap_T, \\
	\spacer
	\bigg\|
		\dfrac{\pa^{m} \mathcal{U}}{\pa \eta^m}(\cdot,t)
	\bigg\|_{ L^p(0, \infty)}
		\leq
			\kap_T \, (1 + t^{\frac{1}{2 p} - \frac{m}{2}}) \, (1 + \ep^{\frac{1}{2p} - \frac{m}{2}}),\\
	\bigg\|
		\dfrac{\pa^{k+m} G}{\pa \tau^k \pa \eta^m} (\cdot,\cdot,t)
	\bigg\|_{L^p((0, L_{\tau})\times(0, \infty))}
		\leq \kap_T \, (1 + t^{\frac{1}{2 p} - \frac{m}{2}}) \, \ep^{\frac{1}{2p} - \frac{m}{2}},
	\end{array}
\right.
\end{equation}
for $1 \leq p \leq \infty$, $k \geq 0$, and $0 \leq m \leq 2$.
As before, the boundary data is assumed incompatible with the initial
condition in the sense that $g(\tau, 0)$ may not necessarily vanish.

To estimate $\Psi$ solution of (\ref{e:HEAT_like_eq}), we will utilize the
solution  of (\ref{e:HEAT_eq}) with $g(t)$ replaced by
$g(\tau, t)$, denoted by $\Phi(t,\tau,\eta)$.  Thanks to Lemma \ref{L:L2}, we
have that
\begin{equation}\label{e:heat_tan_Lp}
                                        	\Big\|
                                        		\dfrac{\pa^{k+m}
\Phi}{\pa \tau^k \pa \eta^m}(t,\cdot,\cdot)
					\Big\|_{L^p((0, L_{\tau}) \times (0, \infty))}
                                            \leq \kap_T \, (1 + t^{\frac{1}{2 p} - \frac{m}{2}}) \, \ep^{\frac{1}{2 p} - \frac{m}{2}} ,
                                        \text{ }
					1 \leq p \leq \infty,
					\text{ }
					k \geq 0,
					\text{ }
					0 \leq m \leq 2,
\end{equation}
which, in turn, yields the following estimates for $\Psi$.

\begin{lemma}\label{l:Lp_est_heat_like_sol}
Assuming that the data $\mathcal{U}$, $G$, and $g$ satisfy
(\ref{e:heat_like_data}) and are sufficiently regular, we have for all $k\in
\ZZ_+$,
$1 < p \leq \infty$,
\begin{equation}\label{e:Lp_heat_like_sol}
\left\{
\begin{array}{l}
	\spacer
	\Big\| \dfrac{\pa^{k} \Psi}{\pa \tau^k} \Big\|_{L^{\infty}(0, T; L^p((0, L_{\tau}) \times (0, \infty)))}
	+
	\ep^{\frac{1}{2p} + \frac{1}{4}}
	\Big\|
		\nabla \dfrac{\pa^{k} \Psi}{\pa \tau^k}
	\Big\|_{L^{2}(0, T; L^2((0, L_{\tau}) \times (0, \infty)))}
                                                                \leq
                                                                        \kap_T \ep^{\frac{1}{2p}},\\
         \spacer
	\Big\| \dfrac{\pa^{k+1} \Psi}{\pa \tau^k \pa \eta} \Big\|_{L^{\infty}(0, T; L^2((0, L_{\tau}) \times (0, \infty)))}
	+
	\ep^{\frac{1}{2}}
	\Big\|
		\nabla \dfrac{\pa^{k+1} \Psi}{\pa \tau^k \pa \eta}
	\Big\|_{L^{2}(0, T; L^2((0, L_{\tau}) \times (0, \infty)))}
                                                                \leq
                                                                        \kap_T \ep^{-\frac{1}{4}},\\
	\Big\| \dfrac{\pa^{k+1} \Psi}{\pa \tau^k \pa \eta} \Big\|_{L^{\infty}(0, T; L^1((0, L_{\tau}) \times (0, \infty)))}
	\leq \kap_T,
\end{array}
\right.
\end{equation}
with $k \geq 0$,
for a constant $\kap_T$ depending on $T$ and the data, but independent of
$\ep$.
\end{lemma}

\begin{proof}
We denote:
\begin{equation}\label{e:difference_heat_like}
\widetilde{\Psi}
	:= \Psi - \Phi,
\end{equation}
and observe that $\widetilde{\Psi}$ satisfies the following inital-boundary
value problem:
\begin{equation}\label{e:difference_heat_like_eq}
\left\{
        \begin{array}{rl}
                \spacer
                        \dfrac{\pa \widetilde{\Psi}}{\pa t}
                        - \ep \dfrac{\pa^2 \widetilde{\Psi}}{\pa \tau^2}
                        - \ep \dfrac{\pa^2 \widetilde{\Psi}}{\pa \eta^2}
                        + \mathcal{U}(\eta) \dfrac{\pa \widetilde{\Psi}}{\pa \tau}
                        =
                        G
                        +
                        \ep \dfrac{\pa^2 \Phi}{\pa \tau^2}
                        -
                        \mathcal{U}(\eta) \dfrac{\pa \Phi}{\pa \tau},
                        &
                        0 < \tau < L_{\tau}, \,
                        \eta, \,t > 0,\\
                \spacer
			\widetilde{\Psi} \text{ is periodic} & \hspace{-2mm} \text{ in } \tau \text{ with period } L_{\tau},\\
                \spacer
                        \widetilde{\Psi} = 0 ,
                        & \eta = 0,\\
                \spacer
                        \widetilde{\Psi} \rightarrow 0,
                        & \text{as } \eta \rightarrow \infty,\\
                        \widetilde{\Psi} = 0,
                        & t= 0.
        \end{array}
\right.
\end{equation}

To prove (\ref{e:Lp_heat_like_sol})$_1$,
we multiply (\ref{e:difference_heat_like_eq})$_1$ by $\wi{\Psi}^{p-1}$
where $p>1$ is a simple fraction $q/r$ with an even integer $q$.
Then, by integrating over $(0, L_{\tau}) \times (0, \infty)$, we find that
\begin{equation}\label{e:energy_difference_1}
\begin{array}{l}
	\spacer\displaystyle
	\dfrac{1}{p} \dfrac{d}{dt} \| \wi{\Psi} \|_{L^p((0, L_{\tau})\times(0, \infty))}^p
	+
	\ep (p-1)
	\int_{0}^{\infty} \int_{0}^{L_{\tau}}
		|\nabla \wi{\Psi}|^2 \wi{\Psi}^{p-2}
	\, d\tau d\eta \\
		\qquad \qquad \qquad
		\spacer \displaystyle
			=
				\int_{0}^{\infty} \int_{0}^{L_{\tau}}
					\bigg(
					G
					+
					\ep \dfrac{\pa^2 \Phi}{\pa \tau^2}
                        			-
                        			\mathcal{U}(\eta) \dfrac{\pa \Phi}{\pa \tau}
					\bigg)
					\wi{\Psi}^{p-1}
				\, d\tau d\eta\\
		\qquad \qquad \qquad
		\spacer \displaystyle
			\leq
				\bigg\{
				\int_{0}^{\infty} \int_{0}^{L_{\tau}}
					\bigg(
					G
					+
					\ep \dfrac{\pa^2 \Phi}{\pa \tau^2}
                        			-
                        			\mathcal{U}(\eta) \dfrac{\pa \Phi}{\pa \tau}
					\bigg)^{p}
				\, d\tau d\eta
				\bigg\}^{\frac{1}{p}}
				\bigg\{
				\int_{0}^{\infty} \int_{0}^{L_{\tau}}
					\wi{\Psi}^p
				\, d\tau d\eta
				\bigg\}^{\frac{p-1}{p}}\\
		\qquad \qquad \qquad
		\displaystyle
			\leq
				\kap \bigg\|
						G
						+
						\ep \dfrac{\pa^2 \Phi}{\pa \tau^2}
                        				-
                        				\mathcal{U}(\eta) \dfrac{\pa \Phi}{\pa \tau}
					\bigg\|_{L^p((0, L_{\tau})\times(0, \infty))}^p
				+
				\kap	\|
						\wi{\Psi}
					\|_{L^p((0, L_{\tau})\times(0, \infty))}^p.

\end{array}
\end{equation}
The hypotheses on the data (\ref{e:heat_like_data}) and
the bounds (\ref{e:heat_tan_Lp}) on $\widetilde{\Psi}$ then give:
\begin{equation}\label{e:energy_difference_2}
	\dfrac{1}{p} \dfrac{d}{dt} \| \wi{\Psi} \|_{L^p((0, L_{\tau})\times(0, \infty))}^p
	+
	\ep (p-1)
	\|
		\wi{\Psi}^{\frac{p-2}{2}} \,
		\nabla \wi{\Psi}
	\|_{L^2((0, L_{\tau})\times(0, \infty))}^2
			\leq
				\kap \ep^{\frac{1}{2}}
				+
				\kap	\|
						\wi{\Psi}
					\|_{L^p((0, L_{\tau})\times(0, \infty))}^p,
\end{equation}
which implies that
\begin{equation}\label{e:energy_difference_3}
	\left\{
	\begin{array}{l}
		\spacer
		\| \Psi \|_{L^{\infty}(0, T; L^p((0, L_{\tau})\times(0, \infty)))}
			\leq
				\kap_T \ep^{\frac{1}{2p}}, \\
		\| \nabla \Psi \|_{L^{2}(0, T; L^2((0, L_{\tau})\times(0, \infty)))}
			\leq
				\kap_T \ep^{- \frac{1}{4}}.
	\end{array}
	\right.
\end{equation}
given the continuity of the $L^p$ norm in $p$,
the estimate (\ref{e:energy_difference_3})$_1$ is valid for any $1 < p \leq \infty$.

Next, since  any tangential derivative of  satisfies an
equation similar to (\ref{e:difference_heat_like_eq})),
one can verify (\ref{e:Lp_heat_like_sol})$_1$ for  $k > 0$ in an analogous
manner.

To show (\ref{e:Lp_heat_like_sol})$_2$, using the regularity on the data and
$\Phi$, we restrict the equation for $\widetilde{\Psi}$ to $\eta = 0$ and find
that:
\begin{equation}\label{e:BC_Neumann_difference}
	- \ep \dfrac{\pa^2 \wi{\Psi}}{\pa \eta^2}
		=
			G
			- \ep \dfrac{\pa^2 g}{\pa \tau^2},
	\quad
	\text{at }
	\eta=0,
\end{equation}
using also that that $\Phi|_{\eta = 0} = g$.

Hence, after differentiating (\ref{e:difference_heat_like_eq}) in $\eta$, we
obtain an equation for $\pa \wi{\Psi}/ \pa \eta$ supplemented by a Neumann
boundary condition:
\begin{equation}\label{e:difference_heat_like_eq_H1}
\left\{
        \begin{array}{l}
                \spacer
                        \dfrac{\pa }{\pa t} \bigg(  \dfrac{\pa \widetilde{\Psi}}{\pa \eta} \bigg)
                        - \ep \dfrac{\pa^2 }{\pa \tau^2} \bigg(  \dfrac{\pa \widetilde{\Psi}}{\pa \eta} \bigg)
                        - \ep \dfrac{\pa^2 }{\pa \eta^2} \bigg(  \dfrac{\pa \widetilde{\Psi}}{\pa \eta} \bigg)
                        + \mathcal{U} \dfrac{\pa }{\pa \tau} \bigg(  \dfrac{\pa \widetilde{\Psi}}{\pa \eta} \bigg)\\
                        \qquad \qquad
                        =
                        -
                        \dfrac{\pa \mathcal{U}}{\pa \eta} \, \widetilde{\Psi}
                        +
                        \dfrac{\pa}{\pa \eta}
                        \bigg(
                        G
                        +
                        \ep \dfrac{\pa^2 \Phi}{\pa \tau^2}
                        -
                        \mathcal{U} \dfrac{\pa \Phi}{\pa \tau}
                        \bigg),
                        \quad
                        0 < \tau < L_{\tau}, \,
                        \eta, \,t > 0,\\
                \spacer
			\dfrac{\pa \widetilde{\Psi}}{\pa \eta} \text{ is periodic  in } \tau \text{ with period } L_{\tau},\\
                \spacer
                        \dfrac{\pa^2 \widetilde{\Psi}}{\pa \eta^2}
                        =
                        - \ep^{-1} G
			+
			\dfrac{\pa^2 g}{\pa \tau^2} ,
                        \quad \eta = 0,\\
                \spacer
                        \dfrac{\pa \widetilde{\Psi}}{\pa \eta} \rightarrow 0,
                        \quad \text{as } \eta \rightarrow \infty,\\
                        \dfrac{\pa \widetilde{\Psi}}{\pa \eta} = 0,
                        \quad t= 0.
        \end{array}
\right.
\end{equation}
We next multiply (\ref{e:difference_heat_like_eq_H1})$_1$ by $\pa \wi{\Psi}/ \pa
\eta$ and
integrate over $(0, L_{\tau}) \times (0, \infty)$:
\begin{equation}\label{e:energy_difference_H1_1}
\begin{array}{l}
	\spacer \displaystyle
	\dfrac{1}{2} \dfrac{d}{dt}
	\bigg\| \dfrac{\pa \widetilde{\Psi}}{\pa \eta} \bigg\|_{L^p((0, L_{\tau})\times(0, \infty))}^2
	+
	\ep
	\bigg\| \nabla \dfrac{\pa \widetilde{\Psi}}{\pa \eta} \bigg\|_{L^p((0,
L_{\tau})\times(0, \infty))}^2
			=
				\ep \int_{0}^{L_{\tau}}
					\bigg[
					\dfrac{\pa^2 \widetilde{\Psi}}{\pa \eta^2}
					\dfrac{\pa \widetilde{\Psi}}{\pa \eta}
					\bigg]_{\eta = 0}
				d\tau\\
		\qquad \qquad \qquad \quad
		\displaystyle
				+
				\int_{0}^{\infty} \int_{0}^{L_{\tau}}
					\bigg(
					-
                        			\dfrac{\pa \mathcal{U}}{\pa \eta} \, \widetilde{\Psi}
                        			+
                        			\dfrac{\pa}{\pa \eta}
                        			\bigg(
                        			G
                        			+
                        			\ep \dfrac{\pa^2 \Phi}{\pa \tau^2}
                        			-
                        			\mathcal{U} \dfrac{\pa \Phi}{\pa \tau}
                        			\bigg)
					\bigg)
					\dfrac{\pa \widetilde{\Psi}}{\pa \eta}
				\, d\tau d\eta.
\end{array}
\end{equation}

Using  Lemma \ref{L:Trace}, we estimate the first
term on the right-hand side of (\ref{e:energy_difference_H1_1}) as follows:
\begin{equation}\label{e:energy_difference_H1_2}
	\begin{array}{l}
		\displaystyle \spacer
		\ep
		\bigg|
		\int_{0}^{L_{\tau}}
					\bigg[
					\dfrac{\pa^2 \widetilde{\Psi}}{\pa \eta^2}
					\dfrac{\pa \widetilde{\Psi}}{\pa \eta}
					\bigg]_{\eta = 0}
					d\tau \bigg|
					\leq
					\bigg|
					\int_{0}^{L_{\tau}}
					\bigg[
					\Big(
						-G + \ep \dfrac{\pa^2 g}{\pa \tau^2}
					\Big)
					\dfrac{\pa \widetilde{\Psi}}{\pa \eta}
					\bigg]_{\eta = 0}
					d\tau \bigg|\\
				\quad
					\spacer \displaystyle
					\leq
					\kap \Big\|
						\dfrac{\pa \widetilde{\Psi}}{\pa
\eta}\lfloor_{\eta=0}
						\Big\|_{L^2(0, L_{\tau})} \leq
					\kap \Big\|
						\dfrac{\pa \widetilde{\Psi}}{\pa \eta}
						\Big\|_{L^2((0, L_{\tau}) \times (0, \infty))}^{\frac{1}{2}}
						\Big\|
						\dfrac{\pa \widetilde{\Psi}}{\pa \eta}
						\Big\|_{H^1((0, L_{\tau}) \times (0, \infty))}^{\frac{1}{2}}\\
				\quad
					\spacer \displaystyle
					\leq
					\kap \Big\|
						\dfrac{\pa \widetilde{\Psi}}{\pa \eta}
						\Big\|_{L^2((0, L_{\tau}) \times (0, \infty))}
						+
					\kap
						\Big\|
						\dfrac{\pa \widetilde{\Psi}}{\pa \eta}
						\Big\|_{L^2((0, L_{\tau}) \times (0, \infty))}^{\frac{1}{2}}
						\Big\|
						\nabla
						\dfrac{\pa \widetilde{\Psi}}{\pa \eta}
						\Big\|_{L^2((0, L_{\tau}) \times (0, \infty))}^{\frac{1}{2}}\\
				\quad
					\spacer \displaystyle
					\leq
					\kap
					+
					\Big\|
						\dfrac{\pa \widetilde{\Psi}}{\pa \eta}
					\Big\|_{L^2((0, L_{\tau}) \times (0, \infty))}^2
					+
					\kap \ep^{- \frac{1}{2}}
					+
					\kap \ep^{\frac{1}{2}}
					\Big\|
						\dfrac{\pa \widetilde{\Psi}}{\pa \eta}
					\Big\|_{L^2((0, L_{\tau}) \times (0, \infty))}
					\Big\|
						\nabla
						\dfrac{\pa \widetilde{\Psi}}{\pa \eta}
					\Big\|_{L^2((0, L_{\tau}) \times (0, \infty))}\\
				\quad
					\displaystyle
					\leq
					\kap \ep^{- \frac{1}{2}}
					+
					\Big\|
						\dfrac{\pa \widetilde{\Psi}}{\pa \eta}
					\Big\|_{L^2((0, L_{\tau}) \times (0, \infty))}^2
					+
					\dfrac{1}{2}
					\ep
					\Big\|
						\nabla
						\dfrac{\pa \widetilde{\Psi}}{\pa \eta}
					\Big\|_{L^2((0, L_{\tau}) \times (0, \infty))}^2.
	\end{array}
\end{equation}

We can then estimate the second term on the right-hand side of
(\ref{e:energy_difference_H1_1}) as follows:
\begin{equation}\label{e:energy_difference_H1_3}
\begin{array}{l}
	\spacer \displaystyle
		\bigg|
		\int_{0}^{\infty} \int_{0}^{L_{\tau}}
					\bigg(
					-
                        			\dfrac{\pa \mathcal{U}}{\pa \eta} \, \widetilde{\Psi}
                        			+
                        			\dfrac{\pa}{\pa \eta}
                        			\bigg(
                        			G
                        			+
                        			\ep \dfrac{\pa^2 \Phi}{\pa \tau^2}
                        			-
                        			\mathcal{U} \dfrac{\pa \Phi}{\pa \tau}
                        			\bigg)
					\bigg)
					\dfrac{\pa \widetilde{\Psi}}{\pa \eta}
				\, d\tau d\eta \bigg|\\
	\quad
	\leq
		\begin{array}{rl}
			\left\{
				\begin{array}{l}
					\spacer
					\Big\| \dfrac{\pa \mathcal{U}}{\pa \eta} \Big\|_{L^{\infty}((0, L_{\tau}) \times (0, \infty))}
					\| \widetilde{\Psi} \|_{L^2((0, L_{\tau}) \times (0, \infty))}\\
					\spacer
					+
					\Big\| \dfrac{\pa G}{\pa \eta} \Big\|_{L^2((0, L_{\tau}) \times (0, \infty))}
					+
					\ep
					\Big\| \dfrac{\pa^3 \Phi}{\pa \eta \pa \tau^2} \Big\|_{L^2((0, L_{\tau}) \times (0, \infty))}\\
					\spacer
					+
					\Big\| \dfrac{\pa \mathcal{U}}{\pa \eta} \Big\|_{L^{\infty}((0, L_{\tau}) \times (0, \infty))}
					\Big\| \dfrac{\pa \Phi}{\pa \tau} \Big\|_{L^2((0, L_{\tau}) \times (0, \infty))}\\
					+
					\| \mathcal{U} \|_{L^{\infty}((0, L_{\tau}) \times (0, \infty))}
					\Big\| \dfrac{\pa^2 \Phi}{\pa \eta \pa \tau} \Big\|_{L^2((0, L_{\tau}) \times (0, \infty))}
				\end{array}
			\right\}
			&
			\hspace{-2mm}
			\Big\|
			\dfrac{\pa \widetilde{\Psi}}{\pa \eta}
			\Big\|_{L^2((0, L_{\tau}) \times (0, \infty))} \spacer
		\end{array}\\
	\quad
	\spacer
	\leq
		\kap_T(1 + t^{-\frac{1}{2}}) \, \ep^{-\frac{1}{4}}
		\Big\|
			\dfrac{\pa \widetilde{\Psi}}{\pa \eta}
		\Big\|_{L^2((0, L_{\tau}) \times (0, \infty))}\\
	\quad
	\leq
		\kap_T(1 + t^{-\frac{1}{2}}) \, \ep^{-\frac{1}{2}}
		+
		\kap_T(1 + t^{-\frac{1}{2}})
		\Big\|
			\dfrac{\pa \widetilde{\Psi}}{\pa \eta}
		\Big\|_{L^2((0, L_{\tau}) \times (0, \infty))}^2.
\end{array}
\end{equation}

Combining (\ref{e:energy_difference_H1_1})-(\ref{e:energy_difference_H1_3}), we find that
\begin{equation}\label{e:energy_difference_H1_4}
\begin{array}{l}
	\spacer \displaystyle
	\dfrac{d}{dt}
	\bigg\| \dfrac{\pa \widetilde{\Psi}}{\pa \eta} \bigg\|_{L^p((0, L_{\tau})\times(0, \infty))}^2
	+
	\ep
	\bigg\| \nabla \dfrac{\pa \widetilde{\Psi}}{\pa \eta} \bigg\|_{L^p((0, L_{\tau})\times(0, \infty))}^2\\		\qquad \qquad \qquad
		\spacer \displaystyle
			\leq
		\kap_T(1 + t^{-\frac{1}{2}}) \, \ep^{-\frac{1}{2}}
		+
		\kap_T(1 + t^{-\frac{1}{2}})
		\Big\|
			\dfrac{\pa \widetilde{\Psi}}{\pa \eta}
		\Big\|_{L^2((0, L_{\tau}) \times (0, \infty))}^2.
\end{array}
\end{equation}
Using the Gronwall's Lemma with the integrating factor $exp(-\kap_T t - 2 \kap_T
t^{1/2})$,
we deduce that
$$
\Big\| \dfrac{\pa  \wi{\Psi}}{  \pa \eta} \Big\|_{L^{\infty}(0, T; L^2((0, L_{\tau}) \times (0, \infty)))}
	+
	\ep^{\frac{1}{2}}
	\Big\|
		\nabla \dfrac{\pa  \wi{\Psi}}{ \pa \eta}
	\Big\|_{L^{2}(0, T; L^2((0, L_{\tau}) \times (0, \infty)))}
                                                                \leq
                                                                        \kap_T
\ep^{-\frac{1}{4}},
$$
and
(\ref{e:Lp_heat_like_sol})$_2$ with $k=0$ follows from the uniform bounds
(\ref{e:heat_tan_Lp}) on $\Phi$.
(\ref{e:Lp_heat_like_sol})$_2$ for  $k > 0$ can be verified in an analogous
manner.

To verify (\ref{e:Lp_heat_like_sol})$_3$,  we introduce a standard convex
regularization of the absolute
value, $\mathcal{F}_{\lambda}$, $\lambda>0$,  defined as
\begin{equation}\label{e:F_delta}
        \mathcal{F}_{\lambda} (x)
                = \sqrt{\lambda^2 + x^2}.
\end{equation}
and obtain:
\begin{equation}\label{e:L^1_PPF_TEMP1}
\begin{array}{l}
\displaystyle \spacer
        \dfrac{d}{dt} \int_{0}^{L_{\tau}} \int_0^{\infty}
        \mathcal{F}_{\lambda}\Big( \dfrac{\pa \wi{\Psi}}{\pa \eta} \Big)
        \, d\eta d\tau\\
\displaystyle \spacer \quad
                =
                    \int_{0}^{L_{\tau}} \int_0^{\infty}
                    \mathcal{F}^{\pri}_{\lambda}\Big( \dfrac{\pa \wi{\Psi}}{\pa \eta} \Big)
                    \dfrac{\pa}{\pa t} \Big( \dfrac{\pa \wi{\Psi}}{\pa \eta} \Big)
                    \, d\eta d\tau\\
\displaystyle \spacer \quad
        =
            \ep
            \int_{0}^{L_{\tau}} \int_0^{\infty}
                    \mathcal{F}^{\pri}_{\lambda} \Big( \dfrac{\pa \wi{\Psi}}{\pa \eta} \Big)
                    \Delta \dfrac{\pa \wi{\Psi}}{\pa \eta}
                    \, d\eta d\tau
            -
            \int_{0}^{L_{\tau}} \int_0^{\infty}
                    \mathcal{F}^{\pri}_{\lambda} \Big( \dfrac{\pa \wi{\Psi}}{\pa \eta} \Big)
                    \mathcal{U} \dfrac{\pa }{\pa \tau}\Big( \dfrac{\pa \wi{\Psi}}{\pa \eta} \Big)
                    \, d\eta d\tau\\
\displaystyle
\quad \quad
            +
            \int_{0}^{L_{\tau}} \int_0^{\infty}
                    \mathcal{F}^{\pri}_{\lambda} \Big( \dfrac{\pa \wi{\Psi}}{\pa \eta} \Big)
                    \big(  \text{right-hand side of (\ref{e:difference_heat_like_eq_H1})} \big)_1
                    \, d\eta d\tau.
\end{array}
\end{equation}

Now convexity of $\mathcal{F}_\lambda$ inplies that
\begin{equation*}\label{e:L^1_PPF_TEMP2}
\begin{array}{rl}
    \spacer
        \Delta \Big( \mathcal{F}_{\lambda} \Big( \dfrac{\pa \wi{\Psi}}{\pa \eta} \Big) \Big)
                &
                =
                    \mathcal{F}^{\pri}_{\lambda} \Big( \dfrac{\pa \wi{\Psi}}{\pa \eta} \Big)
                    \Delta \dfrac{\pa \wi{\Psi}}{\pa \eta}
                    +
                    \mathcal{F}^{\pri \pri}_{\lambda} \Big( \dfrac{\pa \wi{\Psi}}{\pa \eta} \Big)
                    \Big\{
                    \Big( \dfrac{\pa^2 \wi{\Psi}}{\pa \tau \pa \eta} \Big)^2
                    +
                    \Big( \dfrac{\pa^2 \wi{\Psi}}{\pa \eta^2} \Big)^2
                    \Big\}\\
                &
                \geq
                    \mathcal{F}^{\pri}_{\lambda} \Big( \dfrac{\pa \wi{\Psi}}{\pa \eta} \Big)
                    \Delta \dfrac{\pa \wi{\Psi}}{\pa \eta}.
\end{array}
\end{equation*}
Then, integrating by parts, the first term on the right-hand side of (\ref{e:L^1_PPF_TEMP1}) can be estimated as
\begin{equation}\label{e:L^1_PPF_TEMP3}
\begin{array}{l}
\displaystyle \spacer
    \ep
            \int_{0}^{L_{\tau}} \int_0^{\infty}
                    \mathcal{F}^{\pri}_{\lambda} \Big( \dfrac{\pa \wi{\Psi}}{\pa \eta} \Big)
                    \Delta \dfrac{\pa \wi{\Psi}}{\pa \eta}
                    \, d\eta d\tau \\
    \displaystyle \spacer \qquad \qquad
    \leq
            \ep
            \int_{0}^{L_{\tau}} \int_0^{\infty}
                    \Delta \Big( \mathcal{F}_{\lambda} \Big( \dfrac{\pa \wi{\Psi}}{\pa \eta} \Big) \Big)
                    \, d\eta d\tau
    =
            \ep
            \int_{0}^{L_{\tau}}
            	\bigg[
                    \mathcal{F}^{\pri}_{\lambda} \Big( \dfrac{\pa \wi{\Psi}}{\pa \eta} \Big)
                    \dfrac{\pa^2 \wi{\Psi}}{\pa \eta^2}
                 \bigg]_{\eta = 0} \, d\tau.
\end{array}
\end{equation}

Using periodicity in the $\tau$-direction and
the fact that $\mathcal{U}$ is a function in $\eta$ only,
we observe that
the second term on the right-hand side of (\ref{e:L^1_PPF_TEMP1}) is identically zero:
\begin{equation}\label{e:L^1_PPF_TEMP4}
\begin{array}{l}
\displaystyle \spacer
        \int_{0}^{L_{\tau}} \int_0^{\infty}
                    \mathcal{F}^{\pri}_{\lambda} \Big( \dfrac{\pa \wi{\Psi}}{\pa \eta} \Big)
                    \mathcal{U}
                    \dfrac{\pa }{\pa \tau}\Big( \dfrac{\pa \wi{\Psi}}{\pa \eta} \Big)
                    \, d\eta d\tau\\
        \qquad \displaystyle \spacer
        =
            \int_0^{\infty} \int_{0}^{L_{\tau}}
                        \dfrac{\pa}{\pa \tau} \Big( \mathcal{F}_{\lambda} \Big( \dfrac{\pa \wi{\Psi}}{\pa \eta} \Big) \Big)
                        \, \mathcal{U}
                        \, d\tau d\eta
        =
            \int_0^{\infty}
                    \Big[
                        \mathcal{F}_{\lambda} \Big( \dfrac{\pa \wi{\Psi}}{\pa \eta} \Big) \, \mathcal{U}
                    \Big]_{\tau = 0}^{\tau = L_{\tau}}
                    \, d\eta
        = 0.
\end{array}
\end{equation}

We can finally conclude that:
\begin{equation}\label{e:L^1_PPF_TEMP5}
\begin{array}{l}
\displaystyle \spacer
        \dfrac{d}{dt} \int_{0}^{L_{\tau}} \int_0^{\infty}
        \mathcal{F}_{\lambda}\Big( \dfrac{\pa \wi{\Psi}}{\pa \eta} \Big)
        \, d\eta d\tau\\
\displaystyle \spacer
        \leq
            \ep
            \int_{0}^{L_{\tau}}
            	\bigg[
                    \mathcal{F}^{\pri}_{\lambda} \Big( \dfrac{\pa \wi{\Psi}}{\pa \eta} \Big)
                    \dfrac{\pa^2 \wi{\Psi}}{\pa \eta^2}
                 \bigg]_{\eta = 0} \, d\tau
            +
            \int_{0}^{L_{\tau}} \int_0^{\infty}
                    \mathcal{F}^{\pri}_{\lambda} \Big( \dfrac{\pa \wi{\Psi}}{\pa \eta} \Big)
                    \big(  \text{right-hand side of (\ref{e:difference_heat_like_eq_H1})} \big)_1
                    \, d\eta d\tau.
\end{array}
\end{equation}

The equation for $\wi{\Psi}$ at the boundary $\eta=0$ yields  a uniform bound in
$\eps$ on
$\ep \pa^2 \wi{\Psi}_2 / \pa \eta^2\lfloor_{\eta=0}$  in $L^{\infty}(0, T;
L^1(0, L_{\tau}))$.
It follows that the right-hand side of (\ref{e:difference_heat_like_eq_H1}) is
in $L^{\infty}(0, T; L^1((0, L_{\tau})\times(0, \infty)))$ uniformuly in $\ep$.
Then, since $|\mathcal{F}^{\pri}_{\lambda} (\cdot)| \leq 1$  in $\mathbb{R}$ as well,
we conclude that
\begin{equation*}\label{e:L^1_PPF_TEMP5-1}
    \dfrac{d}{dt} \int_{0}^{L_{\tau}} \int_0^{\infty}
        \mathcal{F}_{\lambda}\Big( \dfrac{\pa \wi{\Psi}}{\pa \eta} \Big)
        \, d\eta d\tau
    \leq
            \kap_T,
    \quad
    \text{independent of } \lambda.
\end{equation*}
Since the integral of $\mathcal{F}_{\lambda}$ over $(0, L_{\tau})\times(0, \infty)$ is positive at each time and $\pa \wi{\Psi} / \pa \eta = 0$ at $t = 0$, we see that
\begin{equation}\label{e:L^1_PPF_TEMP5-1-1}
    \lim_{\lambda \rightarrow 0}
    \int_{0}^{L_{\tau}} \int_0^{\infty}
        \mathcal{F}_{\lambda}\Big( \dfrac{\pa \wi{\Psi}}{\pa \eta} \Big)
        \, d\eta d\tau
    \leq
            \kap_T,
    \quad
    \text{uniformly in } 0 < t < T.
\end{equation}

Using (\ref{e:heat_tan_Lp}) and (\ref{e:Lp_heat_like_sol})$_2$, thanks to the Lebesgue dominated convergence theorem,
we deduce from (\ref{e:L^1_PPF_TEMP5-1-1}) that
\begin{equation*}\label{e:L^1_PPF_TEMP6}
        \Big\| \dfrac{\pa \wi{\Psi}}{\pa \eta} \Big\|_{L^1((0, L_{\tau})\times(0, \infty))}
                =
                    \int_{0}^{L_{\tau}} \int_0^{\infty}
                                \lim_{\lambda \rightarrow 0} \,
                                \mathcal{F}_{\lambda}\Big( \dfrac{\pa \wi{\Psi}}{\pa \eta} \Big)
                =
                    \lim_{\lambda \rightarrow 0}
                    \int_{0}^{L_{\tau}} \int_0^{\infty}
                                \mathcal{F}_{\lambda}\Big( \dfrac{\pa \wi{\Psi}}{\pa \eta} \Big)
                \leq
                    \kap_T,
\end{equation*}
independent of $\ep$, $t$, and $\lambda$, uniformly in time $0 < t < T$.
Hence (\ref{e:Lp_heat_like_sol})$_3$ {with $k=0$}
follows from (\ref{e:heat_tan_Lp}) and the inequality above.
Equation (\ref{e:Lp_heat_like_sol})$_3$ for $k>0$ can be proved similarly as
well.
\end{proof}

%
%
\section{A few auxiliary results}\label{A:Lemmas}

In this Appendix, we collect a few auxiliary results, which are needed for the
analysis of previous sections.

\cref{L:Trace} below contains a well-known trace inequality, mostly
used in the special case where $p = q = q' = 2$. A complete proof of this fact
and \cref{C:TraceCor} can be found in
\cite{Kel14Observations}.

\begin{lemma}[Trace lemma]\label{L:Trace}
	Let $p \in (1, \iny)$, $q \in [1, \iny]$, and let $q'$ be \Holder
conjugate to $q$.
Then, there exists a constant $C = C(\Omega)$
    such that for all $f \in W^{1, p}(\Omega)$,
    \begin{align*}
        \norm{f}_{L^p(\Gamma)}
            \le C \norm{f}_{L^{(p - 1) q}(\Omega)}
            		^{1 - \frac{1}{p}}
                \norm{f}_{W^{1, q'}(\Omega)}
                	^{\frac{1}{p}}.
    \end{align*}
    If, in addition, $f \in W^{1, p}(\Omega)$ has mean zero or $f \in W^{1,
p}_0(\Omega)$,
    \begin{align*}
        \norm{f}_{L^p(\Gamma)}
            \le C \norm{f}_{L^{(p - 1) q}(\Omega)}
            		^{1 - \frac{1}{p}}
                \norm{\grad f}_{L^{q'}(\Omega)}
                	^{\frac{1}{p}}.
    \end{align*}
\end{lemma}

\begin{cor}\label{C:TraceCor}
    For any $\vv \in H$,
    \begin{align*}
        \norm{\vv}_{L^2(\Gamma)}
            \le C \norm{\vv}_{L^2(\Omega)}
            		^{\frac{1}{2}}
                \norm{\grad \vv}_{L^2(\Omega)}
                	^{\frac{1}{2}}
    \end{align*}
    and for any $\vv \in V \cap H^2(\Omega)$,
    \begin{align*}
        \norm{\curl \vv}_{L^2(\Gamma)}
            \le C \norm{\curl \vv}_{L^2(\Omega)}
            		^{\frac{1}{2}}
                \norm{\grad \curl \vv}_{L^2(\Omega)}
                	^{\frac{1}{2}}.
    \end{align*}
\end{cor}

We recall that \ $H = \set{\vv \in L^2(\Omega) |
            \, \dv \vv = 0, \, \vv \cdot \n = 0 \text{ on } \Gamma}$, and \
    $V = \set{\vv \in H^1_0 (\Omega) | \, \dv \vv = 0}$.

Lastly, we state for the reader's convenience Young's inequality for multiple
factors:

\begin{lemma}\label{L:Youngs}
    Let $a, b, c \ge 0$, $p, q, p_1, q_1 \ge 1$ with $p^{-1} + q^{-1} = p_1^{-1}
    + p_2^{-1} = 1$. Then
    \begin{align*}
        a b
            &\le \frac{a^p}{p} + \frac{b^q}{q}, \\
        a b c
            &\le \frac{a^{p p_1}}{p p_1}
                + \frac{b^{p p_2}}{p p_2}
                + \frac{c^q}{q}.
    \end{align*}
\end{lemma}

%
%
\section{Vorticity accumulation on the boundary}

We close the Appendices with a short discussion of a result that is used
to derive a quantitative estimate for the vorticity production at the boundary
that persists in the vanishing viscosity limit.

It is shown in \cite{Kel08} that the classical vanishing viscosity limit, that
is, convergence of the NSE solution to an EE solution in the energy norm,
holds if and only if vorticity accumulates on the boundary in the manner
described in \cref{T:Kel}. (The specific 3D form of this condition is derived in
\cite{Kel14Observations}.) In \cref{T:Kel}, $\mu$ is the measure supported on
$\Gamma$ for which $\mu\vert_\Gamma$ corresponds to the normalized Lebesgue
measure on $\Gamma$ (arc length in 2D, surface area in 3D). Then $\mu$ is also a
member of $H^1(\Omega)^*$, the dual space to $H^1(\Omega)$.

\begin{theorem}\label{T:Kel}
 Let $\Omega$ be a  bounded domain in $\R^2$ or $\R^3$ of class $C^2$. Then,
    \begin{align*}
        2D: \quad
            \uu^\eps &\to \uu^0 \text{ in } L^\iny(0, T; H) \\
            &\iff
            \curl \uu^\eps \to \curl \uu^0 - (\uu^0 \cdot \BoldTau) \mu
                \text{ in } L^\iny(0, T; (H^1(\Omega)^2)^*), \\
        3D: \quad
        \uu^\eps &\to \uu^0 \text{ in } L^\iny(0, T; H) \\
            &\iff
            \curl \uu^\eps \to \curl \uu^0 + (\uu^0 \times \n) \mu
                \text{ in } L^\iny(0, T; (H^1(\Omega)^3)^*).
    \end{align*}
\end{theorem}

The following simple corollary of \cref{T:Kel} is derived in \cite{Kel14Observations}.

\begin{cor}\label{C:ConvKelMeasure}
    Under the hypotheses of \cref{T:Kel}, if $\uu^\eps \to \uu^0 \text{ in }
L^\iny(0, T; H)$ and
    $\norm{\curl \uu^\eps - \curl \uu^0}_{L^\iny(0, T; L^1(\Omega))}
    \le \kappa_T$ then
    \begin{align*}
        2D: \quad
            &\curl \uu^\eps \to \curl \uu^0 - (\uu^0 \cdot \BoldTau) \mu
                \text{ in } L^\iny(0, T; \Cal{M}(\ol{\Omega})), \\
        3D: \quad
            &\curl \uu^\eps \to \curl \uu^0 + (\uu^0 \times \n) \mu
                \text{ in } L^\iny(0, T; \Cal{M}(\ol{\Omega})),
    \end{align*}
    where $\Cal{M}(\ol{\Omega})$ is the space of Radon measures on
$\ol{\Omega}$,
\end{cor}

Above, $\blds{\tau}$ is the unit tanget vector to the boundary,
defined as $J \, \blds{n}$, where $\blds{n}$ is the unit outer normal, and $J$
is rotation counterclockwise by $\pi/2$.

The regularity of the boundary $\pa\Omega$ in the results above is sufficient
for the applications in this manuscript, but it is not expected to be optimal.
In fact, results of De Giorgi on weak convergence of gradients suggest that
similar statements hold for much rougher domains, namely sets of finite
perimenter. We do not investigate this point further in this work.

%
%
\section*{Acknowledgments}

The authors acknowledge the support in this work of various grants and institutions, as follows. G.-M. Gie: the Research - RI Grant,
{Office of the Executive Vice President for Research and Innovation}, University of Louisville; G.-M. Gie and J. Kelliher: US National Science Foundation (NSF) grants DMS-1009545 and DMS-1212141; M. Lopes Filho: CNPq grants \#200434/2011-0 and \#306886/2014-6; A. Mazzucato: NSF grants DMS-1009713, DMS-1009714, and DMS-1312727; H. Nussenzveig Lopes: CAPES grant BEX 6649/10-6, CNPq grant \#307918/2014-9, and  FAPERJ grant \# E-26/202.950/2015. J. Kelliher, M. Lopes Filho, and H. Nussenzveig Lopes: the National Institute for Pure and Applied Mathematics (IMPA) in Rio de Janeiro in residence Spring 2014; M. Lopes Filho and H. Nussenzveig Lopes: the Dept. of Mathematics at the University of California, Riverside in residence Fall 2011; A. Mazzucato: the Institute for Pure and Applied Mathematics (IPAM) at UCLA in residence Fall 2014 (IPAM receives major support from the NSF); the second through fifth authors: the Institute for Computational and Experimental Research in Mathematics (ICERM) in Providence, RI, in residence Spring 2017 (ICERM receives major support  from NSF under Grant No. DMS-1439786). The authors thank Gregory Eyink for pointing out a useful reference.

\bibliographystyle{plain}
\end{document}